\documentclass[12pt]{amsart}
\usepackage{bbm}
\usepackage[english]{babel}
\usepackage{color}
\usepackage{anysize}
\usepackage{geometry}
\marginsize{2.0cm}{2.0cm}{2.0cm}{2.0cm}
\usepackage{hyperref}
\usepackage{url}
\usepackage[normalem]{ulem}
\usepackage{hyperref}
\usepackage{mathtools}
\mathtoolsset{showonlyrefs}
\theoremstyle{plain}
\numberwithin{equation}{section}
\newtheorem{theorem}{Theorem}[section]
\newtheorem{corollary}[theorem]{Corollary}
\newtheorem{proposition}[theorem]{Proposition}
\newtheorem{lemma}[theorem]{Lemma}
\newtheorem{remark}[theorem]{Remark}

\newcommand{\cW}{{\mathcal{W}}}
\newcommand{\E}{{\mathbb E}}
\newcommand{\R}{{\mathbb R}}
\newcommand{\C}{{\mathbb C\hspace{0.05 ex}}}
\newcommand{\N}{{\mathbb N}}
\newcommand{\fs}{{\mathfrak{s}}}
\newcommand{\fd}{{\mathfrak{d}}}
\newcommand{\fv}{{\mathfrak{v}}}
\newcommand{\fu}{{\mathfrak{u}}}

\newcommand{\cf}[1]{{\mathbbm 1}_{\{#1\}}} 
\newcommand{\rme}{{\rm e}}
\newcommand{\rmd}{{\rm d}}
\newcommand{\vep}{\varepsilon}
\newcommand{\defset}[2]{ \left\{ #1\left|\,
#2\makebox[0cm]{$\displaystyle\phantom{#1}$}\right.\!\right\} }
\newcommand{\set}[1]{\{#1\}}
\newcommand{\mean}[1]{\left\langle #1\right\rangle}
\title[Control of the telegraph process]{Quantitative control of Wasserstein distance between Brownian motion and the Goldstein--Kac telegraph process}
\author{Gerardo Barrera}
\address{University of Helsinki, Department of Mathematics and Statistics. P.O. Box 68, Pietari Kalmin katu 5, FI-00014. Helsinki, Finland.}
\email{gerardo.barreravargas@helsinki.fi}
\author{Jani Lukkarinen}
\address{University of Helsinki, Department of Mathematics and Statistics. P.O. Box 68, Pietari Kalmin katu 5, FI-00014. Helsinki, Finland.}
\email{jani.lukkarinen@helsinki.fi}

\begin{document}
\begin{abstract}
In this manuscript, we provide a non-asymptotic process level control between the telegraph process and the Brownian motion with suitable diffusivity constant via a Wasserstein distance with quadratic average cost.
In addition, we derive non-asymptotic estimates for the corresponding time average $p$-th moments.
The proof relies on coupling techniques such as coin-flip coupling,  synchronous coupling and the Koml\'os--Major--Tusn\'ady coupling.
\end{abstract}

\maketitle

\tableofcontents

\section{\textbf{Introduction}}\label{sec:intro}

The so-called (Goldstein--Kac) telegraph process is perhaps the simplest example of a random evolution, see \cite{Goldstein1951} and \cite{Kac1974}. It describes the movement of a particle which starts at time zero
from the origin and moves with a finite constant speed $v_0$  on the line.
The initial direction of the motion, positive or negative, is chosen at random with the same probability.
The changes of direction are
driven by a homogeneous Poisson process of a positive constant intensity $\lambda$. In other words, 
when a jump occurs in the Poisson process, the particle instantaneously takes the opposite direction and keeps moving with the same speed (it just flips the sign of its velocity) until the next jump in the Poisson process happens, then it takes the opposite direction again, and so on.
This is in fact an example of a zigzag process considered in \cite{Bierkens2019}.
 Later on, we provide a precise mathematical definition in \eqref{def:path}.
In \cite{Goldstein1951} and \cite{Kac1974},
it is shown that  the family of probability densities of the particle position $(f(x,t):x\in \mathbb{R},t\geq 0)$ solves the 
hyperbolic partial differential equation
\begin{equation}\label{eq:hpde}
\frac{\partial^2 }{\partial t^2}f(x,t)+2\lambda \frac{\partial }{\partial t}f(x,t)=v_0^2 \frac{\partial^2 }{\partial x^2}f(x,t)\,.
\end{equation}
Since \eqref{eq:hpde} appears in electrical wave propagation models,
it is known as the telegraph equation or the hyperbolic heat equation.
An explicit solution of \eqref{eq:hpde} can be carried out in 
terms of special functions, see for instance Chapter~2 in \cite{Kolesnik2013}.

The telegraph process has been proposed as an alternative to diffusion models and, as such, extensively studied by the probability and physics communities.
Its generalisations are ubiquitous in applications, including transport phenomena in physical and biological systems, and it has produced a vast mathematics literature; standard references include \cite{Bogachev2011,DeGregorio2021,
DiCrescenzo2013,DiCrescenzo2018, Foong1994, Ghosh2014, Goldstein1951,
Iacus2009,Janssen1990,Kac1974,Kolesnik2018,Kolesnik1998,Kolesnik2015,
Kolesnik2013,
KolesnikTurbin1998,Martinucci2020,
Orsingher1990,Stadje2004} and further references may be found from therein.
It is also used in the context of risk theory
and to model financial markets,
see  \cite{DiMasi1994}, \cite{Kolesnik2013} and \cite{Mazza2004}.
Statistics for the telegraph process are done in  \cite{Iacus2009}. More recently, in \cite{Cinque2021} the authors compute the distribution of the maximum of the asymmetric telegraph process.

Despite its simplicity and the already existing results 
about the agreement with Brownian motion using the above marginals, it is not easy to connect these two processes on the level of realisations, i.e., on the process level.
The present work is motivated by the need to understand a similar connection to diffusion in an extension of the telegraph process where
the stochastic process is augmented into an interacting particle system, ordered along a circle,
and each particle follows an independent telegraph process but there additionally is an interaction potential acting between the nearest neighbour particles.  The resulting system is called a {\em velocity flip model\/}: the connection between them and suitably chosen diffusion processes have been explored, for instance,  in \cite{Bernardin2011,Lukkarinen2014,Lukkarinen2016,Simon2013}
but process level understanding of the connection is still lacking, even in the simplest case of a harmonic interaction potential.

Our aim is to provide a process level control between the telegraph process and the Brownian motion with suitable diffusivity constant via a Wasserstein distance with quadratic average cost.
Often, one can  rely on generators to see if a Markov process converges to another process, see for instance \cite{Ethier1986} and \cite{Pinsky1991}. There are several ways to quantify the preceding convergence, for instance, by the Wasserstein distance (Kantorovich distance).
For the setting of two pure jump Markov $\mathbb{R}^d$-valued processes with bounded intensity of jumps, the time evolution of the Wasserstein distance between them can be 
written as an integro-differential equation in terms of their generators and Kantorovich potentials,
see for instance Theorem~3.1 in \cite{Alfonsi2018}.
In dimension one, the latter remains true for piecewise deterministic Markov processes, see
Theorem~6.4 in \cite{Alfonsi2018}.
However, we do not follow this approach since the relevant Kantorovich potentials are hard to compute and also the process level control, in which we are interested, is with respect to an average quadratic cost function given in \eqref{def:metric} below.
This is in fact a natural distance between process; we refer to \cite{BionNadal2019} for further details.
For diffusions, under regular assumptions on their coefficients,  by analytical approach the Wasserstein distance and total variation distance can be estimated, see Theorem~1.1 in \cite{Bogachev2016}.
By coupling approach, in \cite{Eberle2019} the authors obtain bounds of the total variation distance between two multidimensional It\^o processes with different drifts.
We point out that in general, it is difficult to obtain rates of convergence by the generator approach or for functional limit theorems. One possible approach is offered by the so-called Stein's method, see \cite{Briand2021}, \cite{Coutin2020}, although also this tool is difficult to implement in the present case.

In this manuscript we present non-asymptotic estimates between the telegraph process and the Brownian motion with suitable diffusivity constant. 
The main mathematical tools for the rigorous control of the estimates are coupling techniques.
We do not require explicit knowledge of the marginals of the corresponding processes.
Roughly speaking, the idea is to compare the telegraph process with a decoupled process (perturbed random walk) via the coin-flip coupling (see \cite{Koskinen2020} or Appendix~A in \cite{Saksman2020}). Then we compare the decoupled process and the Brownian motion with suitable diffusivity constant.
Since the decoupled process has zero mean at even jumps, the latter is done in two steps
 via the synchronous coupling and the celebrated Koml\'os--Major--Tusn\'ady coupling (see \cite{Komlos1975} and \cite{Komlos1976}).
The argument is detailed in Section~\ref{Sec:out}.
As a consequence of our main result, Theorem~\ref{th:main}, we obtain non-asymptotic estimates for the time average of $p$-th moments.

The paper is organised as follows:
Section~\ref{sec:setting} lays out the setting, the main result  Theorem~\ref{th:main} and its consequences formulated as Corollary~\ref{cor:tele} and
Corollary~\ref{cor:absmoments}.
Section~\ref{Sec:out} sketches the steps leading to the proof of the main result Theorem~\ref{th:main}, which boils down to the proofs of Proposition~\ref{prop:coinflipcoupling}, Proposition~\ref{prop:naturalcoupling} and Lemma~\ref{lem:couplingKMT}.
Proposition~\ref{prop:coinflipcoupling} is proved in Section~\ref{sec:proof}, Proposition~\ref{prop:naturalcoupling} is shown in Section~\ref{sec:naturalcoupling} and Lemma~\ref{lem:couplingKMT} is proved Section~\ref{sec:KMTcoupling}.
Finally, 
there is an Appendix which collects main technical results and estimates used in the main text.
It is divided in four subsections: Appendix~\ref{ap:tools} yields quantitative estimates of moments, Appendix~\ref{ap:coupling} establishes coupling estimates between the free velocity flip model and a suitable Brownian motion, Appendix~\ref{ap:delta} gives integration formulas for the Dirac $\delta$-constrain probability measure, and Appendix~\ref{ap:basic} provides basic auxiliary results.

\section{\textbf{The setting and the main results}}\label{sec:setting}
This section is divided in four subsections.
We start by introducing the model and the notation in Subsection~\ref{subsec:constr}.
Then, we discuss related works in Subsection~\ref{subsect:related}. Next,  we define the metric in which we compare our processes in Subsection~\ref{subsect:metric}.
Finally,  the main result and its consequences are given in Subsection~\ref{subsect:main}.

\subsection{\textbf{Construction of the free velocity flip model}}\label{subsec:constr}

Let $\N=\{1,2,\ldots,\}$,  $\N_0=\N\cup\{0\}$ and $\R_*:=(0,\infty)$.
For any $\fd\in \R_*^\N:=\{z=(z_1,z_2,\ldots, ):z_j\in \R_*~\textrm{ for all }~j\in \N\}$, $\fd=(\delta_1,\delta_2,\ldots,)$, we set
$t_0(\fd):=0$ and 
\begin{equation}\label{eq:suma}
 t_n(\fd) := \sum_{k=1}^n \delta_k\, \quad\mathrm{for\; any }\quad n\in \mathbb{N}\,.
\end{equation}
We define the set of admissible sequences 
$\Delta_\infty$ by
\begin{equation}\label{def:domain}
\Delta_\infty:=\left\{\fd\in \R_*^\N:\sum_{k=1}^{\infty}\delta_k=\infty\right\}\,.
\end{equation}
In other words, $\fd\in \Delta_\infty$ if and only if 
$t_n(\fd)\to \infty$ as $n\to \infty$.
For each $\fd\in \Delta_\infty$
we set 
\begin{equation}\label{eq:poproc}
\begin{split}
N:\Delta_\infty& \longrightarrow \mathcal{D}([0,\infty),\mathbb{R})\\
&\hspace{-0.5cm}
\fd \hspace{0.3cm} \longmapsto \hspace{0.3cm} N(\cdot;\fd):t\hspace{0.1cm}\mapsto \hspace{0.1cm}   N(t;\fd) = \sup\defset{n\in \N_0}{t_n(\fd)\leq t}\,,\\
 \end{split}
\end{equation}
where $\mathcal{D}([0,\infty),\mathbb{R})$ denote the set of C\`adl\`ag paths defined on $[0,\infty)$ and taking values on $\mathbb{R}$. We observe that
the function $t\mapsto N(t;\fd)$ is 
non-decreasing.

The sequence $\fd$ is taken to encode the waiting times between jumps
for a certain piecewise linear path $(X(t;\fd):0\leq t<\infty)$.  
We point out that then $N(t;\fd)$ corresponds to the number of jumps which have occurred by the time $t$.
The process $X$, which can be identified with a ``free velocity flip process path in one dimension'', is defined by the following properties:
\begin{enumerate}
\item Instantaneous position $X(t;\fd)\in \R$, 
\item starting velocity is $v_0\in \R$, and 
\item at each ``jump'' the velocity flips its sign, $V(t;\fd)\mapsto -V(t;\fd)$. 
More precisely,
\[
V(t;\fd):=v_0(-1)^{N(t;\fd)} \quad \textrm{ for all } \quad t\geq 0\,.
\] 
\end{enumerate}
For simplicity, we assume that the initial position is zero, that is, $X(0;\fd)=0$.
Then
we explicitly define the free velocity flip path $X$ as follows: 
\begin{equation}\label{def:path01}
\begin{split}
X:\Delta_\infty& \longrightarrow \mathcal{C}([0,\infty),\mathbb{R})\\
&\hspace{-0.5cm}
\fd \hspace{0.3cm} \longmapsto \hspace{0.3cm} X(\cdot;\fd):t\hspace{0.1cm}\mapsto \hspace{0.1cm}  X(t;\fd):=\int_{0}^{t}
\rmd s  \,V(s;\fd)\,,\\
 \end{split}
\end{equation}
where $\mathcal{C}([0,\infty),\mathbb{R})$ denotes the set of continuous paths defined on $[0,\infty)$ and taking values on $\mathbb{R}$.
It is not hard to see that for any $\fd\in \Delta_\infty$ and $t\geq 0$,
\begin{equation}\label{def:path}
 X(t;\fd)
 = v_0\sum_{k=1}^{M} (-1)^{k-1} \delta_k + v_0(-1)^{M} (t-t_{M}(\fd))\,, \quad M=N(t;\fd)\,,
\end{equation}
with the convention that $\sum_{k=a}^{b} (\cdots)=0$ whenever $a>b$, $a,b\in \mathbb{Z}$.
In particular, for $t\in[0,\delta_1)$, we have $X(t;\fd)=v_0 t$.  Also, by definition of $N$, we have 
\[
\left|v_0(-1)^M (t-t_M(\fd))\right|\leq |v_0|(t_{M+1}(\fd)-t_M(\fd))=|v_0|\delta_{M+1}\,.
\]

\noindent
{\textbf{Notation}.}
When the sequence $\fd=(\delta_k:k\in \mathbb{N})$ 
is distributed according to the infinite product measure of exponential distributions with parameter $\lambda>0$,
from here on  we always write 
$\fs=(s_1,s_2\ldots,)$ instead of $\fd$. 

We note that $\fs\in \Delta_\infty$ almost surely and also
 $N(t;\fs)<\infty$  for any $t\ge 0$ almost surely. 
In fact,
$N(t;\fs)$ is the unique (random) value $n\in \N_0$ for which $t_n(\fs)\leq t<t_{n+1}(\fs)$ and
the process 
\begin{equation}\label{eq:poissonprocess}
(N(t;\fs):t\geq 0)\quad
\textrm{ is called a Poisson process of intensity } \lambda\,.
\end{equation}  For further details about Poisson processes we recommend the monographies  \cite{Brmaud2020} and \cite{Mikosch2009}. 
We denote by $(\Omega,\mathcal{F},\mathbb{P})$ the complete probability space where $X^{\fs}:=(X(t;\fs):t\geq 0)$ is defined and by $\mathbb{E}$ the expectation with respect to $\mathbb{P}$.
The vector process
$((X(t;\fs),V(t;\fs)):t\geq 0)$ is Markovian on the state space
$\mathbb{R}\times \{-v_0,v_0\}$, see for instance Section~12.1, p. 469 in \cite{Ethier1986}.
However, the marginal process $X^{\fs}$ is not itself Markovian for any non-zero initial velocity $v_0$.

Note that for $v_0=0$ we have $X(t;\fs)=0$ for all $t\geq 0$. In the rest of the manuscript we always assume that $v_0\in \R\setminus\{0\}$.
We point out that the free velocity flip model corresponds to the so-called Goldstein--Kac telegraph process when the initial velocity $v_0$ is chosen at random from the uniform distribution on $\{-1,1\}$ (Rademacher distribution) and independent  from the Poisson process $(N(t;\fs):t\geq 0)$. 
In this scenario, a straightforward computation yields
\[
\mathbb{E}\big[e^{-zV(t)}\big]=\cosh(z)\quad \textrm{ for any } z \in \mathbb{R}\quad \textrm{ and }\quad t\geq 0\,,
\]
where $\cosh$ denotes the hyperbolic cosine function, and hence the velocity process $(V(t):t\geq 0)$ is stationary.

In what follows we do not assume that $v_0$ is random. Nevertheless, our main result and its consequences hold true as soon as we assume that $v_0$ is a random variable with finite inverse second moment and independent of the Poisson process $(N(t;\fs):t\geq 0)$.
Hence, it covers the Goldstein--Kac telegraph process. 
The variance of $X(t;\fs)$, $\mathrm{Var}(X(t;\fs))$, is given by
\begin{equation}\label{eq:forvar}
\mathrm{Var}\left[X(t;\fs)\right]
=\frac{v^2_0}{\lambda^2}\left(\lambda t+e^{-2\lambda t}-\frac{e^{-4\lambda t}}{4}-\frac{3}{4}\right)\,
\quad \textrm{ for all }\quad t\geq 0\,.
\end{equation}
For details, see Item~(2) of Lemma~\ref{lem:momentskac} in Appendix~\ref{ap:coupling}. 

We also choose a target spatial scale, $L>0$, and time scale, $T>0$, in which we wish to control the position of the process $X$. The evolution of the process $X$ will then
be most conveniently described in terms of the following scaling parameters:
\begin{equation}\label{e:scalings}
T_*:=\lambda T\quad \textrm{ and } \quad {L}_*=|v_0|^{-1}\lambda L\,.
\end{equation}
By formula \eqref{eq:forvar} we have
\[
\mathrm{Var}\left[{L}^{-1}{X(T;\fs)}\right]=
\frac{T_*}{L^2_*}
+\frac{1}{L^2_*}\left(e^{-2T_*}-\frac{e^{-4T_*}}{4}-\frac{3}{4}\right)\,.
\]
The latter suggests that for $T_*\to \infty$, $L_*\to \infty$, $T_*/L^2_*\to \zeta$ for some $\zeta>0$,
the scaling process $(L^{-1}X(t;\fs):0\leq t\leq T)$ should behave as the Brownian motion with suitable diffusivity coefficient.

\subsection{\textbf{Related works}}\label{subsect:related}

It is well-known that, under a suitable scaling, the telegraph process satisfies a functional central limit theorem. To be more precise, if $\lambda \to \infty$ and $v_0\to \infty$ such that $v^2_0/\lambda \to 1$, then
the telegraph process $(X(t;\fs):0\leq t<\infty)$ converges weakly in $\mathcal{C}([0,\infty),\mathbb{R})$ to a standard Brownian motion $W=(W(t):t\geq 0)$, see for instance Theorem~1.1 in \cite{Bogachev2011} and the references therein.
Using the fact that the  evolution equation associated to $(X(t;\fs):0\leq t<\infty)$ is given by the telegraph partial differential equation \eqref{eq:hpde}, asymptotics for the transition probability can be carried out in terms of Bessel functions, see \cite{Orsingher1990}.
However, up to our knowledge, 
the unique reference discussing about 
 rate of convergence of the fixed rescaled time marginal to the corresponding Gaussian distribution is \cite{Janssen1990}, in which an expansion of length two is carried out for the
marginal $X(t;\fs)$ based on the marginal $W(t)/\sqrt{\lambda}$.  The error term is of order $\mathcal{O}(t)$ as $t\to \infty$.

We emphasize that the main 
Theorem~\ref{th:main} of this paper states a non-asymptotic estimate in the average quadratic Wasserstein metric between the processes 
$(L^{-1}X(t;\fs):t\geq 0)$ and  a Brownian motion (not necessarily standard). This, in particular, implies Theorem~1.1 in \cite{Bogachev2011} when
$\lambda \to \infty$ and $v_0\to \infty$ in such a way that $v^2_0/\lambda \to 1$.
The latter is referred as a singular perturbation scale in  \cite{Janssen1990}.

\subsection{\textbf{Comparing two stochastic processes}}\label{subsect:metric}

\subsubsection{\textbf{Coupling and Wasserstein metric}}

The Wasserstein metric is used to measure the distance between two probability
measures on a Radon space. Polish spaces are standard examples of Radon spaces.
Let $d\in \N$ and consider two
stochastic processes 
$(X(t):t\geq 0)$ and $(Y(t):t\geq 0)$
defined on the probability space $(\Omega,\mathcal{F}, \mathbb{P})$ and 
 taking values in $\R^d$.
 The Skorokhod space (Radon space)
consisting of the set of right-continuous with left limits functions taking values in $\mathbb{R}^d$ is denoted by $\mathcal{D}([0,\infty),\mathbb{R}^d)$ and 
assume that the trajectories of the processes $(X(t):t\geq 0)$ and $(Y(t):t\geq 0)$ belong to
$\mathcal{D}([0,\infty),\mathbb{R}^d)$.
Let $T>0$ be a \textit{target time scale}, and denote by $\mu_1(\rmd X)$ and $\mu_2(\rmd Y)$ the probability measures 
of the processes
$(X(t):t\geq 0)$ and $(Y(t):t\geq 0)$, respectively, restricted to times $t$ with $0\leq t\leq T$.
Then we define the \textit{average quadratic cost function}
\begin{equation}\label{eq:costfunction}
 c_2(X,Y) :=  \frac{1}{T} \int_{0}^{T}\rmd s \,\|X(s)-Y(s)\|^2\,,
\end{equation}
where $\|\cdot\|$ denotes the Euclidean norm on $\R^d$.
The corresponding Wasserstein distance of order $2$ with cost function $c_2$ between $\mu_1$ and $\mu_2$, $\cW_2(\mu_1,\mu_2)$, is defined by
\begin{equation}\label{def:metric}
\begin{split}
\cW_2(\mu_1,\mu_2):&= \left(\inf_{\gamma} \int \gamma(\rmd X,\rmd Y)
 c_2(X,Y)\right)^{{1}/{2}}\\
 & = \left(\inf_{\gamma} \int \gamma(\rmd X,\rmd Y) \frac{1}{T} \int_{0}^{T}\rmd s \|X(s)-Y(s)\|^2\right)^{{1}/{2}}\,,
\end{split}
\end{equation}
where the infimum is taken over all  couplings $\gamma$ between the probability measures
$\mu_1$ and $\mu_2$. 
In other words, for all continuous and bounded observables $h:\mathcal{D}([0,T],\R^d)\to \C$ it follows
\[
\int \gamma(\rmd X,\rmd Y) h(X)=\int \mu_1(\rmd X) h(X)
\quad \mathrm{and}\quad
\int \gamma(\rmd X,\rmd Y) h(Y)=\int \mu_2(\rmd Y) h(Y)\,.
\]
For shorthand, we write $\cW_2(X,Y)$ in place of $\cW_2(\mu_1,\mu_2)$.
We point out that 
\eqref{def:metric} defines a natural distance between stochastic processes, see for instance \cite{BionNadal2019}.
Indeed, 
$\cW_2$ defines a metric that metrizes the weak topology on the set of probability measures $\eta$ on $L^2([0,T ],\R^d)$ such that 
\[
\int_{L^2([0,T],\R^d)}\eta(\rmd X)\int_{0}^{T}\rmd s \|X(s)\|^2<\infty\,.
\] 
Unfortunately, the numerical computation of $\cW_2$ or any
other Wasserstein metric on an infinite dimensional space is very difficult.
For basic definitions, properties and notions related to couplings and Wasserstein metrics, we refer to \cite{Panaretos2020} and \cite{Villani09}.

\subsubsection{\textbf{Global estimate for moments}}

Let $w:\R^d\to [0,\infty)$
be a weight function 
 which yields an $L^2$-integrable observable with respect to $\mu_1$ and $\mu_2$. To be more precise, we require that 
\begin{equation}\label{eq:Cdef}
\begin{split}
&C_1^2:=\int \mu_1(\rmd X) \frac{1}{T}\int_{0}^{T}\rmd s\, (w(X(s)))^2<\infty\quad \mathrm{and}\\
& 
C_2^2:=\int \mu_2(\rmd Y) \frac{1}{T} \int_{0}^{T}\rmd s\, (w(Y(s)))^2 <\infty\,.
\end{split}
\end{equation}
Consider then any measurable function $f:\R^d\to \C$ which is ``$w$-Lipschitz'', in the following precise sense: there is a Lipschitz constant $K\ge 0$ such that 
\begin{equation}\label{def:lip}
 |f(x)-f(y)|\leq K \frac{w(x)+w(y)}{2}\|x-y\|\quad \mathrm{ for\; all}\quad x,y\in \R^d\,,
\end{equation}
where $|\cdot|$ denotes the modulus of a complex number.
The preceding condition \eqref{def:lip} is chosen to obtain upper bounds for the error as we see in \eqref{eq:momentCs}.
The preceding definition \eqref{def:lip} is motivated by the following observation: for any $p>0$ it follows that
\begin{equation}\label{eq:momentosw2}
||x|^p-|y|^p |\leq K\frac{w(x)+w(y)}{2}|x-y|\quad \textrm{ for all }\quad x,y\in \mathbb{R}\,,
\end{equation}
where $K=2p$ and $w:\mathbb{R}\to [0,\infty)$ is given by $w(z)=|z|^{p-1}$ for all $z\in \mathbb{R}$. For details see 
Lemma~\ref{lem:momentabs} in Appendix~\ref{ap:coupling}.
Inequality \eqref{eq:momentosw2} implies that for the $p$-th moment observables, the estimation of the corresponding $C_1$ and $C_2$ needs implicitly good estimates of the $2(p-1)$-th moments.

\noindent
For short, we use the following standard notation $\mean{h}_{\mu}$ to denote $\int \mu(\rmd X) h(X)$ for any $h:\mathcal{D}([0,T],\mathbb{C})\to \mathbb{C}$.
Then, in particular, for any coupling $\gamma$ between $\mu_1$ and $\mu_2$ we have 
\begin{equation}
\begin{split}
\left|\mean{\frac{1}{T} \int_{0}^{T}\rmd s f(X(s))}_{{\mu_1}}-\mean{\frac{1}{T} \int_{0}^{T}\rmd s f(Y(s))}_{{\mu_2}}\right| & \\
&\hspace{-6cm} = 
\left| \int \gamma(\rmd X,\rmd Y)
\frac{1}{T} \int_{0}^{T}\rmd s \left(f(X(s))-f(Y(s))\right)\right|
\\ 
&\hspace{-6cm} \le
\int \gamma(\rmd X,\rmd Y)
\frac{1}{T} \int_{0}^{T}\rmd s \frac{K}{2} (w(X(s))+w(Y(s)))\|X(s)-Y(s)\|
\\ &\hspace{-6cm}  \le
\frac{K}{2}
\left[\int  \gamma(\rmd X,\rmd Y)
\frac{1}{T} \int_{0}^{T}\rmd s (w(X(s))+w(Y(s)))^2\right]^{{1}/{2}}\\
&\hspace{-5cm}
\times
\left[ \int \gamma(\rmd X,\rmd Y) \frac{1}{T} \int_{0}^{T}\rmd s 
\|X(s)-Y(s)\|^2\right]^{{1}/{2}}
\\ &\hspace{-6cm}  \le
\frac{K}{\sqrt{2}} \left[C_1^2+C_2^2\right]^{{1}/{2}}
\left[ \int \gamma(\rmd X,\rmd Y)\, c_2(X,Y)\right]^{{1}/{2}}\,,
\end{split}
\end{equation}
where we have used H\"older's inequality twice.
Since this is true for any coupling $\gamma$,
we obtain an \textit{explicit estimate for time-averages}
using the Wasserstein distance  
$\cW_2$ defined in \eqref{def:metric}. In other words, for any $f$ satisfying $\eqref{def:lip}$ we have the global estimate
\begin{equation}\label{eq:momentCs}
\left|\mean{\frac{1}{T} \int_{0}^{T}\rmd s f(X(s))}_{{\mu_1}}-\mean{\frac{1}{T} \int_{0}^{T}\rmd s f(Y(s))}_{{\mu_2}}\right| 
\leq 
K{\max\{C_1,C_2\}}  \cW_2(X,Y)\,.
\end{equation}
Moments of random variables are one of the most interesting and useful observables both from theoretical and practical points of view.
Hence,
the right-hand side of \eqref{eq:momentCs} can be read as follows: good approximations in terms of the above $\cW_2$-distance imply similar approximations for a large class of physically relevant time-averages.

\subsection{\textbf{Main result and consequences}}\label{subsect:main}

Given any stochastic process $S=(S(t):t\geq 0)$ and a target time $T>0$, for convenience we define the projection process up to time $T$ as follows:
\[
\mathbb{S}_{[0,T]}:=(S(t):0\leq t\leq T)\,.
\]
Recall that $X^\fs=(X(t;\fs):t\geq 0)$ denotes the velocity flip model defined in \eqref{def:path}, where $\fs$ is sampled from an  infinite product of exponential distribution of parameter $\lambda$. For any $L>0$ and $T>0$ we consider the corresponding projection process 
$L^{-1}\mathbb{X}^{\fs}_{[0,T]}$ and we compare it with the projection process $\mathbb{B}_{[0,T]}$ of a Brownian motion $B=(B(t):t\geq 0)$ with diffusivity constant given by $\sigma^2=L^{-2}v^2_0\lambda^{-1}$. 
The independent coupling between $L^{-1}\mathbb{X}^{\fs}_{[0,T]}$ and $\mathbb{B}_{[0,T]}$ yields the following crude estimate:
for any $L>0$, $T>0$, $\lambda>0$ and $v_0\in \mathbb{R}\setminus\{0\}$ it follows that
\begin{equation}\label{eq:crudeboundKB}
\cW_2\left(L^{-1}\mathbb{X}^{\fs}_{[0,T]},\mathbb{B}_{[0,T]}\right)\leq
\left(\frac{1}{4T_* L^2_*}
\left(1-e^{-2 T_*}\right)
-\frac{1}{2L_*^2}+
\frac{T_*}{L^2_*}\right)^{{1}/{2}}\,,
\end{equation}
where the scaling parameters are given by $T_*=\lambda T$ and ${L}_*=|v_0|^{-1}\lambda L$.
For details for we refer to 
Lemma~\ref{lem:indcoup} in 
Appendix~\ref{ap:coupling}.
If we do not use the independent coupling and we just apply H\"older's inequality for the cross term 
$\mathbb{E}[L^{-1}X(s;\fs)B(s)]$,
we pay the price of a factor $\sqrt{2}$ in the right-hand side of the preceding estimate.
Nevertheless, for $T_*\to \infty$ and $L_*\to \infty$ in such way that $T_*/L^2_*\to \zeta>0$ for some $\zeta>0$, the right-hand side of \eqref{eq:crudeboundKB} tends to $\zeta^{1/2}$
and hence it is not informative.

The main result of this manuscript is the following.
\begin{theorem}[Brownian motion approximation of the free velocity flip model]\label{th:main}
\begin{equation}\label{eq:mainresult}
\cW_2\left(L^{-1}\mathbb{X}^{\fs}_{[0,T]},\mathbb{B}_{[0,T]}\right)\leq C\sqrt{{T_*}{L^{-2}_*}}T^{-1/4}_*\big(\sqrt{\ln(T_*+3)}+T^{-3/4}_*\big)+CL^{-1}_*\,,
\end{equation}
where the constants 
$T_*$ and $L_*$ are given by
\begin{equation}\label{e:scalings2020}
T_*:=\lambda T,\quad  \quad {L}_*=|v_0|^{-1}\lambda L\,,
\end{equation}
and the diffusivity constant of the Brownian motion $B:=(B(t):t\geq 0)$ is defined by 
\begin{equation}\label{eq:diffusivity}
\sigma^2:=L^{-2}\frac{v_0^2}{\lambda}\,.
\end{equation}
\end{theorem}

We continue to rely on the notations and assumptions in Theorem~\ref{th:main}.
We point out that the average cost and the time-one cost are linked by a deterministic time-change of the processes.
\begin{remark}[The average cost vs. the time-one cost]
\label{rem:averagecost}
The cost function $c_2$ defined in \eqref{eq:costfunction} depends on the target time $T$ and hence the corresponding Wasserstein distance $\mathcal{W}_2$ defined in \eqref{def:metric} also depends on $T$.
For any $t\geq 0$ we set
\[
\widetilde{X}(t;\fs):=L^{-1}X(Tt;\fs)\quad \textrm{ and }\quad
\widetilde{B}(t):=B(Tt)\,. 
\]
Then we have 
\[
\cW_2\left(L^{-1}\mathbb{X}^{\fs}_{[0,T]},\mathbb{B}_{[0,T]}\right)=
\widetilde{\cW_2}\left(\widetilde{\mathbb{X}}^{\fs}_{[0,1]},\widetilde{\mathbb{B}}_{[0,1]}\right)\,,
\]
where the cost function $\widetilde{c_2}$ of $\widetilde{\mathcal{W}_2}$ is given by
\[
\widetilde{c_2}(\widetilde{X},\widetilde{B})=\int_{0}^{1}
|\widetilde{X}(t;\fs)-\widetilde{B}(t)|^2 \rmd t\,.
\]
We point out that the cost function $\widetilde{c_2}$ does not depend on the target time $T$ and hence the distance $\widetilde{\mathcal{W}_2}$ also does not depend on $T$.
\end{remark}
Using the scale invariance property of the Brownian motion, we stress the dependence of \eqref{eq:diffusivity}
in
\eqref{eq:mainresult}.

\begin{remark}[The diffusivity constant $\sigma^2$]\label{rem:difconst}
We note that \eqref{eq:mainresult} can be rewritten as
\begin{equation*}
\cW_2\left(L^{-1}\mathbb{X}^{\fs}_{[0,T]},\sigma\mathbb{W}_{[0,T]}\right)\leq C\sqrt{{T_*}{L^{-2}_*}}T^{-1/4}_*\big(\sqrt{\ln(T_*+3)}+T^{-3/4}_*\big)+CL^{-1}_*\,,
\end{equation*}
where $(W(t):t\geq 0)$ is a standard Brownian motion.
\end{remark}

We point out that the diffusive scaling limit is deduced from \eqref{eq:mainresult}.
\begin{remark}[Diffusive scaling limit]\label{rem:dsl}
For $T_*\to \infty$ and $L_*\to \infty$ in such way that $T_*/L^2_*\to \zeta>0$ for some $\zeta>0$, the right-hand side of \eqref{eq:mainresult} is $\mathcal{O}(T^{-1/4}_*\sqrt{\ln(T_*+3)})$ and therefore tends to zero. For instance,  choosing the diffusive scaling $L=\sqrt{\frac{Tv^2_0}{\lambda \zeta}}$ and taking $T\to \infty$ produces such a limit.
\end{remark}
Now, we describe the meaning of the parameters \eqref{e:scalings2020}.
\begin{remark}[The meaning of the parameters]
Note that $T_*$ is the mean of the number of jumps at time $T$ and that $L_*=\frac{1}{|v_0|\lambda^{-1}}L$ where $|v_0|\lambda^{-1}$ is the mean free path of the unscaled process $X^s$. In other words, $\mean{|X(t_1(\fs);\fs)|}_{\fs}=|v_0|\lambda^{-1}$, where $t_{1}(\fs)=s_1$.
\end{remark}
\begin{remark}[A word about the constant $C$]
The pure constant $C$ of Theorem~\ref{th:main} can be estimated explicitly from the proof. However, the estimates given in the proof are not optimal and the  most likely lead to overestimation.
\end{remark}
Next corollary yields a Brownian motion approximation of the telegraph process.
\begin{corollary}[Brownian motion approximation of the telegraph process]\label{cor:tele}
Suppose that all assumptions made in 
Theorem~\ref{th:main} hold.
Let $L>0$, $T>0$ and $\lambda>0$.
In addition, assume that the initial velocity $v_0$ is chosen with the uniform distribution on the set $\{-c,c\}$ for some $c>0$ and independent of the Poisson process $(N(t;\fs):t\geq 0)$ given in \eqref{eq:poissonprocess}. 
Then there exists a positive constant $C$ such that for any $L>0$, $T>0$ and $\lambda>0$
it follows that
\begin{equation}\label{eq:telemain}
\cW_2\left(L^{-1}\mathbb{X}^{\fs}_{[0,T]},\mathbb{B}_{[0,T]}\right)\leq  C\sqrt{{T_*}{L^{-2}_*}}T^{-1/4}_*\big(\sqrt{\ln(T_*+3)}+T^{-3/4}_*\big)+CL^{-1}_*\,.
\end{equation}
\end{corollary}

We emphasize that the large flip rate limit is implied from \eqref{eq:mainresult}.
\begin{remark}[Large flip rate limit]
If we choose
$L=1$ and $T$ arbitrary but fixed, $\lambda\to \infty$ and $|v_0|=c\to \infty$ in such a way that $c^2/\lambda\to 1$ we reproduce the limit used in Theorem~1.1 of \cite{Bogachev2011}. On the other hand, then we have
\[
C\sqrt{{T_*}{L^{-2}_*}}T^{-1/4}_*\big(\sqrt{\ln(T_*+3)}+T^{-3/4}_*\big)+CL^{-1}_*=
\mathcal{O}\big(\lambda^{-1/4}\sqrt{\ln(\lambda)}\big)\,.
\]
Therefore, Corollary~\ref{cor:tele} with the help of Remark~\ref{rem:difconst} implies that the process 
$\mathbb{X}^{\fs}_{[0,T]}$ converges
weakly in $\mathcal{C}[0,T]$ to the distribution of a standard Brownian motion $\mathbb{W}_{[0,T]}$.
\end{remark}
In addition, Theorem~\ref{th:main} allows us to derive estimates for the time average of the $p$-th moments as follows.
\begin{corollary}[Time average of $p$-th  moments]\label{cor:absmoments}
Suppose that all assumptions made in 
Theorem~\ref{th:main} hold.
Let $p>0$ and consider the observable $f(x)=|x|^p$ for any $x\in \mathbb{R}$. Then for any $L>0$, $T>0$, $\lambda>0$ and $v_0\in \mathbb{R}\setminus\{0\}$ it follows that 
\begin{equation}
\begin{split}
& \left|\mean{\frac{1}{T} \int_{0}^{T}\rmd s \,f(L^{-1}X(s;\fs))}_{{\mu_1}}-\mean{\frac{1}{T} \int_{0}^{T}\rmd s\, f(B(s))}_{{\mu_2}}\right|\\
&\hspace{9cm}\leq
2p\max\{C_1,C_2\}\cW_2\left(L^{-1}\mathbb{X}^{\fs}_{[0,T]},\mathbb{B}_{[0,T]}\right)\,,
\end{split}
\end{equation}
where $C^2_1=C^2_2=1$ if $p=1$, whereas if $p\neq 1$ we have
\[
C^2_1\leq 
\frac{\widetilde{C}T^{p-1}_*}{p L^{2(p-1)}_*}
+
\frac{{C^\prime}T^{p-2}_*}{(p-1) L^{2(p-1)}_*},\quad  \quad
C^2_2=\frac{2^{p-1}\Gamma(\frac{2p-1
}{2})}{p\sqrt{\pi}}\frac{T^{p-1}_*}{L^{2(p-1)}_*}\,,
\]
the constants 
$T_*$ and $L_*$ are the scaling parameters that appear in  \eqref{e:scalings2020},
$\widetilde{C}$ and $C^\prime$ are the constants that appear for $r=2(p-1)$ in Item~(2), \eqref{eq:momentskacr}, 
of Lemma~\ref{lem:momentskac} in 
Appendix~\ref{ap:coupling} and $\Gamma$ denotes the Gamma function.

\noindent
In addition, if  $p\in \mathbb{N}$ and $g(x)=x^p$ for $x\in \mathbb{R}$, then 
\begin{equation}
\begin{split}
& \left|\mean{\frac{1}{T} \int_{0}^{T}\rmd s\, g(L^{-1}X(s;\fs))}_{{\mu_1}}-\mean{\frac{1}{T} \int_{0}^{T}\rmd s\, g(B(s))}_{{\mu_2}}\right| \\
&\hspace{9cm}
\leq
2p\max\{C_1,C_2\}\cW_2\left(L^{-1}\mathbb{X}^{\fs}_{[0,T]},\mathbb{B}_{[0,T]}\right)\,.
\end{split}
\end{equation} 
\end{corollary}

\begin{proof}
The proof follows directly from inequality \eqref{eq:momentCs}, 
Lemma~\ref{lem:momentabs} in Appendix~\ref{ap:coupling} and 
Lemma~\ref{lem:averagew}
 in Appendix~\ref{ap:coupling}.
\end{proof}

Additionally, Theorem~\ref{th:main} holds true in $\mathcal{W}_p$ 
for $p\in {[1,2]}$.
\begin{remark}[Convergence in $\mathcal{W}_p$, 
$p\in {[1,2]}$]\label{eq:p01}
Analogously to \eqref{def:metric}, for any $p\geq 1$ we define the $p$-th average distance $\mathcal{W}_p$ distance by
\begin{equation}
\begin{split}
\cW_p(\mu_1,\mu_2):&= \left(\inf_{\gamma} \int \gamma(\rmd X,\rmd Y)
 (c_2(X,Y))^{p/2}\right)^{{1}/{p}}.
\end{split}
\end{equation}
By  H\"older's inequality we obtain that Theorem~\ref{th:main}, Corollary~\ref{cor:tele} and Corollary~\ref{cor:absmoments} are valid for the $\mathcal{W}_p$ for $p\in [1,2]$.
\end{remark}

\section{\textbf{Outline of the proof of Theorem~\ref{th:main}}}\label{Sec:out}

The idea of the proof is to compare a suitable rescaled free velocity flip process $L^{-1}X=(L^{-1}X(t;\fs):t\geq 0)$ with a simplified process (decoupled process) $Y:=(Y(t):t\geq 0)$ in which computations can be carried out.
However, the process $Y$ only possesses zero mean at even jump times. Hence, it is reasonable to compare the simplified process $Y$ at even jump times with a Brownian motion $B=(B(t):t\geq 0)$ with a suitable diffusivity constant.

By the triangle inequality, the estimation of the left-hand side of \eqref{eq:mainresult} is divided in three main components as follows:
\begin{equation}\label{eq:I1I2I3}
\begin{split}
\mathcal{W}_2\left(L^{-1}\mathbb{X}^{\fs}_{[0,T]},\mathbb{B}_{[0,T]}\right)&\leq
\mathcal{W}_2\left(L^{-1}\mathbb{X}^{\fs}_{[0,T]}, \mathbb{Y}_{[0,T]}\right)+ 
\mathcal{W}_2\left(\mathbb{Y}_{[0,T]},\mathbb{Z}_{[0,T]}\right)+
\mathcal{W}_2\left(\mathbb{Z}_{[0,T]},\mathbb{B}_{[0,T]}\right)\,.
\end{split}
\end{equation}
where the process $Z=(Z(t):t \geq 0)$ is the C\`{a}dl\`{a}g process obtained from the process $Y$ at even jumps by a constant piecewise interpolation between even jump times. The explicit construction will be given below.

\subsection{\textbf{Definition of the process $Y$}}\label{sec:defY}
Decoupling techniques are very useful in the theory and applications of Probability and Statistics to replace (or reduce) the dependence of the system. For further discussion we recommend the monograph \cite{delaPea1999}.
We note that \eqref{def:path} can be seen as a perturbed random walk. However, we point out that $n=N(t;\fs)$ and $X(t;\fs)$ are highly  dependent.
To overcome this difficulty,
we define a suitable decoupled process $(Y(t):t\geq 0)$ which is indeed, a perturbed random walk. Then using the cost function \eqref{eq:costfunction} and the metric \eqref{def:metric}
 we quantify the error term between
$L^{-1}\mathbb{X}^{\fs}_{[0,T]}$ and  $\mathbb{Y}_{[0,T]}$.

Let $\fu=(u_j:j\in \mathbb{N})$ be sequence of independent and identically distributed (i.i.d.\ for short) random variables with exponential distribution of parameter one defined in the probability space $(\Omega_1,\mathcal{F}_1,\mathbb{P}_1)$ .
For each (deterministic) $n\in \mathbb{N}$, we define $\fu^{(n)}:=(u^{(n)}_k:k\in \mathbb{N})$ as follows
\begin{equation}\label{eq:utruncado}
u^{(n)}_k:=u_k\quad \textrm{ for  } k\in \{1,\ldots,n\}\quad \textrm{ and } 
\quad u^{(n)}_k:=2\beta_n\quad \textrm{ for } k\in \{n+1,n+2,\ldots,\}\,.
\end{equation}
where $\beta_n=n$ (the value of $\beta_n$ will later be fixed by an optimisation argument, hence the somewhat peculiar notation here).
Then we also define
\begin{equation}\label{eq:skn}
\delta^{(n)}_k:=\frac{T}{\beta_n}u^{(n)}_k\quad \textrm{ for  } k\in \N\,.
\end{equation}
Note that then $\delta^{(n)}_k=2T$ if $k\ge n+1$.
We consider the corresponding ``jump-times''
\begin{equation}\label{eq:tmn}
\quad t^{(n)}_m:=\sum\limits_{k=1}^{m}{\delta^{(n)}_k}\quad \textrm{ for all }\quad m\in \mathbb{N} \quad \textrm{ and } \quad t^{(n)}_0=0\,,
\end{equation}
and collect its increments as $\fd^{(n)}=(\delta^{(n)}_1,\delta^{(n)}_2,\ldots,)$.
Since $\fd^{(n)}\in \Delta_\infty$, where $\Delta_\infty$ is defined in \eqref{def:domain}, using the map \eqref{def:path} we can define the stochastic process $(Y(t;n,\fd^{(n)}):t\geq 0)$  by the following rule: for any $t\ge 0$
\begin{equation}\label{def:Ytrunk}
\begin{split}
Y(t;n,\fd^{(n)})&:=L^{-1}X(t;\fd^{(n)})\\
&=L^{-1}v_0\sum_{k=1}^{M} (-1)^{k-1} \delta^{(n)}_k + L^{-1}v_0(-1)^{M} (t-t_{M}(\fd^{(n)}))\,, \quad M=N(t;\fd^{(n)})\,,
\end{split}
\end{equation}
where $(N(t;\fd^{(n)}):t\geq 0)$ is given in \eqref{eq:poproc}.
In addition, we take a
random variable $K$ with Poisson distribution with parameter $T_*=\lambda T$ defined in the probability space $(\Omega_2,\mathcal{F}_2,\mathbb{P}_2)$ and independent of $\fu$. 
For short, for any $\lambda>0$,
$(p_n(\lambda):n\in \N_0)$ denotes the Poisson distribution $\mathsf{Po}(\lambda)$, that is,
\begin{equation}\label{eq:poissonprob} 
p_n(\lambda):=\frac{1}{n!}\lambda^{n}   \rme^{-\lambda}\quad \textrm{ for any } \quad n\in \N_0\,.
\end{equation}

From now on, we will use a convenient generic indicator function notation 
$\cf{P}$ where $P$ is some condition; more precisely, given a condition $P(x)$ which depends on the variable $x$, the corresponding indicator function is defined as the map
$x\mapsto\cf{P}(x)\in \set{0,1}$ where $\cf{P}(x)=1$, if $P(x)$ is true, and
$\cf{P}(x)=0$ otherwise.  We will also often use the alternative notation $\cf{P(x)}$ to denote $\cf{P}(x)$.

Let 
$(\Omega,\mathcal{F},\mathbb{Q}):=(\Omega_1\times \Omega_2,\mathcal{F}_1\otimes\mathcal{F}_2,\mathbb{P}_1\otimes\mathbb{P}_2)$ be the product probability space of 
$(\Omega_1,\mathcal{F}_1,\mathbb{P}_1)$ and $(\Omega_2,\mathcal{F}_2,\mathbb{P}_2)$ and denote by $\mathbb{E}_{\mathbb{Q}}$ the expectation with respect to $\mathbb{Q}$.
On $(\Omega,\mathcal{F},\mathbb{Q})$ 
we define the stochastic process $(Y(t;K,\fu):t\geq 0)$ as follows: for each $\omega=(\omega_1,\omega_2)\in \Omega$ we set
\begin{equation}\label{def:KU}
Y(t;K,\fu)(\omega):=\sum_{n=0}^\infty Y(t;n,\fd^{(n)})(\omega_1) \cf{K=n}(\omega_2)\quad \textrm{ for any }\quad t\geq 0\,,
\end{equation}
where for each (not random) $n\in \mathbb{N}$, $(Y(t;n,\fd^{(n)}):t\geq 0)$
is defined in \eqref{def:Ytrunk} and for $n=0$ we define $Y(t;0,\fd^{(0)}):=L^{-1}v_0t$ for all $t\geq 0$.
Let $F:\mathcal{C}([0,T],\mathbb{R})\to \mathbb{R}$ be a bounded measurable function. In particular, it only depends on the values of the process up to time $T$  such as $c_1(X)$ and $c_2(X)$ given in \eqref{eq:Cdef}.
By Fubini's theorem we have 
\[
 \E_\mathbb{Q}[F(Y(\cdot;K,\fu))]=\sum_{n=0}^\infty \E_\mathbb{Q}[F(Y(\cdot;K,\fu)) \cf{K=n}]\,,
\]
where 
\begin{equation}\label{def:Yinte}
\begin{split}
\E_\mathbb{Q}[F(Y(\cdot;K,\fu))\cf{K=n}]&=\E_\mathbb{Q}[F(Y(\cdot;n,\fu))\cf{K=n}]=p_n(T_*)   \E_\mathbb{Q}[F(Y(\cdot;n,\fu))]\\
&= p_n(T_*)
   \int_{\R_*^{n}}\rmd^{n} \fu \, \rme^{- \sum_{k=1}^n u_k}
  F\left(L^{-1} X\left(\cdot;T\beta^{-1}_n \fu^{(n)}\right)\right) \,,
\end{split}
\end{equation}
$\fu^{(n)}$ is given in \eqref{eq:utruncado}, and  the map $X$ is defined in \eqref{def:path01}.
As evident in  the last equality, the function $\fu\mapsto F(Y(\cdot;n,\fu))$ only depends on the components $u_1,\ldots, u_n$.

Bearing all this in mind, we compare $L^{-1}\mathbb{X}^{\fs}_{[0,T]}$ with the decoupled process $\mathbb{Y}_{[0,T]}$.

\begin{proposition}[Coin-flip coupling]\label{prop:coinflipcoupling}
There exists a (pure) constant $\kappa_1>0$ such that for any $L>0$, $T>0$, $\lambda>0$ and $v_0\in \mathbb{R}\setminus\{0\}$ it follows that
\begin{equation}\label{eq:cflipine}
\mathcal{W}_2\left(L^{-1}\mathbb{X}^{\fs}_{[0,T]}, \mathbb{Y}_{[0,T]}\right)\leq \kappa_1 \sqrt{{T_*}{L^{-2}_*}}T^{-1/4}_*\sqrt{\ln(T_*+3)}+\kappa_1L^{-1}_*,
\end{equation}
where the constants 
$T_*$ and $L_*$ are given in \eqref{e:scalings2020}.
\end{proposition}
The complete proof is given in Section~\ref{sec:proof} and it is based on the coin-flip coupling 
(see \cite{Koskinen2020} or Appendix~A in \cite{Saksman2020}).

\subsection{\textbf{Definition of the process $Z$}}

In the sequel, we define the process $Z$.
We note that \eqref{eq:suma} and \eqref{eq:poproc} implies
$N(t_{n};\fs)=n$ for all $n\in \mathbb{N}_0$.
Then \eqref{def:path} yields for all $n\in \mathbb{N}$
\begin{equation}
\begin{split}
L^{-1}X(t_{2n}(\fs);\fs)-L^{-1}X(t_{2(n-1)}(\fs);\fs)=-L^{-1}v_0\left(s_{2n}-s_{2n-1}\right)\,.
\end{split}
\end{equation}
That is, at even jumps, the increments of the process $L^{-1}X$ are i.i.d.\ random variables which can be described explicitly.
Recall the notations \eqref{eq:utruncado}, \eqref{eq:skn} and \eqref{eq:tmn} used to define the process $Y$.
In what follows, we define the stochastic process $Z$ as the constant C\`{a}dl\`{a}g interpolation between the values of $Y$ at even jumps.

We recall that $\lfloor \cdot \rfloor$ denotes the so-called floor function, that is, for any $x\in \mathbb{R}$, $\lfloor x \rfloor$ gives the largest integer less than or equal than $x$. 
For $n\in \mathbb{N}\setminus\{0,1\}$ we define the C\`adl\`ag version of the random walk $(\eta^*_k(\fs):k\in \mathbb{N}_0)$ given by
\begin{equation}\label{def:Ztrunk}
Z(t;n,\fs)=\sum\limits_{k=1}^{\widetilde{n}}\cf{t\geq  t_{2k}(\fs)}\eta^*_k(\fs) \quad \textrm{ for any }\quad t\geq 0\,,
\end{equation}
where $\widetilde{n}=\lfloor n/2 \rfloor$,
\begin{equation}\label{def:randomwalk}
\eta^*_k(\fs)=-L^{-1}v_0(s_{2k}-s_{2k-1}),\quad k\in \mathbb{N},\quad \eta^*_0(\fs)=0\,,
\end{equation}
and $\fs=(s_1,s_2,\ldots,)=\lambda^{-1}\mathfrak{u}=\lambda^{-1}(u_1,u_2,\ldots)$
is distributed according to the infinitely product measure of exponential distributions with parameter $\lambda>0$.
For $n\in \{0,1\}$ we set $Z(t;n,\fs)=0$ for all $t\geq 0$. Analogously to \eqref{def:KU} we define
on $(\Omega,\mathcal{F},\mathbb{Q})$ 
the stochastic process $(Z(t;K,\fu):t\geq 0)$ as follows: for each $\omega=(\omega_1,\omega_2)\in \Omega$ it is given by
\begin{equation}\label{def:KUZ}
Z(t;K,\fs)(\omega):=\sum_{n=0}^\infty Z(t;n,\fs)(\omega_1) \cf{K=n}(\omega_2)\quad \textrm{ for any }\quad t\geq 0\,.
\end{equation}
We remark that the process $Y$ and the process $Z$
are constructed in the same probability space $(\Omega,\mathcal{F},\mathbb{Q})$.

\begin{proposition}[Synchronous coupling]\label{prop:naturalcoupling}
There exists a (pure) constant $\kappa_2>0$ such that for any $L>0$, $T>0$, $\lambda>0$ and $v_0\in \mathbb{R}\setminus\{0\}$ it follows that
\begin{equation}\label{eq:pepino}
\mathcal{W}_2\left(\mathbb{Y}_{[0,T]}, \mathbb{Z}_{[0,T]}\right)\leq \kappa_2 \sqrt{{T_*}{L^{-2}_*}}
T^{-1/4}_*(1+T^{-3/4}_*)\,,
\end{equation}
where the constants $T_*$ and $L_*$ are given in \eqref{e:scalings2020}.
\end{proposition}
The complete proof can be found in Section~\ref{sec:naturalcoupling} and 
it is based on the synchronous coupling of the processes $Y$ and $Z$ which are constructed with the same sequence of random variables as an input.
\begin{lemma}[Koml\'os--Major--Tusn\'ady coupling]\label{lem:couplingKMT}
There exists a (pure) constant $\kappa_3>0$ such that for any $L>0$, $T>0$, $\lambda>0$ and $v_0\in \mathbb{R}\setminus\{0\}$ it follows that
\begin{equation}
\mathcal{W}_2\left(\mathbb{Z}_{[0,T]}, \mathbb{B}_{[0,T]}\right)\leq \kappa_3 \sqrt{T_*L^{-2}_*}T^{-1/4}_*
\big(
1+T^{-3/4}_*\big)\,,
\end{equation}
where the constants $T_*$ and $L_*$ are given in \eqref{e:scalings2020}
and the diffusivity constant of the Brownian motion $B:=(B(t):t\geq 0)$ is defined in \eqref{eq:diffusivity}.
\end{lemma}
The proof is given in Section~\ref{sec:KMTcoupling} and it relies on the Koml\'os--Major--Tusn\'ady coupling (see \cite{Komlos1975} and \cite{Komlos1976}).

We stress the fact that Theorem~\ref{th:main} is just a consequence of what we have already stated up to here.
\begin{proof}[Proof of Theorem~\ref{th:main}]
It follows directly from inequality \eqref{eq:I1I2I3} with the help of Proposition~\ref{prop:coinflipcoupling}, Proposition~\ref{prop:naturalcoupling} and Lemma~\ref{lem:couplingKMT}.
\end{proof}

\section{\textbf{Proof of Proposition~\ref{prop:coinflipcoupling}: Coin-flip coupling}}\label{sec:proof}

In this section we prove Proposition~\ref{prop:coinflipcoupling}. We start with some preliminaries.
\subsection{\textbf{Expectation for continuous observables of the free velocity flip model}}
Since then $N(t;\fs)<\infty$ almost surely, for any (regular enough, e.g., positive and measurable) observable $F(\mathbb{X}^{\fs}_{[0,T]})$ 
which depends only on the path up to times $T>0$, such as $c_1$ and $c_2$ defined in \eqref{eq:Cdef},
we can decompose the events depending on the realisation of $N(T;\fs)$. 
We recall the generic indicator function notation $\cf{P}$ from Section~\ref{sec:defY}, and obtain
\[
 \E[F(\mathbb{X}^{\fs}_{[0,T]})] = \sum_{n=0}^\infty \E[F(\mathbb{X}^{\fs}_{[0,T]}) \cf{N(T;\fs)=n}]\,.
\]
This partition helps, since for any $n$ the path up to time $T$ and thus also the value of $F(\mathbb{X}^{\fs}_{[0,T]})$, only depends on $s_k$, for $0\leq k\leq n+1$ which is crucial in our argument. 
First, for $n=0$, 
we can directly evaluate 
\[
\E[F(\mathbb{X}^{\fs}_{[0,T]}) \cf{N(T;\fs)=0}]
=F(I) \rme^{-\lambda T}\,,
\quad \textrm{ where }\quad I(t)=v_0 t\quad \textrm{ for all }\quad t\geq 0\,.
\] 
Therefore, 
in the following computation we only need to consider values $n\in \N$.

For such $n$,  we have the following explicit integral representation
\begin{equation*}
  \E[F(\mathbb{X}^{\fs}_{[0,T]}) \cf{N(T;\fs)=n}]
  = \lambda^{n+1}
  \int_{\R_*^{n+1}} \rmd^{n+1}\fs \, \prod_{k=1}^{n+1}\rme^{-\lambda s_k}
  F(\mathbb{X}^{\fs}_{[0,T]}) \cf{t_n(\fs)\leq T} \cf{s_{n+1}>T-t_{n}(\fs)}\,,
\end{equation*}
where $(t_n(\fs):n\in \mathbb{N}_0)$ are the jump times defined in \eqref{eq:suma}.
The notation is somewhat formal, since the full path 
$\mathbb{X}^{\fs}$ depends on the full sequence $\fs$.  However, whenever both indicator functions are non-zero in the above integrand, we note that $\mathbb{X}^{\fs}_{[0,T]}$ only depends on the values $s_k$, $k=1,2,\ldots,n$.
Explicitly, defining $v_k:= s_k/T$, $k=1,2,\ldots,n$, and setting somewhat arbitrarily $v_k=2$ for $k>n$, we find that 
then $\mathbb{X}^{\fs}_{[0,T]}=X(\cdot;T\fv)_{[0,T]}$
since the conditions in the integrand guarantee that the jump $n+1$ occurs after $T$ which is also ensured by the above choice $T v_{n+1}=2T$.  In this case, for simplicity, we also drop the restriction to the time-interval $[0,T]$ from the notation, and simply write $F(\mathbb{X}^{\fs}_{[0,T]}) = F(X(\cdot;T \fv))$ in the above integral.

Performing the implied change of variables, setting $\tilde{r}=s_{n+1}/T$, and using Fubini's theorem to reorder the integrals, we obtain
\begin{equation}\label{eq:integralrepre}
\begin{split}
  \E[&F(\mathbb{X}^{\fs}_{[0,T]}) \cf{N(T;\fs)=n}]
 \\ 
 & 
  = T_*^{n+1}
  \int_{\R_*^{n}} \rmd^{n}\fv \,
    \left(\int_{\R_*} \rmd\tilde{r}\,
\rme^{-T_* (\tilde{r} + t_n(v))}
  F(X(\cdot;T \fv)) \cf{t_n(\fv)\leq 1}
   \cf{\tilde{r}>1-t_{n}(\fv)}\right)\,,
\end{split}
\end{equation}
where $T_*=\lambda T$, and we have used the obvious scaling  property 
$t_n(T\fv)=T t_n(\fv)$.
The $\tilde{r}$-integral can then be evaluated explicitly using $\int_{1-t_n(\fv)}^\infty \rmd\tilde{r}\,
\rme^{-T_* (\tilde{r} + t_n(v))}=T_*^{-1} \rme^{-T_*}$.
The resulting integral representation can be further simplified by relying on Dirac $\delta$-functions whose precise construction as non-negative Radon measures is explained in Appendix~\ref{ap:delta}.
In particular, applying item iv) of 
Lemma~\ref{lem:deltaprop} justifies
the result from the following formal computation:
\begin{equation}\label{def:Xint}
\begin{split}
  \E[& F(\mathbb{X}^{\fs}_{[0,T]})  \cf{N(T;\fs)=n}] 
 \\ 
 & 
    = T_*^{n}\rme^{-T_*}
  \int_{\R_*^{n}} \rmd^{n}\fv \,
  F(X(\cdot;T \fv)) \cf{t_n(\fv)\leq 1}
 \\  
 & 
    = T_*^{n}\rme^{-T_*}
  \int_{\R_*^{n}} \rmd^{n}\fv \,
  F(X(\cdot;T \fv)) \cf{t_n(\fv)\leq 1}
  \int_{\R_*} \rmd r \,\delta(t_n(\fv)-r)
 \\  
 & 
    = T_*^{n}\rme^{-T_*}
  \int_{\R_*} \rmd r \,  \cf{r\leq 1} \left(\int_{\R_*^{n}} \rmd^{n}\fv \,
  F(X(\cdot;T \fv))
\delta(t_n(\fv)-r)\right)
 \\  
 & 
  = \frac{1}{n!}T_*^{n}   \rme^{-T_*} \int_0^1 \rmd r\, 
n r^{n-1}
  (n-1)!\int_{\R_*^{n}} \rmd^{n} \fu \,\delta(t_n(\fu)-1)
  F(X(\cdot;T r \fu)) \,,
\end{split}
\end{equation}
with a slight abuse of notation since the values of the sequence ``$Tr \fu$'' for $k>n$ are not rescaled by $r$, i.e., they are still equal to $2T$. 

Since $\int_{\R_*^{n}} \rmd^{n} \fu \,\delta(u_n(\fu)-1)=\frac{1}{(n-1)!}$, the remaining
integrals form a product of two probability measures for $r$ and $\fu$, i.e., 
\begin{equation}\label{eq:probr}
\rmd r \cdot nr^{n-1}\cf{r\in (0,1]}
\end{equation}
and
\begin{equation}\label{eq:probu}
\rmd^{n}\fu \cdot (n-1)! \delta(t_n(\fu)-1)\,.
\end{equation}
By straightforward computations we have
\[
 \mean{r}=1-\frac{1}{n+1}\,,\quad \mean{r^2}-\mean{r}^2 = 
 \frac{n}{n+2}-\frac{n^2}{(n+1)^2}= \frac{n}{(n+1)^2(n+2)}\leq \frac{1}{(n+1)^2}\,,
\]
where the expectations are computed with respect to the probability measure \eqref{eq:probr}.
Hence, $r$ is distributed close to $1$ with standard deviation $\mathcal{O}(1/n)$ as $n\to \infty$.
Using the permutation invariance of \eqref{eq:probu}, we also find the following statistics for components of $\fu$, 
\[
 \mean{u_i}  = \frac{1}{n}\,, \quad \mean{u_i^2}  = \frac{2}{n(n+1)}
 \,, \quad \mean{u_i^3}  = \frac{6}{n(n+1)(n+2)}
 \,, \quad \mean{u_i u_j}= \frac{1}{n(n+1)}\,,\quad i\ne j\,,
\]
where the expectations are computed with respect to the probability measure \eqref{eq:probu}.
For details see Lemma~\ref{lem:moments} 
in Appendix~\ref{ap:tools}.
In particular, the standard deviation of each component $u_i$ is also $\mathcal{O}(1/n)$.

We now rescale the jumps for the process $(Y(t;n,\fs):0\leq t\leq T)$  so that a good coupling can be formed between it and $(L^{-1}X(t;n,\fs):0\leq t\leq T)$.  
Assume that $n\in \N$ is deterministic, and for shorthand we set    
$f(T\fv):=F(L^{-1} X(\cdot;T\fv))$. Then \eqref{def:Yinte} implies  
\begin{equation}\label{eq:nu1nu2}    
\begin{split}
\E_{\mathbb{Q}}[F(Y(\cdot;n,\fu))]
  &= 
  \beta^n_n \int_{\R_*^{n}}\rmd^{n} \fv \, \rme^{-\beta_n \sum_{k=1}^n v_k} f(T \fv)  
\\ 
& \hspace{-2cm}= 
  \beta^n_n \int_{\R_*}\rmd r \, \rme^{-\beta_n r} 
   \int_{\R_*^{n}}\rmd^{n} \fv \, \delta\!\left(\sum_{k=1}^n v_k-r\right) f(T \fv)    
\\ & \hspace{-2cm}
  = 
  \beta^n_n \int_{\R_*}\rmd r \, r^{n-1} \rme^{-\beta_n r} 
   \int_{\R_*^{n}}\rmd^{n} \fu \, \delta\!\left(\sum_{k=1}^n u_k-1\right) f(T r \fu)    
\\ 
 & \hspace{-2cm}
  = 
  \frac{1}{(n-1)!}\beta^n_n \int_{\R_*}\rmd r \, r^{n-1} \rme^{-\beta_n r} 
   (n-1)! \int_{\R_*^{n}}\rmd^{n} \fu \, \delta\!\left(\sum_{k=1}^n u_k-1\right) f(T r \fu)
   \,.
\end{split}
\end{equation}
\subsection{\textbf{Coin flip type coupling}}
Let
\begin{equation}\label{eq:defnu1nu2}
\nu_1(\rmd r) =n r^{n-1} \cf{0< r\leq 1} \rmd r
 \quad \textrm{ and } \quad 
\nu_2(\rmd r) = \frac{1}{(n-1)!}\beta^n r^{n-1} \rme^{-\beta r}\cf{r>0} \rmd r\,
\end{equation}
denote probability measures on $\R_ *$.
We note that 
$\nu_2$ is the Gamma distribution of parameters $n$ and $\beta$.
Moreover, it has
mean $n/\beta$
and variance $n/\beta^2$.
There is a choice of $\beta$ which makes the densities agree ``with a high probability'',
and the following computations yield that the original choice $\beta=\beta_n=n$ suffices.
By \eqref{def:Xint} and \eqref{eq:nu1nu2} we only need to generate a good coupling between 
$\nu_1$ and $\nu_2$.
For this case, one good choice is using a ``coin-flip type'' coupling. 
Let
\begin{equation}\label{def:g}
 g(r) := n! \cf{0<r\leq 1} \beta^{-n} \rme^{\beta r}\,
\end{equation}
and define a coupling $\gamma_0$
by setting the action  for any observable $h:\R^2_*\to \R$ as follows
\begin{equation}
\begin{split}
 \int_{\R_*^2} \gamma_0(\rmd r_1,\rmd r_2) h(r_1,r_2)
 &=  \int_{\R_*} \nu_2(\rmd r) \min\{1,g(r)\} h(r,r)
 \\ & \qquad
 +  \int_{\R_*} \nu_2(\rmd r_1) \int_{\R_*} \nu_2(\rmd r_2)\, 
 \frac{1}{Z_{\beta,n}} (1-g(r_2))_+ (g(r_1)-1)_+ h(r_1,r_2)\,,
\end{split}
\end{equation}
where $(x)_+ := x\cf{x> 0}$ and 
\begin{equation}\label{eq:normalisation}
 Z_{\beta,n}=\int_{\R_*} \nu_2(\rmd r) (1-g(r))_+ = \int_{\R_*} \nu_2(\rmd r) (g(r)-1)_+ >0\,.
\end{equation}
This is indeed a coupling between $\nu_1$ and $\nu_2$, see for instance relation (5.11) of \cite{Lukkarinen2019}, where it has been applied the techniques from Appendix~A in \cite{Saksman2020}.
We stress that $\nu_1$, $\nu_2$, $g$ and $\gamma_0$ depend on $n$, however, for ease of notation we drop its dependence.

Finally, using this $\gamma_0$ we define a coupling  $\gamma$ between $L^{-1}\mathbb{X}^{\fs}_{[0,T]}$ and  $\mathbb{Y}_{[0,T]}$  as follows
\begin{equation}\label{eq:gammacoupling}
 \begin{split}
 & \E_\gamma[F(L^{-1}\mathbb{X}^{\fs}_{[0,T]},\mathbb{Y}_{[0,T]})]
 =  \rme^{-T_*}  
 F(L^{-1} I,L^{-1} I)
 + \sum_{n=1}^\infty 
p_n(T_*) 
 \\ & \qquad 
\times \int_{\R_*^2} \gamma_0(\rmd r_1,\rmd r_2) 
   (n-1)! \int_{\R_*^{n}}\rmd^{n} \fu \, \delta\!\left(\sum_{k=1}^n u_k-1\right) 
   F(L^{-1}X(\cdot;T r_1 \fu),L^{-1} X(\cdot;T r_2 \fu))\,,
 \end{split}
 \end{equation}
where $I=I(t)=v_0t$ for all $t\geq 0$.
One can now check that this has the right marginals.
Thus for the choice $F=c_2$ given in \eqref{eq:costfunction}, we need to control
 \begin{equation}
 c_2(L^{-1}X(\cdot;T r_1 \fu),L^{-1} X(\cdot;T r_2 \fu))
 = \frac{1}{L^2 T} \int_0^{T}\rmd t\,|X(t;T r_1 \fu)-X(t;T r_2 \fu)|^2\,.
 \end{equation}
For $r_2=r_1$, the value of the preceding cost is clearly zero whatever $n$ and $\fu$ are.  Therefore, 
\begin{equation}\label{eq:W2est}
\begin{split}
 & \E_\gamma[c_2(L^{-1}\mathbb{X}^{\fs}_{[0,T]},\mathbb{Y}_{[0,T]})]
 = \sum_{n=1}^\infty p_n(T_*)
   (n-1)! \int_{\R_*^{n}}\rmd^{n} \fu \, \delta\!\left(\sum_{k=1}^n u_k-1\right)\\ & \quad
\times  \int_{\R_*} \nu_2(\rmd r_1) \int_{\R_*} \nu_2(\rmd r_2)\, 
 \frac{1}{Z_{\beta,n}} (1-g(r_2))_+ (g(r_1)-1)_+ c_2(L^{-1}X(\cdot;T r_1 \fu),L^{-1} X(\cdot;T r_2 \fu))\,. 
\end{split}
\end{equation}
Then  for $t\in [0,T]$ we rewrite $X(t;T r_1 \fu)-X(t;T r_2 \fu)$  as follows
\begin{equation}\label{eq:nin}
\begin{split}
 X(t;T r_1 \fu)-X(t;T r_2 \fu)&= v_0 D_1(t,\fu,r_1,r_2,T) + v_0 D_2(t,\fu,r_1,r_2,T)\\
 &\qquad+ v_0 D^{\prime}_1(t,\fu,r_1,r_2,T) + v_0 D^{\prime}_2(t,\fu,r_1,r_2,T)\,,
\end{split}
\end{equation}
where, writing $n_i=n_i(t)=N(t;T r_i \fu)\leq n$, 
we have
\begin{equation}\label{eq:D1estimate}
D_1(t,\fu,r_1,r_2,T):= T(r_1-r_2) \sum_{k=1}^{n_1} (-1)^{k-1} u_k\,,
\end{equation}
\begin{equation}\label{eq:D2estimate}
|D_2(t,\fu,r_1,r_2,T)|:= T r_2 \big|\sum_{\min\{n_1,n_2\}<k\leq \max\{n_1,n_2\}} (-1)^{k-1} u_k \big|\,,
\end{equation}
\begin{equation}\label{eq:incompletejumps}
\begin{split}
D^{\prime}_1(t,\fu,r_1,r_2,T)&:=
(-1)^{n_1} (t-t_{n_1}(T r_1\fu))\,,\\
D^{\prime}_2(t,\fu,r_1,r_2,T)&:=
(-1)^{n_2} (t-t_{n_2}(T r_2\fu))\,.
\end{split}
\end{equation}
For short, for each $j\in \{1,2\}$ we write $D_j$, $D^{\prime}_j$ instead of $D_j(t,\fu,r_1,r_2,T)$,
$D^{\prime}_j(t,\fu,r_1,r_2,T)$, respectively.
By  H\"older's inequality we obtain the following upper bound 
\[
 |X(t;T r_1 \fu)-X(t;T r_2 \fu)|^2 \leq 4v^2_0 \left(|D_1|^2 + |D_2|^2 + |D^{\prime}_1|^2 + |D^{\prime}_2|^2\right)\, \textrm{ for all }\quad t\in [0,T]\,.
\]
In what follows, we  find a positive constant $\kappa_1$  that does not depend on $v_0$, $\lambda$, $L$ and $T$ and such that 
\begin{equation}\label{equ:W}
\begin{split}
\mathcal{W}^2_2\left(L^{-1}\mathbb{X}^{\fs}_{[0,T]}, \mathbb{Y}_{[0,T]}\right)& \leq |v_0|^2\E_\gamma[c_2(L^{-1} X,Y)]\\
&\leq \kappa_1 |v_0|^2 \left(\lambda^{-3/2}{L}^{-2}{T}^{1/2}
\ln(\lambda{T}+3)+\lambda^{-2}{L}^{-2}\right)\,,
\end{split}
\end{equation}
where $\gamma$ is the coupling measure defined in \eqref{eq:gammacoupling}.
By \eqref{e:scalings} 
we have that $T_*=\lambda T$ and ${L}_*=|v_0|^{-1}\lambda L$ and hence 
inequality \eqref{equ:W} reads as follows
 \begin{equation}
\mathcal{W}^2_2\left(L^{-1}\mathbb{X}^{\fs}_{[0,T]}, \mathbb{Y}_{[0,T]}\right)
 \leq \kappa_1 T_*{L}^{-2}_*{T}^{-1/2}_*
\ln({T}_*+3)+\kappa_1 L^{-2}_*\,.
\end{equation}

\subsection{\textbf{Comparison between $L^{-1}\mathbb{X}^{\fs}_{[0,T]}$ and $\mathbb{Y}_{[0,T]}$}}
By \eqref{eq:W2est} and \eqref{eq:nin} we have 
\begin{equation}\label{eq:couplingflip}
\begin{split}
  \E_\gamma[c_2(L^{-1}\mathbb{X}^{\fs}_{[0,T]},\mathbb{Y}_{[0,T]})]&\leq 4v^2_0\sum_{n=1}^\infty p_n(T_*)
   (n-1)! \int_{\R_*^{n}}\rmd^{n} \fu \, \delta\!\left(\sum_{k=1}^n u_k-1\right)\\ 
   & \qquad
\times  \int_{\R_*} \nu_2(\rmd r_1) \int_{\R_*} \nu_2(\rmd r_2)\, 
 \frac{1}{Z_{\beta,n}} (1-g(r_2))_+ (g(r_1)-1)_+\\
 &\qquad \quad 
(L^2T)^{-1}\int_0^T \rmd t 
(D^2_1+D^2_2+(D^\prime_1)^2+(D^\prime_2)^2)\,. 
\end{split}
\end{equation}
We start with the apparently worst term $D_2$.
\\

\noindent
\textbf{Estimates on $D_2$}.
We point out that the term $D_2$ given in \eqref{eq:D2estimate} is complicated since $n, n_1, n_2$ and $\fu$ are not independent, and the components of $\fu$ satisfy the constraint $u_1+\cdots+u_n=1$.
 
\begin{lemma}[Remainder $D_2$]\label{lem:D2term}
There is a positive constant $C_2$ such that for all $T>0$, $\lambda>0$ and $L>0$ it follows that
\begin{equation}\label{eq:D2term}
\begin{split}
\mathcal{R}_2:=v^2_0
\sum_{n=1}^\infty p_n(T_*)\,
(L^2 T)^{-1}  
\int_{\R_*^{n}}&\,(n-1)!\,\rmd^{n}\fu\, 
\delta\left(\sum_{k=1}^n u_k-1\right)\,
\int_{\R_*} \nu_2(\rmd r_1)\,  \int_{\R_*} \nu_2(\rmd r_2)\,  
\int_0^{T}
\rmd t\,  
\times\\
&\hspace{-1cm}
|D_2|^2
\frac{1}{Z_{\beta,n}} (1-g(r_2))_+ (g(r_1)-1)_+\leq C_2T_*{L}^{-2}_*{T}^{-1/2}_*
\ln({T}_*+3)\,,
\end{split}
\end{equation}
where 
the probability measure $\nu_2$ is given in \eqref{eq:nu1nu2},
the function $g$ is given in \eqref{def:g} and the normalisation constant $Z_{\beta,n}$ is given in \eqref{eq:normalisation}.
\end{lemma}
\begin{proof}
By Fubini's theorem the left hand-side  of \eqref{eq:D2term} can be written as 
\begin{equation}\label{eq:aimD2}
\begin{split}
\mathcal{R}_2=v^2_0
\sum_{n=1}^\infty p_n(T_*)\,
 \frac{L^{-2} T}{Z_{\beta,n}}   \int_0^{T} &
\rmd t\, 
\int_{\R_*} \nu_2(\rmd r_1)\,  (g(r_1)-1)_+  \int_{\R_*} \nu_2(\rmd r_2)\, r^2_2(1-g(r_2))_+   
\times\\
 &\hspace{-1cm}
\int_{\R_*^{n}}\, 
(n-1)!\,\rmd^{n}\fu\, 
\delta\left(\sum_{k=1}^n u_k-1\right)\, \big|\sum_{\min\{n_1,n_2\}<k\leq \max\{n_1,n_2\}} (-1)^{k-1} u_k \big|^2\,.
\end{split}
\end{equation}
For each $i\in \{1,2\}$
 recall that $n_i=n_i(t)=N(t;T r_i \fu)\leq n$ and
note that if $n_1=n_2$, then $D_2=0$.  
By definition of $n_1$ and $n_2$ we have
\[
t_{n_i}(Tr_i \fu)\leq t<t_{n_i+1}(Tr_i \fu)\quad \textrm{ for each } i\in \{1,2\}\,,
\]
which reads as 
\[
Tr_i\sum\limits_{k=1}^{n_i}u_k\leq t<Tr_i\sum\limits_{k=1}^{n_i+1}u_k
\quad \textrm{ for each } i\in \{1,2\}\,.
\]
Since the integration with respect to the measures
$\nu_2(\rmd r_1)$ and $\nu_2(\rmd r_2)$ are on the set $\mathbb{R}_*$,
we assume that $r_1>0$ and $r_2>0$ and then we have  
\begin{equation}\label{eq:Trr}
\sum\limits_{k=1}^{n_i}u_k\leq \frac{t}{Tr_i}<\sum\limits_{k=1}^{n_i+1}u_k \quad \textrm{ for each } i\in \{1,2\}\,.
\end{equation}

Without loss of generality we may assume that $n_1<n_2$. 
We rewrite the preceding inequalities as follows
\begin{equation}\label{eq:r2n}
\sum\limits_{k=1}^{n_1+1}u_k
+\sum\limits_{k=n_1+2}^{n_2}u_k
\leq \frac{t}{Tr_2}<\sum\limits_{k=1}^{n_1+1}u_k
+\sum\limits_{k=n_1+2}^{n_2}u_k+u_{n_2+1}\,.
\end{equation}
The idea is to do a  change of variables for suitable indexes as follows.
For the indexes $\{1,2,\ldots,n_1,n_1+1\}\cup\{n_2+1,n_2+2,\ldots\}$ we do not do a change of variables.
Also, in the case that $n_2-n_1$ is an even number, we do not do a change of variable for the index $k=n_2$.
To be more precise,
let $\mathcal{I}:=\{n_1+2,n_1+3,\ldots,n_2\}$ be a set of indexes and define $K(n_1,n_2)=\lfloor \frac{n_2-n_1-1}{2} \rfloor$.
For $K\geq 1$ we set  
\begin{equation}\label{eq:cofvar}
\hat u_k :=u_{2k+n_1}+u_{2k+n_1+1}\quad \textrm{ and }\quad
\omega_k :=\frac{u_{2k+n_1}-u_{2k+n_1+1}}{\hat u_k}\,
\end{equation}
for each  $k\in \{1,\ldots,K\}$.
Observe that
\begin{equation}\label{eq:defK}
\sum\limits_{k\in \mathcal{I}}u_k=\sum\limits_{k=1}^{K} \hat u_k+u_{n_2}\cf{n_2-n_1-1\textrm{ is odd}}\,.
\end{equation}
In fact, the preceding equality holds even if $K=0$ recalling the convention about empty sums.
We point out that right-hand side of \eqref{eq:defK} does not depend on $(\omega_k:1\leq k\leq K)$.

In the sequel, for $j_1<j_2$ we show that
\begin{equation}\label{eq:deltaintegral4}
\begin{split}
&\mean{\left(\sum\limits_{k=1}^{K} \hat u_k 
\omega_k
\right)^2 \cf{n_1=j_1,n_2=j_2}}_{\nu(\rmd^{n}\fu)}\\
&\hspace{4cm}=
\frac{1}{3}
\int_{\R_*^{n}}\,(n-1)!
\rmd^{n} \fu \, \delta\!  
\left(\sum_{k=1}^n u_k-1\right)\sum\limits_{k=1}^{K} \hat u^2_k \cf{n_1=j_1,n_2=j_2}\,,
\end{split}
\end{equation}
where
 $\nu(\rmd^{n}\fu)=\rmd^{n}\fu \cdot (n-1)! \delta(t_n(\fu)-1)$.
We stress that $n_i=N(t;T r_i \fu)$.
It is clear that the restriction $n_1=j_1$ does not depend on $\omega_1,\ldots, \omega_{K}$ since we do not do any change of variables for $k\leq j_1+1$. However, a priori 
the condition $n_2=j_2$ could depend on $\omega_1,\ldots, \omega_{K}$.
By \eqref{eq:r2n} and
\eqref{eq:defK} we find that it is not the case.
Using the symmetry,
in the case $n_2<n_1$ we 
set $\mathcal{I}:=\{n_2+2,n_2+3,\ldots,n_1\}$ be a set of indexes and define $K(n_1,n_2)=\lfloor \frac{n_1-n_2-1}{2} \rfloor$.
 
On  both sides of \eqref{eq:deltaintegral4} we can replace the random variable $K(n_1,n_2)$ by the fixed value $K(j_1,j_2)$. For convenience, we denote $K(j_1,j_2)$ by $K$  in the following.
We note that \eqref{eq:deltaintegral4} is trivially true for $K=0$. Hence, we only need to prove for  $K \geq 1$ .
We start by observing
\begin{equation}
\begin{split}
&
\mean{\left(\sum\limits_{k=1}^{K}  \hat u_k 
\omega_k
\right)^2 \cf{n_1=j_1,n_2=j_2}}_{\nu(\rmd^{n}\fu)}\\
&\qquad \qquad=
\int_{\R_*^{n}}\,(n-1)!
\rmd^{n} \fu \, \delta\!  
\left(\sum\limits_{k=1}^{K} \hat u_k+\sum_{k\in \mathcal{J}} u_k-1\right)\left(\sum\limits_{k=1}^{K} \hat u_k 
\omega_k
\right)^2\cf{n_1=j_1,n_2=j_2}\,,
\end{split}
\end{equation}
where 
\[
\mathcal{J}:=
\begin{cases}
(\{1,2,\ldots,n\}\setminus \mathcal{I})\cup {\{n_2\}} & \textrm{ if $n_2-n_1$ is an even number}\,,\\
\{1,2,\ldots,n\}\setminus \mathcal{I} & \textrm{ if $n_2-n_1$ is an odd number}\,.
\end{cases}
\]

The following computations are based on properties of the $\delta$-constraint measures whose construction and main properties are discussed in detail in  Appendix~\ref{ap:delta}.
We apply Lemma~\ref{lem:Borelprop}, in particular, 
we wish to approximate the measure by ordinary Lebesgue integrals using the mollifier functions
$\Phi_\varepsilon(\cdot;1)$ defined there.  To do this rigorously, it is necessary that the rest of the integrand is continuous: for this reason, we first need to approximate the indicator function $\cf{n_1(\fu)=j_1,n_2(\fu)=j_2}$ by a sequence of continuous functions (for fixed $t$, $r_1$, and $r_2$).  This can be done even monotonously using non-negative functions and without introducing a dependence on the variables $\omega_1,\ldots, \omega_{K}$, and thus the limits of the sequence can be taken in the integrands, in the beginning and at the end of the computation, using monotone convergence theorem.  To simplify the notation, we make the computations using the notation 
$\cf{n_1(\fu)=j_1,n_2(\fu)=j_2}$ for one of its continuous approximants.

Replacing the $\delta$-function by the mollifier $\Phi_\varepsilon$ and after using the above continuous approximant for the integrand, we may perform the change of variables \eqref{eq:cofvar} using standard rules of Lebesgue measures.  This yields
\begin{equation}
\begin{split}
\int_{\R_*^{n}}\, &
\rmd^{n} \fu \,   \Phi_\varepsilon \left(\sum\limits_{k=1}^{K} \hat u_k+\sum_{k\in \mathcal{J}} u_k-1\right)\left(\sum\limits_{k=1}^{K} \hat u_k 
\omega_k
\right)^2 \cf{n_1=j_1,n_2=j_2}\\
&=
\int_{\R^n}\,
\rmd \hat u_1 \rmd \omega_1\cdots
\rmd \hat u_K \rmd \omega_K \prod_{k\in \mathcal{J}} \rmd u_k \,
 \Phi_\varepsilon \left(\sum\limits_{k=1}^{K} \hat u_k+\sum_{k\in \mathcal{J}} u_k-1\right)\\
&\qquad\times
\prod_{k=1}^{K}\left(
\cf{\hat u_k>0}\cf{|\omega_k|<1}\frac{\hat u_k}{2}\right)
\prod_{k\in \mathcal{J}}
\cf{u_k>0}
\left(\sum\limits_{k=1}^{K} \hat u_k 
\omega_k
\right)^2\cf{n_1=j_1,n_2=j_2}\\
&=\frac{1}{3}
\int_{\R_*^{K}\times \R_*^{|\mathcal{J}|}}\,
\rmd \hat u_1 \cdots
\rmd \hat u_K \prod_{k\in \mathcal{J}} \rmd u_k \,
 \Phi_\varepsilon \left(\sum\limits_{k=1}^{K} \hat u_k+\sum_{k\in \mathcal{J}} u_k-1\right)\\
&\qquad\times
\prod_{k=1}^{K}\left(
\cf{\hat u_k>0}{\hat u_k}\right)
\prod_{k\in \mathcal{J}}
\cf{u_k>0}
\sum\limits_{k=1}^{K} \hat u^2_k\cf{n_1=j_1,n_2=j_2}\,,
\end{split}
\end{equation}
where the last equality holds by Fubini's theorem, the fact that the variables $\hat u_1,\ldots,\hat u_K$ and $\omega_1,\ldots,\omega_K$ are independent, and that the argument in the function $\Phi_\varepsilon$ does not depend on $\omega_1,\ldots,\omega_K$ (see \eqref{eq:defK}).
Also, $|\mathcal{J}|$ denotes the cardinality of $\mathcal{J}$.
Now, we reintroduce the variables $\omega_1,\ldots,\omega_K$ yielding that final integral is equal to
\begin{equation}
\begin{split}
\frac{1}{3}
\int_{\R^{n}}\,& 
\rmd \hat u_1 \rmd \omega_1\cdots
\rmd \hat u_K \rmd \omega_K\prod_{k\in \mathcal{J}} \rmd u_k \,
 \Phi_\varepsilon \left(\sum\limits_{k=1}^{K} \hat u_k+\sum_{k\in \mathcal{J}} u_k-1\right)\\
&\qquad\times
\prod_{k=1}^{K}\left(
\cf{\hat u_k>0}{\hat u_k}\right)
\prod_{k=1}^{K}\left(
\cf{|\omega_k|<1}\frac{1}{2}\right)
\prod_{k\in \mathcal{J}}
\cf{u_k>0}
\sum\limits_{k=1}^{K} \hat u^2_k \cf{n_1=j_1,n_2=j_2}\,.
\end{split}
\end{equation}
Inverting the change of variables and taking $\epsilon \to 0$ followed by the limit for the continuous approximants,  we deduce \eqref{eq:deltaintegral4}.

For notational convenience we set $K=0$ when $n_1=n_2$.
By the Law of total probability 
we have
\begin{equation}\label{eq:media}
\begin{split}
\mean{\left(\sum\limits_{k=1}^{K} \hat u_k 
\omega_k
\right)^2}_{\nu(\rmd^{n}\fu)}&=
\sum\limits_{j_1=1,j_2=1}^n 
\mean{\left(\sum\limits_{k=1}^{K} \hat u_k 
\omega_k
\right)^2 \cf{n_1=j_1,n_2=j_2}}_{\nu(\rmd^{n}\fu)}\\
&=\frac{1}{3}
\sum\limits_{j_1=1,j_2=1}^n
\int_{\R_*^{n}}\,\nu(\rmd^{n}\fu)
\left(\sum\limits_{k=1}^{K} \hat u^2_k\right) \,
\cf{n_1=j_1,n_2=j_2}\\
&\leq \frac{1}{3}
\sum\limits_{j_1=1,j_2=1}^n
\int_{\R_*^{n}}\,\nu(\rmd^{n}\fu)
\max\limits_{1\leq k\leq K} \hat u_k
\left(\sum\limits_{k=1}^{K} \hat u_k\right) \,
\cf{n_1=j_1,n_2=j_2}\\
&\leq \frac{2}{3}
\sum\limits_{j_1=1,j_2=1}^n
\int_{\R_*^{n}}\,\nu(\rmd^{n}\fu)
\max\limits_{1\leq k\leq n} u_k
\left(\sum\limits_{k=1}^{K} \hat u_k\right) \,
\cf{n_1=j_1,n_2=j_2}\,.
\end{split}
\end{equation}
Recall that $r_1>0$ and $r_2>0$.
For $n_1<n_2$, \eqref{eq:Trr} reads
\[
 t_{n_1}(\fu)\leq \frac{t}{T r_1}< t_{n_1+1}(\fu)\quad \textrm{ and } \quad
 t_{n_2}(\fu)\leq \frac{t}{T r_2}<t_{n_2+1}(\fu)\,.
\]
which with the help of \eqref{eq:defK} implies
\begin{equation}
\sum\limits_{k=1}^{K} \hat u_k+u_{n_2}\cf{n_2-n_1-1\textrm{ is odd}}
=
\sum\limits_{k\in \mathcal{I}} u_k = 
t_{n_2}(\fu)-t_{n_1+1}(\fu)\leq 
\frac{t}{T r_2}-\frac{t}{T r_1}.
\end{equation}
Analogously, for $n_2<n_1$ we obtain
\begin{equation}
\sum\limits_{k\in \mathcal{I}} u_k = 
t_{n_1}(\fu)-t_{n_2+1}(\fu)\leq 
\frac{t}{T r_1}-\frac{t}{T r_2}.
\end{equation}
Combining 
the preceding inequality with \eqref{eq:media} and Fubini's theorem yields
\begin{equation}\label{eq:boundc}
\begin{split}
\mean{\left(\sum\limits_{k=1}^{K} \hat u_k 
\omega_k
\right)^2}_{\nu(\rmd^{n}\fu)}&
\leq \frac{2}{3}
\sum\limits_{j_1=1,j_2=1}^n
\int_{\R_*^{n}}\,\nu(\rmd^{n}\fu)
\max\limits_{1\leq k\leq n} u_k
\left|
\frac{t}{T r_2}-\frac{t}{T r_1}\right| \,
\cf{n_1=j_1,n_2=j_2}\\
&=\frac{2t}{3T}\left|
\frac{1}{r_2}-\frac{1}{r_1}\right|
\sum\limits_{j_1=1,j_2=1}^n
\int_{\R_*^{n}}\,\nu(\rmd^{n}\fu)
\left(\max\limits_{1\leq k\leq n} u_k\right)
\,\cf{n_1=j_1,n_2=j_2}\\
&=\frac{2t}{3T}\left|
\frac{1}{r_2}-\frac{1}{r_1}\right|
\mean{\max\limits_{1\leq k\leq n} u_k}_{\nu(\rmd^{n}\fu)}\,.
\end{split}
\end{equation}
By Lemma~\ref{lem:maxmoments} 
there exists a positive constant $C$ such that 
\[
\mean{\max\limits_{1\leq k\leq n} u_k}_{\nu(\rmd^{n}\fu)}\leq \frac{C\ln(n+1)}{n}\quad \textrm{ for all } n\in \mathbb{N}\,.
\] 
Combining the preceding inequalities in \eqref{eq:boundc} we obtain
\begin{equation}\label{eq:nuestimate}
\mean{\left(\sum\limits_{k=1}^{K} \hat u_k 
\omega_k
\right)^2}_{\nu(\rmd^{n}\fu)}\leq 
\frac{2C t}{3T}\left|
\frac{1}{r_2}-\frac{1}{r_1}\right|
\frac{\ln(n+1)}{n}\,.
\end{equation}
Inserting \eqref{eq:nuestimate} in
\eqref{eq:aimD2} we have
\begin{equation}\label{eq:all}
\begin{split}
\mathcal{R}_2&\leq 
\frac{C v^2_0L^{-2} T^2}{3}\sum_{n=1}^\infty p_n(T_*)\,
 \frac{\ln(n+1)}{n}\frac{1}{Z_{\beta,n}} 
\int_{\R_*} \nu_2(\rmd r_1)\,  (g(r_1)-1)_+ \times\\
&\qquad\qquad \int_{\R_*} \nu_2(\rmd r_2)\, r^2_2(1-g(r_2))_+    \left|
\frac{1}{r_2}-\frac{1}{r_1}\right|\,.
\end{split}
\end{equation}
Note that
\begin{equation}\label{eq:nu2cota}
\int_{\R_*} \nu_2(\rmd r_2)\, r^2_2(1-g(r_2))_+ \leq \mean{r^2}_{\nu_2}=\frac{1}{n}+1\leq 2\,.
\end{equation}
We need to estimate 
\begin{equation}
\frac{1}{Z_{\beta,n}}\int_{\R_*} \nu_2(\rmd r_1) \int_{\R_*} \nu_2(\rmd r_2)\, 
  (1-g(r_2))_+ (g(r_1)-1)_+ r_2\left|1-\frac{r_2}{r_1}\right| \,.
\end{equation}
Here, it is important that $\beta=n$.
Since $r\mapsto g(r)$ is strictly monotone increasing from $g(0)<1$ to $g(1)>1$,
there is a unique $r_*\in (0,1)$ such that $g(r_*)=1$, and then $g(r)<1$ if and only if $r<r_*$.  Since $g(r)=0$ for $r>1$, the above integrand can be non-zero
only if $0<r_2<r_*$ and $r_*<r_1\leq 1$.  In particular, then $r_2<r_1$, and we have
\[
 0<1-\frac{r_2}{r_1}\leq 1-r_2\,.
\]
Since $Z_{\beta,n} =\int_{\R_*} \nu_2(\rmd r) (g(r)-1)_+>0$ and $\beta=n$, we have $\mean{r}_{\nu_2}=1$ and 
$\mean{(1-r)^2}_{\nu_2}=\frac{1}{n}$.
Then
\begin{equation}\label{eq:nu2cota1}
\begin{split}
\frac{1}{Z_{\beta,n}}&\int_{\R_*} \nu_2(\rmd r_1) \int_{\R_*} \nu_2(\rmd r_2)\, 
  (1-g(r_2))_+ (g(r_1)-1)_+ r_2\left|1-\frac{r_2}{r_1}\right |\\
 &\leq 
\int_{\R_*} \nu_2(\rmd r_2)  (1-g(r_2))_+ r_2|1-r_2| \leq 
\mean{r^2}^{1/2}_{\nu_2}
\mean{(1-r)^2}^{1/2}_{\nu_2}
\leq  \frac{\sqrt{2}}{\sqrt{n}}\,. 
\end{split}
\end{equation}
Combining \eqref{eq:nu2cota} and \eqref{eq:nu2cota1} in 
 \eqref{eq:all} we deduce
\begin{equation}\label{e:boundd2}
\mathcal{R}_2\leq 
\frac{\sqrt{2}Cv^2_0L^{-2} T^2}{3}\sum_{n=1}^\infty p_n(T_*)\,
 \frac{\ln(n+1)}{n^{3/2}}\,.
\end{equation}
By Lemma~\ref{lem:momentpoisson} in Appendix~\ref{ap:tools} and the fact that $T_*=\lambda T$ and $L_*=|v_0|^{-1}\lambda L$  yields the existence of a constant $\widetilde{C}$ such that 
\begin{equation}\label{eq:cota}
\mathcal{R}_2\leq 
\frac{\sqrt{2}Cv^2_0L^{-2} T^2}{3}
\frac{\widetilde{C}\ln(\lambda T+3)}{(\lambda T)^{3/2}}= 3^{-1}C_2{L}^{-2}_*{T}^{1/2}_*
\ln({T}_*+3)\,,
\end{equation}
where $C_2=\sqrt{2}C\widetilde{C}$. This finishes the proof of Lemma~\ref{lem:D2term}.
\end{proof}
\noindent
\textbf{Estimates on $D^{\prime}_1$ and $D^{\prime}_2$}.
In this subsection we estimate the contribution of the incomplete jumps given in \eqref{eq:incompletejumps}.
\begin{lemma}[Incomplete jumps]\label{lem:D12term}
There is a positive constant $C_{1,2}$ such that for all $T>0$, $\lambda>0$ and $L>0$ it follows that
\begin{equation}\label{eq:D12term}
\begin{split}
\mathcal{R}_{1,2}:=v^2_0
\sum_{n=1}^\infty p_n(T_*)\,
(L^2 T)^{-1}  
\int_{\R_*^{n}}&\,(n-1)!\,\rmd^{n}\fu\, 
\delta\left(\sum_{k=1}^n u_k-1\right)\,
\int_{\R_*} \nu_2(\rmd r_1)\,  \int_{\R_*} \nu_2(\rmd r_2)\,  
\int_0^{T}
\rmd t\,  
\times\\
 &\hspace{-1cm}
(|D^\prime_1|^2+|D^\prime_2|^2)
\frac{1}{Z_{\beta,n}} (1-g(r_2))_+ (g(r_1)-1)_+\leq C_{1,2} T_* L^{-2}_* T^{-1/2}_*\,,
\end{split}
\end{equation}
where 
the probability measure $\nu_2$ is given in \eqref{eq:nu1nu2},
the function $g$ is given in \eqref{def:g} and the normalisation constant $Z_{\beta,n}$ is given in \eqref{eq:normalisation}.
\end{lemma}
\begin{proof}
We start by noticing that \eqref{eq:gammacoupling} and  \eqref{eq:defnu1nu2}
yield $0<r_1\leq 1$ and $r_2>0$.
If $n_1,n_2<n$, the remaining ``incomplete jumps'' can be estimated by
\[
 |D^{\prime}_1|\leq T r_1 u_{n_1+1}\,, \quad |D_2'|\leq T r_2 u_{n_2+1}\,. 
\]
By \eqref{eq:poissonprocess} and the fact that 
we are integrating over the measure $\delta(\sum_{k=1}^{n}{u_k}-1)$, we observe that
$n_i(t)=n$ occurs if and only if $t\ge t_n(T r_i u)=T r_i \sum_{k=1}^n u_k = T r_i$.
Then $0\leq t-t_n(T r_i u)= t-T r_i\leq T (1-r_i)$.
Therefore, in this case $|D^{\prime}_i|\leq T (1-r_i)$ which yields $0<r_2\leq 1$.
Hence, for $n_1(t)=n$ we have 
\[
 \frac{1}{L^2 T}  \int_{T r_1}^{T} \rmd t\, |D^{\prime}_1|^2\leq 
 \frac{T^2}{L^2}(1-r_1)^3\,.
\]
Similarly, for $n_2(t)=n$ we obtain 
\[
 \cf{r_2\leq 1}\frac{1}{L^2 T}  \int_{T r_2}^{T} \rmd t\, |D^{\prime}_2|^2\leq 
 \cf{r_2\leq 1}\frac{T^2}{L^2}(1-r_2)^3\,.
\]
For all other times, we estimate by splitting the integral at the points $0=t_0(T r_1 \fu)<t_1(T r_1 \fu)<\ldots<t_n(T r_1 \fu)=Tr_1$,
\[
 \frac{1}{L^2 T}  \int_0^{T r_1} \rmd t\, |D^{\prime}_1|^2=
\frac{1}{L^2 T}\sum_{k=0}^{n-1}  \int_{t_k(T r_1 \fu)}^{t_{k+1}(T r_1 \fu)} \rmd t\, |D^{\prime}_1|^2\leq \frac{1}{L^2 T}
 \sum_{k=0}^{n-1} T^3 r_1^3 u_{k+1}^3 \leq \frac{ T^2}{L^2} \sum_{k=1}^{n} u_{k}^3
  \,. 
\]
Analogously, for $r_2\leq 1$ we have
\[
 \frac{1}{L^2 T}  \int_0^{T r_2} \rmd t\, |D^{\prime}_2|^2\leq \frac{1}{L^2 T} 
 \sum_{k=0}^{n-1} T^3 r_2^3 u_{k+1}^3 \leq \frac{ T^2}{L^2} r_2^3\sum_{k=1}^{n} u_{k}^3
  \,.
\]
Finally, if $r_2>1$ and denoting $n':= N(T;T r_2 \fu)<n$, we estimate
\begin{equation}
\begin{split}
 \frac{1}{L^2 T}  \int_0^{T} \rmd t\, |D^{\prime}_2|^2&\leq \frac{1}{L^2 T} 
 \left(\sum_{k=0}^{n'-1}  \int_{t_k(Tr_2\fu)}^{t_{k+1}(Tr_2\fu)} \rmd t\, T^2 r_2^3 u_{k+1}^2
 +  \int_{t_{n'}(Tr_2\fu)}^{t_{n'+1}(Tr_2\fu)}\rmd t\,  T^2 r_2^2 u_{n'+1}^2\right)\\
& \leq \frac{T^2}{L^2} r_2^3\sum_{k=1}^{n} u_{k}^3
  \,.
\end{split}
\end{equation}
Collecting all pieces together we obtain
\[
\frac{1}{L^2 T}  \int_0^{T} \rmd t\, \left(|D^{\prime}_1|^2+|D^{\prime}_2|^2\right)
 \leq \frac{ T^2}{L^2}  \sum_{k=1}^{n} u_{k}^3\left(1+ r_2^3\right)
 + \frac{ T^2}{L^2} \left((1-r_1)^3+ (1-r_2)_+^3\right)\,. 
\]
This needs to be integrated over the measure for $\rmd^n \fu,\nu_2(\rmd r_1), \nu_2(\rmd r_2)$ in \eqref{eq:D12term}.
Since each term is non-negative and depends only on either $r_1$ or $r_2$, we use their 
known simple marginal distributions. By
Lemma~\ref{lem:moments} in Appendix~\ref{ap:tools} we have 
\[
  \mean{u_{k}^3} =  \frac{6}{n (n+1)(n+2)}\quad \textrm{ for all }\quad k\in \{1,\ldots,n\}\,.
\]
Moreover, since $\nu_2$ has Gamma distribution with parameters $n$ and $\beta$, straightforward computations yields
\[
  \mean{(1-r_1)^3}_{\nu_1} = \frac{6}{(n+1)(n+2)(n+3)}\quad \textrm{ and } \quad \mean{r_2^3}_{\nu_2} = \frac{n (n+1)(n+2)}{\beta^3}\,.
\]
By the Cauchy-Schwarz inequality we have 
$\mean{(1-r_2)^3_+}_{\nu_2}^2\leq \mean{(1-r_2)^6_+}_{\nu_2}\leq \mean{(1-r_2)^6}_{\nu_2}$.
Since $\nu_2$ is the Gamma distribution with parameters $n$ and $\beta$,
we point out that $\mean{(1-r_2)^6}_{\nu_2}$  can be evaluated explicitly, and is $O(n^{-3})$ for the choice $\beta=n$ which localises the mean $\mean{r}_{\nu_2}=1$, see Item~(5) of Lemma~\ref{lem:propgamma} in Appendix~\ref{ap:basic}.
That is, there exists a constant $C>0$ such that
\[
 \mean{(1-r_2)^3_+}_{\nu_2}
 = \int_0^1\rmd r\,\frac{1}{(n-1)!}(1-r)^3 n^n r^{n-1} \rme^{-n r} \leq \frac{C}{n^{3/2}}\,.
\]
Hence,
\begin{equation}
\begin{split}
(L^2 T)^{-1}  &
\int_{\R_*^{n}}\,(n-1)!\,\rmd^{n}\fu\, 
\delta\left(\sum_{k=1}^n u_k-1\right)\,
\int_{\R_*} \nu_2(\rmd r_1)\,  \int_{\R_*} \nu_2(\rmd r_2)\,  
\int_0^{T}
\rmd t\,  
\times\\ 
&\qquad
(|D^\prime_1|^2+|D^\prime_2|^2)
\frac{1}{Z_{\beta,n}} (1-g(r_2))_+ (g(r_1)-1)_+\\
&\hspace{-1cm}\leq 
  \frac{ T^2}{L^2} \frac{6}{(n+1)(n+2)} \left(1+ n^{-3}(n+2)(n+1)n\right)
 + \frac{ T^2}{L^2} \left(\frac{6}{(n+1)(n+2)(n+3)}+ \frac{C}{n^{3/2}}\right) .
\end{split}
\end{equation}
We stress that we have already set $\beta=n$.
Finally, we take the final missing expectation over the Poisson distribution of $n$ with parameter $T_*$  and with the help of Lemma~\ref{lem:momentpoisson} in Appendix~\ref{ap:tools} we obtain the existence of a positive constant $C_{1,2}$ such that 
\[
\mathcal{R}_{1,2}\leq 
C_{1,2} T_* L^{-2}_* T^{-1/2}_*\,.
\]
This finishes the proof of Lemma~\ref{lem:D12term}.
\end{proof}

\noindent
\textbf{Estimates on $D_1$}.
In this subsection we estimate the contribution of the term $D_1$ given in \eqref{eq:D1estimate}.
\begin{lemma}[Complete jumps]\label{lem:D1term}
There is a positive constant $C_1$ such that for all $T>0$, $\lambda>0$ and $L>0$ it follows that
\begin{equation}\label{eq:D1term}
\begin{split}
\mathcal{R}_1:=v^2_0
\sum_{n=1}^\infty p_n(T_*)\,
(L^2 T)^{-1}  
\int_{\R_*^{n}}&\,(n-1)!\,\rmd^{n}\fu\, 
\delta\left(\sum_{k=1}^n u_k-1\right)\,
\int_{\R_*} \nu_2(\rmd r_1)\,  \int_{\R_*} \nu_2(\rmd r_2)\,  
\int_0^{T}
\rmd t\,  
\times\\
\qquad\qquad &
|D_1|^2
\frac{1}{Z_{\beta,n}} (1-g(r_2))_+ (g(r_1)-1)_+\leq C_1 {L}^{-2}_*\,,
\end{split}
\end{equation}
where 
the probability measure $\nu_2$ is given in \eqref{eq:nu1nu2},
the function $g$ is given in \eqref{def:g} and the normalisation constant $Z_{\beta,n}$ is given in \eqref{eq:normalisation}.
\end{lemma}
\begin{proof}
By \eqref{eq:D1estimate} we have 
\[
D_1^2 = (r_1-r_2)^2 T^2 \sum_{k,k'=1}^{n} \cf{k,k'\leq N(t;T r_1 \fu)} (-1)^{k+k'} u_k u_{k'}\,.
\]
Note that
\begin{equation}
\begin{split}
& \frac{1}{T} \int_0^{T}\rmd t \cf{k,k'\leq N(t;T r_1 \fu)}=
 \sum_{m=1}^n \cf{k,k'\leq m}\frac{1}{T} \int_0^{T}\rmd t \cf{N(t;T r_1 \fu)=m}
\\ & \qquad\qquad
 = \sum_{m=1}^{n-1} \cf{k,k'\leq m} r_1 u_{m+1}+ 1-r_1
 = 1-r_1 +\sum_{m=2}^{n} \cf{k,k'\leq m-1} r_1 u_{m}\,.
\end{split}
\end{equation}
Therefore,
\begin{equation}
\begin{split}
& \frac{1}{T} \int_0^{T}\rmd t 
 D_1^2 
\\ & \quad 
 = (r_1-r_2)^2 T^2 r_1 \sum_{m=2}^{n} u_m \sum_{k,k'=1}^{m-1}(-1)^{k+k'} u_k u_{k'}
 +(r_1-r_2)^2 T^2 (1-r_1) \sum_{k,k'=1}^{n} (-1)^{k+k'} u_k u_{k'}.
\end{split}
\end{equation}
Taking an expectation over $(n-1)!\,\rmd^{n}\fu\, 
\delta\left(\sum_{k=1}^n u_k-1\right)$
and using Lemma~\ref{lem:moments} in Appendix~\ref{ap:tools} we obtain
\begin{equation}
\begin{split} 
&\mathcal{J}_1:=\int_{\mathbb{R}^n_*}(n-1)!\,\rmd^{n}\fu\, 
\delta\left(\sum_{k=1}^n u_k-1\right)
\frac{1}{L^2 T} \int_0^{T}\rmd t 
 D_1^2 \\
&\hspace{4cm} =
 (r_1-r_2)^2 T^2 L^{-2} r_1 \sum_{m=2}^{n}\sum_{k,k'=1}^{m-1}(-1)^{k+k'}
 \frac{ \cf{k'= k}+  1}{n(n+1)(n+2)}
\\ 
& \hspace{4cm}  \qquad 
 +(r_1-r_2)^2 T^2 L^{-2} (1-r_1) \sum_{k,k'=1}^{n} (-1)^{k+k'}
 \frac{\cf{k'=k} + 1}{n(n+1)}\,.
\end{split}
\end{equation}
Note that for any $m$ we have $\sum_{k=1}^{m-1}(-1)^{k}=\pm 1$. Thus, 
\begin{equation} 
\mathcal{J}_1\leq (r_1-r_2)^2 T^2L^{-2} r_1 
 \frac{2}{n+2}
 +(r_1-r_2)^2 T^2 L^{-2} (1-r_1)
 \frac{2}{n+1}\leq \frac{2}{n+1}
 (r_1-r_2)^2 T^2L^{-2}\,.
\end{equation}
Since
\[
 (r_1-r_2)^2 \leq 2((1-r_2)^2 +(1-r_1)^2 )\,,
\]
\[
\mean{(1-r_1)^2}_{\nu_1}=\frac{2}{(n+1)(n+2)}\quad \textrm{ and } \quad
 \mean{(1-r_2)^2}_{\nu_2}=\frac{1}{n}\,,
\]
we obtain the existence of a positive constant $C$ such that
\[
\int_{\R_*} \nu_2(\rmd r_1)\,  \int_{\R_*} \nu_2(\rmd r_2)
\frac{1}{Z_{\beta,n}} (1-g(r_2))_+ (g(r_1)-1)_+
\mathcal{J}_1\leq C\frac{T^2L^{-2}}{n^2}.
\]
Finally, we take the final missing expectation over the Poisson distribution of $n$ with parameter $T_*$  and with the help of Lemma~\ref{lem:momentpoisson} in Appendix~\ref{ap:tools} we obtain the existence of a positive constant $C_1$ such that 
\[
\mathcal{R}_{1}\leq 
C_1 L^{-2}_* \,.
\]
This finishes the proof of Lemma~\ref{lem:D1term}.
\end{proof}

In the sequel, we stress the fact that Proposition~\ref{prop:coinflipcoupling} is just a consequence of what
we have already proved up to here.
\begin{proof}[Proof of Proposition~\ref{prop:coinflipcoupling}]
Combining Lemma~\ref{lem:D2term}, Lemma~\ref{lem:D12term} and Lemma~\ref{lem:D1term} in \eqref{eq:couplingflip} implies \eqref{eq:cflipine}. This finishes the proof of
Proposition~\ref{prop:coinflipcoupling}.
\end{proof}

\section{\textbf{Proof of Proposition~\ref{prop:naturalcoupling}: Synchronous coupling}}
\label{sec:naturalcoupling}

In this section we prove of Proposition~\ref{prop:naturalcoupling}. 
Recall that $\mathfrak{u}=(u_j:j\in \mathbb{N})$ is a sequence of i.i.d.\ random variables with exponential distribution of parameter one
defined in the probability space $(\Omega_1,\mathcal{F}_1,\mathbb{P}_1)$.
For each (deterministic) $n\in \mathbb{N}$,
\eqref{def:Ytrunk} defines 
\[
Y(t;n,\fd^{(n)})=L^{-1}v_0\sum\limits_{k=1}^{M}(-1)^{k-1}\delta^{(n)}_k
+L^{-1}v_0(-1)^M(t-t_M(\fd^{(n)}))\,,
\quad \textrm{ with }\quad M:=N(t;\fd^{(n)})
\]
for any $t\geq 0$,
where $\fu^{(n)}:=(u^{(n)}_k:k\in \mathbb{N})$, $\mathfrak{\delta}^{(n)}:=(\delta^{(n)}_k:k\in \mathbb{N})$
and $t_M(\fd^{(n)})$ are defined in \eqref{eq:utruncado}, \eqref{eq:skn} and \eqref{eq:tmn}, respectively.
Moreover,
for $n\in \mathbb{N}\setminus\{1\}$ we define the C\`adl\`ag version of the random walk  $(\eta^*_k(\fs):k\in \mathbb{N}_0)$ given in \eqref{def:randomwalk}. That is,
\[
Z(t;n,\fs)=\sum\limits_{k=1}^{\widetilde{n}}\cf{t\geq  t_{2k}(\fs)}\eta^*_k(\fs) \quad \textrm{ for any }\quad t\geq 0\,,
\]
where $\widetilde{n}=\lfloor n/2 \rfloor$
and $\fs=(s_1,s_2,\ldots,)=(1/\lambda)(u_1,u_2,\ldots,)$
is distributed according to the infinitely product measure of exponential distributions with parameter $\lambda>0$.

We define the auxiliary process
\[
\widetilde{Z}(t;n,
\fs)=\frac{T_* }{n}\sum\limits_{k=1}^{\widetilde{n}}\cf{t\geq t_{2k}(\mathfrak{w}^{(n)}(\fs))}\eta^*_k(\fs) \quad \textrm{ for any }\quad t\geq 0\,,
\]
where the sequence $(w^{(n)}_j(\fs):j\in \mathbb{N})$ is defined by setting
\[
w^{(n)}_j(\fs)=\frac{T}{n} u^{(n)}_j=
\frac{T_*}{n} s^{(n)}_j
\quad \textrm{ for any }\quad j\in \mathbb{N}\,.
\]
For notational convenience we write $(w^{(n)}_j:j\in \mathbb{N})$ instead of 
$(w^{(n)}_j(\fs):j\in \mathbb{N})$.
For $n\in \{0,1\}$ we set $Z(t;n,\fs)=\widetilde{Z}(t;n,\fs)=0$ for all $t\geq 0$.
Analogously to \eqref{def:KU} and \eqref{def:KUZ} we define 
$\widetilde{Z}:=(\widetilde{Z}(t;K,\fs):t\geq 0)$, where $K$ is a random variable with Poisson distribution with parameter $T_*=\lambda T$   
defined in the probability space 
$(\Omega_2,\mathcal{F}_2,\mathbb{P}_2)$ and independent of $\fu$. 
Then we let
$(\Omega,\mathcal{F},\mathbb{Q}):=(\Omega_1\times \Omega_2,\mathcal{F}_1\otimes\mathcal{F}_2,\mathbb{P}_1\otimes\mathbb{P}_2)$ be the product probability space of 
$(\Omega_1,\mathcal{F}_1,\mathbb{P}_1)$ and $(\Omega_2,\mathcal{F}_2,\mathbb{P}_2)$.

We compare the processes $Y$, $Z$ and $\widetilde{Z}$ as follows: 
\begin{equation}\label{eq:cYZtildeZ}
\mathcal{W}_2(\mathbb{Y}_{[0,T]},\mathbb{Z}_{[0,T]})\leq 
\mathcal{W}_2(\mathbb{Y}_{[0,T]},\mathbb{\widetilde{Z}}_{[0,T]})+\mathcal{W}_2(\mathbb{\widetilde{Z}}_{[0,T]},\mathbb{Z}_{[0,T]})\,.
\end{equation}
The following lemma provides an estimate for the second term of the right-hand side of \eqref{eq:cYZtildeZ}.
\begin{lemma}\label{lem:compZZ*}
For any $L>0$, $T>0$, $\lambda>0$ and $v_0\in \mathbb{R}\setminus\{0\}$ it follows that
\begin{equation}
\mathcal{W}^2_2(\mathbb{\widetilde{Z}}_{[0,T]},\mathbb{Z}_{[0,T]})\leq 
28{T_*}{L^{-2}_*}
T^{-1/2}_*(1+T^{-1/2}_*)\,.
\end{equation}
In particular, for $T_*\geq 1$ it follows that
\begin{equation}
\mathcal{W}^2_2(\mathbb{\widetilde{Z}}_{[0,T]},\mathbb{Z}_{[0,T]})\leq 
56{T_*}{L^{-2}_*}
T^{-1/2}_*\,.
\end{equation}
\end{lemma}

\begin{proof}
By definition \eqref{def:metric}, the tower property (Item~iv) of Theorem~8.14 in \cite{Klenke2014}) and Fubini's theorem 
 we have
\begin{equation}\label{eq:desifb}
\begin{split}
\mathcal{W}^2_2(\mathbb{Z}_{[0,T]},\mathbb{\widetilde{Z}}_{[0,T]})&\leq \frac{1}{T}
\sum\limits_{n=2}^{\infty} p_n(T_*)
\int_{0}^{T}\mathbb{E}_{\mathbb{Q}}\left[|Z(t;n,\fs)-\widetilde{Z}(t;n,\mathfrak{w}^{(n)})|^2 | K=n
\right]\rmd t\\
&=\frac{1}{T}
\sum\limits_{n=2}^{\infty} p_n(T_*)
\int_{0}^{T}\mathbb{E}_{\mathbb{P}_1}\left[|Z(t;n,\fs)-\widetilde{Z}(t;n,\mathfrak{w}^{(n)})|^2
\right]\rmd t\\
&= \frac{1}{T}
\sum\limits_{n=2}^{\infty} p_n(T_*)
\mathbb{E}_{\mathbb{P}_1}\left[\int_{0}^{T}|Z(t;n,\fs)-\widetilde{Z}(t;n,\mathfrak{w}^{(n)})|^2
\rmd t\right]\,.
\end{split}
\end{equation}
In the sequel, we estimate 
\begin{equation}
\begin{split}
\frac{1}{T}
\int_{0}^{T} \rmd t |Z(t;n,\fs)-\widetilde{Z}(t;n,\mathfrak{w}^{(n)})|^2.
\end{split}
\end{equation}
Since $|x+y|^2\leq 2x^2+2y^2$ for any $x,y\in \mathbb{R}$, we have
\begin{equation}\label{eq:inty}
\begin{split}
|Z(t;n,\fs)-\widetilde{Z}(t;n,\mathfrak{w}^{(n)})|^2&\leq 
2\big|\sum\limits_{k=1}^{\widetilde{n}}\eta^*_k(\fs)\left(\cf{t\geq t_{2k}(\fs)}-\cf{t\geq t_{2k}(\mathfrak{w}^{(n)})}\right)\big|^2\\
&\qquad+2\Big(1-\frac{T_*}{n}\Big)^2\big|\sum\limits_{k=1}^{\widetilde{n}}\eta^*_k(\fs)\cf{t\geq t_{2k}(\mathfrak{w}^{(n)})}\big|^2.
\end{split}
\end{equation}

We start with the estimate of second term of the right-hand side of the preceding inequality.
For any $\ell\in \{1,\ldots,\widetilde{n}\}$ set
\begin{equation}\label{eq:ch1}
\hat{s}_\ell :=s_{2\ell}+s_{2\ell-1}\quad \textrm{ and }\quad
\omega_\ell :=\frac{s_{2\ell}-s_{2\ell-1}}{\hat{s}_\ell}\,
\end{equation}
and observe that
\begin{equation}\label{eq:ch2}
\eta^*_{\ell}(\fs)=-\frac{v_0}{L} \hat{s}_\ell\omega_\ell\quad
\textrm{ for any } \quad\ell\in \{1,\ldots,\widetilde{n}\}\,.
\end{equation}
We stress that
for any $k\in \mathbb{N}$, $2k$ is an even number and, therefore, it is possible to apply pairwise
Lemma~\ref{lem:independencia} in Appendix~\ref{ap:basic}. We find that $t_{2k}(\mathfrak{w}^{(n)})$ is independent of all $\omega_{\ell}$, $\ell\in \{1,\ldots, \widetilde{n}\}$. 
Since
\begin{equation}
\begin{split}
\big|\sum\limits_{k=1}^{\widetilde{n}}\eta^*_k(\fs)\cf{t\geq t_{2k}(\mathfrak{w}^{(n)})}\big|^2=
\sum\limits_{k=1}^{\widetilde{n}}
\sum\limits_{k'=1}^{\widetilde{n}}\eta^*_k(\fs)\eta^*_{k'}(\fs)\cf{t\geq t_{2k}(\mathfrak{w}^{(n)})}\cf{t\geq t_{2k'}(\mathfrak{w}^{(n)})}\,,
\end{split}
\end{equation}
Item~i) and Item~iii) of Lemma~\ref{lem:independencia} in Appendix~\ref{ap:basic} imply
 \begin{equation}
\begin{split}
&\left(1-\frac{T_*}{n}\right)^2
\mean{\big|\sum\limits_{k=1}^{\widetilde{n}}\eta^*_k(\fs)\cf{t\geq t_{2k}(\mathfrak{w}^{(n)})}\big|^2}_{\omega}\\
&\hspace{4cm}=
\left(1-\frac{T_*}{n}\right)^2\sum\limits_{k=1}^{\widetilde{n}}
\sum\limits_{k'=1}^{\widetilde{n}}
\mean{\eta^*_k(\fs)\eta^*_{k'}(\fs)}_{\omega}\cf{t\geq t_{2k}(\mathfrak{w}^{(n)})}\cf{t\geq t_{2k'}(\mathfrak{w}^{(n)})}\\
&\hspace{4cm}=\left(1-\frac{T_*}{n}\right)^2\sum\limits_{k=1}^{\widetilde{n}}
\mean{|\eta^*_k(\fs)|^2}_{\omega}\cf{t\geq t_{2k}(\mathfrak{w}^{(n)})}\\
&\hspace{4cm}\leq \left(1-\frac{T_*}{n}\right)^2
\mean{|\eta^*_1(\fs)|^2}_{\omega}\widetilde{n}\,,
\end{split}
\end{equation}
where $\omega$ denotes the product measure given by the law of the random vector $(\omega_1,\ldots,\omega_{\widetilde{n}})$.
Combining the above inequality with Fubini's theorem, 
\eqref{eq:ch2}, Lemma~\ref{lem:independencia} in Appendix~\ref{ap:basic}, and Item~(4) of Lemma~\ref{lem:propgamma} in Appendix~\ref{ap:basic}
we obtain
 \begin{equation}
\begin{split}
\left(1-\frac{T_*}{n}\right)^2\frac{1}{T}\int_{0}^{T} \rmd t &
\mean{\big|\sum\limits_{k=1}^{\widetilde{n}}\eta^*_k(\fs)\cf{t\geq t_{2k}(\mathfrak{w}^{(n)})}\big|^2}_{\hat{s},\omega}\leq 
\left(1-\frac{T_*}{n}\right)^2
\mean{|\eta^*_1(\fs)|^2}_{\hat{s},\omega}\widetilde{n}\\
&
\leq 
\left(1-\frac{T_*}{n}\right)^2
\frac{v^2_0}{L^2\lambda^2}n=
\frac{T_*}{L^2_*}\left(1-\frac{T_*}{n}\right)^2
\frac{n}{T_*}\,,
\end{split}
\end{equation}
where 
$\hat{s}$ denote the product measure given by the law of the random vector $(\hat{s}_1,\ldots,\hat{s}_{\widetilde{n}})$.
We observe that 
\[
\left(1-\frac{T_*}{n}\right)^2
\frac{n}{T_*}=\left(\frac{n}{T_*}-1\right)+\left(\frac{T_*}{n}-1\right)\leq 
\frac{1}{T_*}\left|n-T_*\right|+\frac{1}{n}\left|T_*-n\right|\,.
\]
Then by the Cauchy Schwarz inequality, Lemma~\ref{lem:momentpoisson} in Appendix~\ref{ap:tools} and 
Lemma~\ref{lem:mopoisfull} in Appendix~\ref{ap:basic}  we deduce
\begin{equation}
\begin{split}
\mean{\Big(1-\frac{T_*}{n}\Big)^2
\frac{n}{T_*}\cf{n\geq 2}}_{\textsf{Po}(T_*)}&\leq 
\frac{1}{T_*}\mean{|n-T_*|\cf{n\geq 2}}_{\textsf{Po}(T_*)}+
\mean{|n-T_*|\frac{1}{n}\cf{n\geq 2}}_{\textsf{Po}(T_*)}
\\
&\leq \frac{1}{T_*}\mean{\left(n-T_*\right)^2}_{\textsf{Po}(T_*)}^{1/2}+
\mean{\left(n-T_*\right)^2}_{\textsf{Po}(T_*)}^{1/2}
\mean{n^{-2}\cf{n\geq 1}}_{\textsf{Po}(T_*)}^{1/2}\\
&\leq \frac{T^{1/2}_*}{T_*}+
{T^{1/2}_*}\sqrt{\frac{16}{T^2_*}} 
=5T^{-1/2}_*\,.
\end{split}
\end{equation}
Collecting all pieces together we obtain
\begin{equation}\label{eq:uno}
\begin{split}
&\mean{\Big(1-\frac{T_*}{n}\Big)^2\frac{1}{T}\int_{0}^{T} \rmd t
\mean{\big|\sum\limits_{k=1}^{\widetilde{n}}\eta^*_k(\fs)\cf{t\geq t_{2k}(\mathfrak{w}^{(n)})}\big|^2}_{\hat{s},\omega}\cf{n\geq 2}}_{\textsf{Po}(T_*)}\leq 5\frac{T_*}{L^2_*}T^{-1/2}_*\,.
\end{split}
\end{equation}

We continue with the estimate of
\[
\frac{1}{T}\int_{0}^{T}\rmd t \big|\sum\limits_{k=1}^{\widetilde{n}}\eta^*_k(\fs)\left(\cf{t\geq t_{2k}(\fs)}-\cf{t\geq t_{2k}(\mathfrak{w}^{n})}\right)\big|^2\,.
\]
By \eqref{eq:ch1}, \eqref{eq:ch2}, Item~i) and Item~iii) of Lemma~\ref{lem:independencia} in Appendix~\ref{ap:basic} we have 
\begin{equation}\label{eq:cita1}
\begin{split}
&\mean{\frac{1}{T}\int_{0}^{T}\rmd t \big|\sum\limits_{k=1}^{\widetilde{n}}\eta^*_k(\fs)\left(\cf{t\geq t_{2k}(\fs)}-\cf{t\geq t_{2k}(\mathfrak{w}^{(n)})}\right)\big|^2}_{\omega}\\
&\hspace{1cm}
=\left(\frac{v_0}{L}\right)^2\sum\limits_{\ell=1}^{\widetilde{n}}\sum\limits_{\ell^\prime=1}^{\widetilde{n}} \hat{s}_\ell \hat{s}_{\ell^\prime}\mean{\omega_\ell\omega_{\ell^\prime}}_{\omega}\times\\
&\hspace{4cm}
\frac{1}{T}\int_{0}^{T}\rmd t
\left(\cf{t>t_{2\ell}(\fs)}-\cf{t>t_{2\ell}({\mathfrak{w}^{(n)}})}\right)\left(\cf{t>t_{2\ell^\prime}(\fs)}-\cf{t>t_{2\ell^\prime}({\mathfrak{w}^{(n)}})}\right)\\
&\hspace{1cm}
=\frac{1}{3}\left(\frac{v_0}{L}\right)^2\sum\limits_{\ell=1}^{\widetilde{n}} \hat{s}^2_\ell
\frac{1}{T}\int_{0}^{T}\rmd t
\left(\cf{t>t_{2\ell}(\fs)}-\cf{t>t_{2\ell}({\mathfrak{w}^{(n)}})}\right)^2\\
&\hspace{1cm}
=\frac{1}{3}\left(\frac{v_0}{L}\right)^2\sum\limits_{\ell=1}^{\widetilde{n}} \hat{s}^2_\ell
\frac{1}{T}\int_{0}^{T}\rmd t
\cf{\min\{t_{2\ell}(\fs),t_{2\ell}({\mathfrak{w}^{(n)}})\}<t<\max\{t_{2\ell}(\fs),t_{2\ell}({\mathfrak{w}^{(n)}})\}}\\
&\hspace{1cm}
\leq \frac{1}{3}\left(\frac{v_0}{L}\right)^2\sum\limits_{\ell=1}^{\widetilde{n}} \hat{s}^2_\ell
\frac{1}{T}\left(
\max\left\{t_{2\ell}(\fs),t_{2\ell}({\mathfrak{w}^{(n)}})\right\}-\min\left\{t_{2\ell}(\fs),t_{2\ell}({\mathfrak{w}^{(n)}})\right\}\right)\\
&\hspace{1cm}
\leq \left(\frac{v_0}{L}\right)^2\sum\limits_{\ell=1}^{\widetilde{n}} \hat{s}^2_\ell \frac{1}{T}
\left|t_{2\ell}(\fs)-t_{2\ell}({\mathfrak{w}^{(n)}})\right|\,.
\end{split}
\end{equation} 
We note that $\hat{s}_1$ has Gamma distribution with parameters $2$ and $\lambda$ and then Lemma~\ref{lem:propgamma} in Appendix~\ref{ap:basic} yields $\mean{\hat{s}^4_1}_{\hat{s}}=5!/\lambda^4$. 
Then we integrate with respect to $\hat{s}$ and using the Cauchy Schwarz inequality we obtain
\begin{equation}
\begin{split}
\mean{
\Big(\frac{v_0}{L}\Big)^2\sum\limits_{\ell=1}^{\widetilde{n}} \hat{s}^2_\ell \frac{1}{T}
\left|t_{2\ell}(\fs)-t_{2\ell}({\mathfrak{w}^{(n)}})\right|}_{\hat{s}}&\leq \frac{v^2_0}{L^2T}\sum\limits_{\ell=1}^{\widetilde{n}} \mean{\hat{s}^4_\ell}^{1/2}_{\hat{s}}
\mean{\left|t_{2\ell}(\fs)-t_{2\ell}({\mathfrak{w}^{(n)}})\right|^2}^{1/2}_{\hat{s}}\\
&=\frac{v^2_0}{L^2T}\mean{\hat{s}^4_1}^{1/2}_{\hat{s}}\sum\limits_{\ell=1}^{\widetilde{n}} 
\mean{\left|t_{2\ell}(\fs)-t_{2\ell}({\mathfrak{w}^{(n)}})\right|^2}^{1/2}_{\hat{s}}\\
&=\frac{v^2_0}{L^2T}\frac{\sqrt{120}}{\lambda^2}\sum\limits_{\ell=1}^{\widetilde{n}} 
\mean{\left|t_{2\ell}(\fs)-t_{2\ell}({\mathfrak{w}^{(n)}})\right|^2}^{1/2}_{\hat{s}}\\
&=\frac{\sqrt{120}}{L^2_*}\frac{1}{T}\sum\limits_{\ell=1}^{\widetilde{n}} 
\mean{\left|t_{2\ell}(\fs)-t_{2\ell}({\mathfrak{w}^{(n)}})\right|^2}^{1/2}_{\hat{s}}\,.
\end{split}
\end{equation}
By Item~ii) of Lemma~\ref{lem:independencia} in Appendix~\ref{ap:basic} we 
have that each $\hat{s}_j$ has Gamma distribution with parameters $2$ and $\lambda$. 
Since $\hat{s}_1,\ldots,\hat{s}_{\widetilde{n}}$ are i.i.d.\ random variables,
we obtain that
 $\sum\limits_{j=1}^{\ell} \hat{s}_j=\sum\limits_{j=1}^{2\ell} s_j$ has Gamma distribution with parameters  $2\ell$ and $\lambda$. Hence, Item~(5) of Lemma~\ref{lem:propgamma} in Appendix~\ref{ap:basic} yields
\[
\mean{\big|\sum\limits_{j=1}^{2\ell} s_j\big|^2}_{\hat s}=(2\ell+1)(2\ell)/\lambda^2\,.
\]
We then obtain
\begin{equation}
\begin{split}
\mean{\left|t_{2\ell}(\fs)-t_{2\ell}({\mathfrak{w}^{(n)}})\right|^2}_{\hat{s}}&=\Big(1-\frac{T_*}{n}\Big)^2
\mean{\big|\sum\limits_{j=1}^{2\ell} s_j\big|^2
}_{\hat{s}}=\Big(1-\frac{T_*}{n} \Big)^2 \frac{(2\ell+1)(2\ell)}{\lambda^2}\,.
\end{split}
\end{equation}
Therefore,
\begin{equation}
\begin{split}
\mean{
\Big(\frac{v_0}{L}\Big)^2\sum\limits_{\ell=1}^{\widetilde{n}} \hat{s}^2_\ell \frac{1}{T}
\big|t_{2\ell}(\fs)-t_{2\ell}({\mathfrak{w}^{(n)}})\big|}_{\hat{s}}&\leq \frac{\sqrt{120}}{L^2_*}\frac{1}{T}
\big|1-\frac{T_*}{n} \big|\sum\limits_{\ell=1}^{\widetilde{n}}\frac{\sqrt{6}\ell}{\lambda}\leq \frac{T_*}{L^2_*}
\big|1-\frac{T_*}{n} \big|\frac{n^{2}}{T^2_*}\,.
\end{split}
\end{equation}
Now, we integrate the right-hand side of the preceding inequality over $n$ for $n\geq 2$. With the help of the Cauchy-Schwarz inequality and Lemma~\ref{lem:mopoisfull} in Appendix~\ref{ap:basic},
we obtain
\begin{equation}
\begin{split}
\mean{9
\frac{T_*}{L^2_*}
\big|1-\frac{T_*}{n} \big|\frac{n^{2}}{T^2_*}\cf{n\geq 2}}_{\textsf{Po}(T_*)}&\leq 9\frac{T_*}{L^2_*}\frac{1}{T^2_*}
\mean{
\big|n-T_* \big|^2}^{1/2}_{\textsf{Po}(T_*)}
\mean{n^2\cf{n\geq 2}}^{1/2}_{\textsf{Po}(T_*)}\\
&\leq 9\frac{T_*}{L^2_*}\frac{1}{T^2_*}T^{1/2}_*(T^2_*+T_*)^{1/2}\leq 9\frac{T_*}{L^2_*}T^{-1/2}_*\left(1+T^{-1/2}_*\right)\,,
\end{split}
\end{equation}
where in the last step we used subadditivity inequality for $x\mapsto x^{1/2}$.
Combining all pieces together we deduce
\begin{equation}\label{eq:dos}
\begin{split}
&\sum_{n\geq 2} p_n(T_*)
\mean{\frac{1}{T}\int_{0}^{T}\rmd t \big|\sum\limits_{k=1}^{\widetilde{n}}\eta^*_k(\fs)\big(\cf{t\geq t_{2k}(\fs)}-\cf{t\geq t_{2k}(\mathfrak{w}^{(n)})}\big)\big|^2}_{\hat{s},\omega}\\
&\hspace{9cm}\leq 
9\frac{T_*}{L^2_*}T^{-1/2}_*\left(1+T^{-1/2}_*\right)\,.
\end{split}
\end{equation}
Now, by \eqref{eq:inty}, \eqref{eq:uno} and \eqref{eq:dos} we obtain
\begin{equation}
\sum_{n\geq 2} p_n(T_*)\frac{1}{T}
\int_{0}^{T}\mathbb{E}_{\mathbb{P}_1}[|Z(t;n,\fs)-\widetilde{Z}(t;n,\mathfrak{w}^{(n)})|^2] \rmd t\leq 28\frac{T_*}{L^2_*}
T^{-1/2}_*\left(1+T^{-1/2}_*\right)\,.
\end{equation}
This concludes the proof of Lemma~\ref{lem:compZZ*}.
\end{proof}
The following lemma provides an estimate for the first term of the right-hand side of \eqref{eq:cYZtildeZ}.
\begin{lemma}\label{lem:YZ*}
For any $L>0$, $T>0$, $\lambda>0$ and $v_0\in \mathbb{R}\setminus\{0\}$ it follows that

\[
\mathcal{W}^2_2(\mathbb{Y}_{[0,T]},\mathbb{\widetilde{Z}}_{[0,T]})
\leq 39 {T_*}{L^{-2}_*}T^{-1/2}_*+
38T_*{L^{-2}_*}e^{-T_*}(T_*+T^2_*).
\]
In addition, for $T_*>0$ we obtain
\[
\mathcal{W}^2_2(\mathbb{Y}_{[0,T]},\mathbb{\widetilde{Z}}_{[0,T]})
\leq 115{T_*}{L^{-2}_*}T^{-1/2}_*\,.
\]
\end{lemma}
\begin{proof}
Analogously to \eqref{eq:desifb} we obtain
\begin{equation}\label{eq:logo}
\begin{split}
\mathcal{W}^2_2(\mathbb{Y}_{[0,T]},\mathbb{\widetilde{Z}}_{[0,T]})&\leq \frac{1}{T}
\sum\limits_{n=2}^{\infty} p_n(T_*)
\mathbb{E}_{\mathbb{P}_1}\left[\int_{0}^{T}|Y(t;n,\fs)-\widetilde{Z}(t;n,\mathfrak{w}^{(n)})|^2
\rmd t\right]\\
&\qquad+\frac{1}{T}p_0(T_*)\int_{0}^T L^{-2}v^2_0 t^2 \rmd t+\frac{1}{T}p_1(T_*)\int_{0}^T \mathbb{E}_{\mathbb{P}_1}[|Y(s;1,\fs)|^2] \rmd s\,.
\end{split}
\end{equation}

We start with the following claim:
\begin{equation}\label{eq:claimYZ}
Y(t_{2k}(\mathfrak{w}^{(n)});n,\mathfrak{w}^{(n)})=\widetilde{Z}(t_{2k}(\mathfrak{w}^{(n)});n,\mathfrak{w}^{(n)})\,
\end{equation}
for all $k\in \mathbb{N}$ and $n\in \mathbb{N}\setminus\{0,1\}$. Indeed, note that
\begin{equation}
\begin{split}
Y(t_{2k}(\mathfrak{w}^{(n)});n,\mathfrak{w}^{(n)})&=L^{-1}v_0\sum\limits_{j=1}^{2k}(-1)^{j-1}w^{(n)}_j=
L^{-1}v_0\sum\limits_{j=1}^{2k}(-1)^{j-1}\frac{T}{n} u^{(n)}_j\\
&=L^{-1}v_0\sum\limits_{j=1}^{2k}(-1)^{j-1}\frac{T\lambda}{n} s_j\\
&=\frac{T_*}{n}\sum_{j=1}^{k}\eta^*_j(\fs)=\widetilde{Z}(t_{2k}(\mathfrak{w}^{(n)});n,\mathfrak{w}^{(n)})\,,
\end{split}
\end{equation}
which yields the claim. Bearing the preceding claim in mind, 
we consider the following split that overestimate our desired estimate: for all $n\in \mathbb{N}\setminus\{0,1\}$
\begin{equation}\label{eq:splitover}
\begin{split}
\frac{1}{T}
&
\int_{0}^{T}\mathbb{E}_{\mathbb{P}_1}[|Y(t;n,\mathfrak{w}^{(n)})-\widetilde{Z}(t;n,\mathfrak{w}^{(n)})|^2] \rmd t\\
&\leq \mathbb{E}_{\mathbb{P}_1}\left[\frac{1}{T}\sum_{\ell=1}^{\widetilde{n}}
\int_{t_{2\ell-2}(\mathfrak{w}^{(n)})}^{t_{2\ell-1}(\mathfrak{w}^{(n)})}|Y(t;n,\mathfrak{w}^{(n)})-\widetilde{Z}(t;n,\mathfrak{w}^{(n)})|^2 \rmd t\right]\\
&\qquad+
\mathbb{E}_{\mathbb{P}_1}\left[\frac{1}{T}\sum_{\ell=1}^{\widetilde{n}}
\int_{t_{2\ell-1}(\mathfrak{w}^{(n)})}^{t_{2\ell}(\mathfrak{w}^{(n)})}|Y(t;n,\mathfrak{w}^{(n)})-\widetilde{Z}(t;n,\mathfrak{w}^{(n)})|^2 \rmd t\right]\\
&\qquad+\mathbb{E}_{\mathbb{P}_1}\left[
\frac{1}{T}\cf{t_{2\widetilde{n}}(\mathfrak{w}^{(n)})\leq T}\int_{t_{2\widetilde{n}}(\mathfrak{w})}^{T} |Y(t;n,\mathfrak{w}^{(n)})-\widetilde{Z}(t;n,\mathfrak{w}^{(n)})|^2  \rmd t
\right].
\end{split}
\end{equation}
 where $\widetilde{n}=\lfloor n/2 \rfloor$. 
For each $\ell\in \{1,\ldots,\widetilde{n}\}$ 
 observe that
\begin{equation}
\begin{split}
\mathbb{E}_{\mathbb{P}_1}&\left[\frac{1}{T}
\int_{t_{2\ell-2}(\mathfrak{w}^{(n)})}^{t_{2\ell-1}(\mathfrak{w}^{(n)})}|Y(t;n,\mathfrak{w}^{(n)})-\widetilde{Z}(t;n,\mathfrak{w}^{(n)})|^2 \rmd t\right]\\
&=\mathbb{E}_{\mathbb{P}_1}\left[\frac{1}{T}
\int_{0}^{t_{1}(\mathfrak{w}^{(n)})}|Y(t;n,\mathfrak{w}^{(n)})|^2 \rmd t\right]\\
&=\mathbb{E}_{\mathbb{P}_1}\left[\frac{1}{T}
\int_{0}^{w^{(n)}_1}L^{-2}v^2_0t^2 \rmd t\right]=T^{-1}\frac{1}{3}L^{-2}v^2_0 \mean{(w^{(n)}_1)^3}\\
&=T^{-1}\frac{1}{3}L^{-2}v^2_0 \frac{T^3}{n^3}\mean{u^3_1}=
\frac{1}{2}\frac{T_*}{L^2_*} \frac{T_*}{n^3}\,,
\end{split}
\end{equation}
due to $w^{(n)}_1:=\frac{T}{n} u^{}_1$ and $u_1$ having an exponential distribution with parameter one.
Analogously, 
\begin{equation}
\begin{split}
\mathbb{E}_{\mathbb{P}_1}&\left[\frac{1}{T}
\int_{t_{2\ell-1}(\mathfrak{w}^{(n)})}^{t_{2\ell}(\mathfrak{w}^{(n)})}|Y(t;n,\mathfrak{w}^{(n)})-\widetilde{Z}(t;n,\mathfrak{w})|^2 \rmd t\right]\\
&=\mathbb{E}_{\mathbb{P}_1}\left[\frac{1}{T}
\int_{t_{1}(\mathfrak{w}^{(n)})}^{t_{2}(\mathfrak{w}^{(n)})}|L^{-1}v_0w^{(n)}_1+L^{-1}v_0(-1)(t-w^{(n)}_1)|^2 \rmd t\right]\\
&=\mathbb{E}_{\mathbb{P}_1}\left[
T^{-1}\frac{1}{3}L^{-2}v^2_0\left((w^{(n)}_2-w^{(n)}_1)^3+(w^{(n)}_1)^3
\right)
\right]\\
&=T^{-1}\frac{1}{3}L^{-2}v^2_0 \mean{(w^{(n)}_1)^3}=
\frac{1}{2}\frac{T_*}{L^2_*} \frac{T_*}{n^3}\,.
\end{split}
\end{equation}
Hence, for $\cf{t_{2\widetilde{n}}(\mathfrak{w}^{(n)})> T}$ 
\begin{equation}\label{eq:nojump}
\begin{split}
\frac{1}{T}
&
\int_{0}^{T}\mathbb{E}_{\mathbb{P}_1}[|Y(t;n,\mathfrak{w}^{(n)})-\widetilde{Z}(t;n,\mathfrak{w}^{(n)})|^2] \rmd t
\leq 2\widetilde{n}\frac{1}{2}\frac{T_*}{L^2_*} \frac{T_*}{n^3}\leq \frac{1}{2}\frac{T_*}{L^2_*} \frac{T_*}{n^2}\,.
\end{split}
\end{equation}
In what follows, we estimate the possible remainder error given by the last jump, that is,
\begin{equation}
\begin{split}
\mathbb{E}_{\mathbb{P}_1}&\left[
\frac{1}{T}\cf{t_{2\widetilde{n}}(\mathfrak{w}^{(n)})\leq T}\int_{t_{2\widetilde{n}}(\mathfrak{w}^{(n)})}^{T} |Y(t;n,\mathfrak{w}^{(n)})-\widetilde{Z}(t;n,\mathfrak{w}^{(n)})|^2  \rmd t
\right]\\
&
=
\mathbb{E}_{\mathbb{P}_1}\left[
\frac{1}{T}\cf{t_{2\widetilde{n}}(\mathfrak{w}^{(n)})\leq T}\int_{t_{2\widetilde{n}}(\mathfrak{w}^{(n)})}^{T} |Y(t;n,\mathfrak{w}^{(n)})-\widetilde{Z}(t_{2\widetilde{n}}(\mathfrak{w}^{(n)});n,\mathfrak{w}^{(n)})|^2  \rmd t
\right]\\
&
=
\mathbb{E}_{\mathbb{P}_1}\left[
\frac{1}{T}\cf{t_{2\widetilde{n}}(\mathfrak{w}^{(n)})\leq T}\int_{t_{2\widetilde{n}}(\mathfrak{w}^{(n)})}^{T} |Y(t;n,\mathfrak{w}^{(n)})-Y(t_{2\widetilde{n}}(\mathfrak{w}^{(n)});n,\mathfrak{w}^{(n)})|^2  \rmd t
\right]\,,
\end{split}
\end{equation}
where the last equality follows by \eqref{eq:claimYZ}.

First, we assume that $n$ is an even number. We note that $2\widetilde{n}=n$ and then
\begin{equation}\label{eq:evencase0}
\begin{split}
\mathbb{E}_{\mathbb{P}_1}&\left[
\frac{1}{T}\cf{t_{2\widetilde{n}}(\mathfrak{w}^{(n)})\leq T}\int_{t_{2\widetilde{n}}(\mathfrak{w}^{(n)})}^{T} |Y(t;n,\mathfrak{w}^{(n)})-Y(t_{2\widetilde{n}}(\mathfrak{w}^{(n)});n,\mathfrak{w}^{(n)})|^2  \rmd t
\right]\\
&=\mathbb{E}_{\mathbb{P}_1}\left[
\frac{1}{T}\cf{t_{2\widetilde{n}}(\mathfrak{w}^{(n)})\leq T}
\int_{t_{2\widetilde{n}}(\mathfrak{w}^{(n)})}^{T} L^{-2}v^2_0|t-t_{2\widetilde{n}}(\mathfrak{w}^{(n)})|^2  \rmd t
\right]\\
&=\mathbb{E}_{\mathbb{P}_1}\left[
\frac{1}{T}\cf{t_{2\widetilde{n}}(\mathfrak{w}^{(n)})\leq T} 
\frac{1}{3}L^{-2}v^2_0(T-t_{2\widetilde{n}}(\mathfrak{w}^{(n)}))^3\right] \\
&=\mathbb{E}_{\mathbb{P}_1}\left[
\frac{1}{T}\cf{t_{2\widetilde{n}}(\mathfrak{w}^{(n)})\leq T} 
\frac{1}{3}L^{-2}v^2_0|T-t_{2\widetilde{n}}(\mathfrak{w}^{(n)})|^3 
\right]\\
&\leq\frac{1}{3T}L^{-2}v^2_0\left(
\mathbb{E}_{\mathbb{P}_1}\left[
|T-t_{2\widetilde{n}}(\mathfrak{w}^{(n)})|^4 
\right]\right)^{3/4}\,,
\end{split}
\end{equation}
where in the last inequality we used H\"older's inequality.
We note that
\begin{equation}\label{eq:evencase}
\frac{1}{3T}L^{-2}v^2_0\left(
\mathbb{E}_{\mathbb{P}_1}\left[
|T-t_{2\widetilde{n}}(\mathfrak{w}^{(n)})|^4 
\right]\right)^{3/4}\leq C_1\frac{T_*}{L^2_*}\frac{T_*}{n^{3/2}},
\end{equation}
where $C_1:=\mean{(u_1-1)^4}^{3/4}=9^{3/4}$ and $u_1$ has exponential distribution with parameter one. 

Now, we assume that $n$ is an odd number. We note that $2\widetilde{n}=n-1$ and then
\begin{equation}
\begin{split}
\mathbb{E}_{\mathbb{P}_1}&\left[
\frac{1}{T}\cf{t_{2\widetilde{n}}(\mathfrak{w}^{(n)})\leq T}\int_{t_{2\widetilde{n}}(\mathfrak{w}^{(n)})}^{T} |Y(t;n,\mathfrak{w})-Y(t_{2\widetilde{n}}(\mathfrak{w}^{(n)});n,\mathfrak{w}^{(n)})|^2  \rmd t
\right]\\
&=\mathbb{E}_{\mathbb{P}_1}\left[
\frac{1}{T}\cf{t_{2\widetilde{n}}(\mathfrak{w}^{(n)})\leq T}
\cf{t_{n}(\mathfrak{w}^{(n)})\leq T}
\int_{t_{2\widetilde{n}}(\mathfrak{w}^{(n)})}^{T} |Y(t;n,\mathfrak{w}^{(n)})-Y(t_{2\widetilde{n}}(\mathfrak{w}^{(n)});n,\mathfrak{w}^{(n)})|^2   \rmd t
\right]\\
&\qquad+\mathbb{E}_{\mathbb{P}_1}\left[
\frac{1}{T}\cf{t_{2\widetilde{n}}(\mathfrak{w}^{(n)})\leq T}
\cf{t_{n}(\mathfrak{w}^{(n)})> T}
\int_{t_{2\widetilde{n}}(\mathfrak{w}^{(n)})}^{T} L^{-2}v^2_0|t-t_{2\widetilde{n}}(\mathfrak{w}^{(n)})|^2  \rmd t
\right]\\
&=\mathbb{E}_{\mathbb{P}_1}\left[
\frac{1}{T}\cf{t_{2\widetilde{n}}(\mathfrak{w}^{(n)})\leq T}
\cf{t_{n}(\mathfrak{w}^{(n)})\leq T}
\int_{t_{2\widetilde{n}}(\mathfrak{w}^{(n)})}^{t_{2\widetilde{n}+1}(\mathfrak{w}^{(n)})} |Y(t;n,\mathfrak{w}^{(n)})-Y(t_{2\widetilde{n}}(\mathfrak{w}^{(n)});n,\mathfrak{w}^{(n)})|^2   \rmd t
\right]\\
&\qquad+\mathbb{E}_{\mathbb{P}_1}\left[
\frac{1}{T}\cf{t_{2\widetilde{n}}(\mathfrak{w}^{(n)})\leq T}
\cf{t_{n}(\mathfrak{w}^{(n)})\leq T}
\int_{t_{2\widetilde{n}+1}(\mathfrak{w}^{(n)})}^{T} L^{-2}v^2_0|t-t_{2\widetilde{n}+1}(\mathfrak{w}^{(n)})|^2  \rmd t
\right]\\
&\qquad+\mathbb{E}_{\mathbb{P}_1}\left[
\frac{1}{T}\cf{t_{2\widetilde{n}}(\mathfrak{w}^{(n)})\leq T}\cf{t_{n}(\mathfrak{w}^{(n)})> T} 
\frac{1}{3}L^{-2}v^2_0(T-t_{2\widetilde{n}}(\mathfrak{w}^{(n)}))^3\right] \\
&\leq  \frac{1}{2}\frac{T_*}{L^2_*} \frac{T_*}{n^3}+
\frac{1}{3T}L^{-2}v^2_0\left(
\mathbb{E}_{\mathbb{P}_1}\left[
|T-t_{2\widetilde{n}+1}(\mathfrak{w}^{(n)})|^4 
\right]\right)^{3/4}\\
& \qquad
+\frac{4}{3T}L^{-2}v^2_0
T^3\big|1-\frac{2\widetilde{n}}{n}\big|^3 
+\frac{4}{3T}L^{-2}v^2_0\Big(
\mathbb{E}_{\mathbb{P}_1}\Big[
\big|\frac{2\widetilde{n}T}{n}-t_{2\widetilde{n}}(\mathfrak{w}^{(n)})\big|^4 
\Big]\Big)^{3/4}\,.
\end{split}
\end{equation}
We note that
\[
\frac{1}{3T}L^{-2}v^2_0\left(
\mathbb{E}_{\mathbb{P}_1}\left[
|T-t_{2\widetilde{n}+1}(\mathfrak{w}^{(n)})|^4 
\right]\right)^{3/4}\leq C_1\frac{T_*}{L^2_*}\frac{T_*}{n^{3/2}}\,,
\]
where $C_1=\mean{(u_1-1)^4}^{3/4}$, where $u_1$ has exponential distribution with parameter one. 
A straightforward computation yields
\[
\frac{4}{3T}L^{-2}v^2_0
T^3\big|1-\frac{2\widetilde{n}}{n}\big|^3=\frac{4}{3}\frac{T_*}{L^2_*}\frac{T_*}{n^{3}}\,.
\]
We also have
\[
\frac{4}{3T}L^{-2}v^2_0\Big(
\mathbb{E}_{\mathbb{P}_1}\Big[
\big|\frac{2\widetilde{n}T}{n}-t_{2\widetilde{n}}(\mathfrak{w}^{(n)})\big|^4 
\Big]\Big)^{3/4}\leq 4C_1\frac{T_*}{L^2_*}\frac{T_*}{n^{3/2}}\,,
\]
\begin{equation}\label{eq:oddcase}
\begin{split}
\mathbb{E}_{\mathbb{P}_1}&\left[
\frac{1}{T}\cf{t_{2\widetilde{n}}(\mathfrak{w}^{(n)})\leq T}\int_{t_{2\widetilde{n}}(\mathfrak{w}^{(n)})}^{T} |Y(t;n,\mathfrak{w}^{(n)})-Y(t_{2\widetilde{n}}(\mathfrak{w}^{(n)});n,\mathfrak{w}^{(n)})|^2  \rmd t
\right]\\
&\hspace{5cm}\leq (2+C_1)\frac{T_*}{L^2_*} \frac{T_*}{n^3}+5C_1\frac{T_*}{L^2_*}\frac{T_*}{n^{3/2}}\,.
\end{split}
\end{equation}
Therefore \eqref{eq:evencase0}, \eqref{eq:evencase} and \eqref{eq:oddcase} yield for any $n\in \mathbb{N}\setminus\{0,1\}$ 
\begin{equation}\label{eq:both}
\begin{split}
\mathbb{E}_{\mathbb{P}_1}&\left[
\frac{1}{T}\cf{t_{2\widetilde{n}}(\mathfrak{w}^{(n)})\leq T}\int_{t_{2\widetilde{n}}(\mathfrak{w}^{(n)})}^{T} |Y(t;n,\mathfrak{w}^{(n)})-Y(t_{2\widetilde{n}}(\mathfrak{w}^{(n)});n,\mathfrak{w}^{(n)})|^2  \rmd t
\right]\\
&\hspace{5cm}\leq (2+C_1)\frac{T_*}{L^2_*} \frac{T_*}{n^3}+5C_1\frac{T_*}{L^2_*}\frac{T_*}{n^{3/2}}\,.
\end{split}
\end{equation}
By \eqref{eq:logo}, \eqref{eq:splitover}, \eqref{eq:both} and Lemma~\ref{lem:momentpoisson} in Appendix~\ref{ap:tools} we obtain
\begin{equation}
\begin{split}
\mathcal{W}^2_2(\mathbb{Y}_{[0,T]},\mathbb{\widetilde{Z}}_{[0,T]})&\leq 
\sum\limits_{n=2}^{\infty} p_n(T_*)
\Big(\frac{T_*}{L^2_*} \frac{T_*}{n^2}+(2+C_1)\frac{T_*}{L^2_*} \frac{T_*}{n^3}+5C_1\frac{T_*}{L^2_*}\frac{T_*}{n^{3/2}}\Big)\\
&\qquad+\frac{L^{-2}v^2_0T^2}{3}p_0(T_*)+\frac{1}{T}p_1(T_*)\int_{0}^T \mathbb{E}_{\mathbb{P}_1}[|Y(s;1,\fs)|^2] \rmd s\\
&\leq 8C_2\frac{T_*}{L^2_*}{T^{-1/2}_*}
+\frac{T^2_*}{3L^{2}_*}e^{-T_*}+e^{-T_*}T_*\frac{1}{T}\int_{0}^T \mathbb{E}_{\mathbb{P}_1}[|Y(s;1,\fs)|^2] \rmd s\,.
\end{split}
\end{equation}
where $C_2:=6C_1+3$.
We point out that for $n=1$ \eqref{eq:oddcase} implies 
\begin{equation}
\frac{1}{T}\int_{0}^T \mathbb{E}_{\mathbb{P}_1}[|Y(s;1,\fs)|^2] \rmd s\leq (2+6C_1)\frac{T_*}{L^2_*} T_*\,.
\end{equation}
This concludes the proof of Lemma~\ref{lem:YZ*}.
\end{proof}

In the sequel, we stress the fact that Proposition~\ref{prop:naturalcoupling} is just a consequence of what
we have already shown up to here.
\begin{proof}[Proof of Proposition~\ref{prop:naturalcoupling}]
Combining Lemma~\ref{lem:compZZ*} and Lemma~\ref{lem:YZ*}  in \eqref{eq:cYZtildeZ} implies \eqref{eq:pepino}. This finishes the proof of
Proposition~\ref{prop:naturalcoupling}.
\end{proof}

\section{\textbf{Proof of Lemma~\ref{lem:couplingKMT}: Koml\'os--Major--Tusn\'ady coupling}}
\label{sec:KMTcoupling}
In this section we 
apply the Koml\'os--Major--Tusn\'ady (KMT) coupling to
prove Lemma~\ref{lem:couplingKMT}. That is, the existence of a constant $\mathcal{K}$ satisfying
\begin{equation}\label{eq:final}
\mathcal{W}^2_2\left(\mathbb{Z}_{[0,T]}, \mathbb{B}_{[0,T]}\right)\leq 
\mathcal{K}T_*L^{-2}_*T^2_*e^{-T_*}+
\mathcal{K}T_*L^{-2}_*
T^{-1/2}_*(1+T^{-3/2}_*)\,.
\end{equation}
Recall that
\begin{equation}\label{eq:inetriv}
\begin{split}
\mathcal{W}^2_2\left(\mathbb{Z}_{[0,T]}, \mathbb{B}_{[0,T]}\right)
\leq \frac{1}{T}\int_{0}^{T}\rmd s~ \mathbb{E}_{\pi}\left[|Z(s)-B(s)|^2\right]\,
\end{split}
\end{equation}
for any coupling $\pi$ between $\mathbb{Z}_{[0,T]}$ and 
$\mathbb{B}_{[0,T]}$. 
Let $\pi^*$ be the so-called Koml\'os--Major--Tusn\'ady coupling given in Theorem~1 of \cite{Komlos1976} (or Lemma~\ref{lem:KMT} in Appendix~\ref{ap:coupling}) gluing with an independent Poisson random variable $K$ with parameter $T_*=\lambda T$. 
Since $K$ has Poisson distribution with parameter $T_*$, inequality \eqref{eq:inetriv},
the tower property (Item~iv) of Theorem~8.14 in \cite{Klenke2014}) and Fubini's theorem imply
\begin{equation}\label{eq:replicafub}
\begin{split}
\mathcal{W}^2_2\left(\mathbb{Z}_{[0,T]}, \mathbb{B}_{[0,T]}\right) &\leq  
p_0(T_*)\frac{1}{T}\int_{0}^{T} \rmd s~
\mathbb{E}_{\pi^*}\left[|Z(s)-B(s)|^2~\big| K=0\right]\\
&\qquad+p_1(T_*)\frac{1}{T}\int_{0}^{T}\rmd s~
\mathbb{E}_{\pi^*}\left[|Z(s)-B(s)|^2~\big| K=1\right]
+J_0\,,
\end{split}
\end{equation}
where
\begin{equation}\label{def:J0}
J_0:=\frac{1}{T}\sum\limits_{n=2}^{\infty} 
p_n(T_*)\int_{0}^{T}\rmd s~
\mathbb{E}_{\pi^*}\left[|Z(s)-B(s)|^2~\big| K=n\right]\,.
\end{equation}
In the sequel we estimate $J_0$. To make the proof easier to follow, we have divided it in 8 steps.

\noindent
\textbf{Step 1: First splitting.}
Let $\widetilde{n}=\lfloor n/2\rfloor$
and split $J_0$ as follows:
\[
J_0:=\frac{1}{T}\sum\limits_{n=2}^{\infty} 
p_n(T_*)\sum\limits_{j=1}^{\widetilde{n}}\int_{\frac{j-1}{\widetilde{n}}T}^{\frac{j}{\widetilde{n}}T}\rmd s~
\mathbb{E}_{\pi^*}\left[|Z(s)-B(s)|^2~\big| K=n\right]\,.
\]
For each $j\in \{1,\ldots,\widetilde{n}\}$ we
write $\Theta_{j-1}B(s):=B(s)-B((j-1)T/\widetilde{n})$ for any $s\geq 0$ and note that
$B(s)=B((j-1)T/\widetilde{n})+\Theta_{j-1}B(s)$. 
Since $(x-y)^2\leq 2x^2+2y^2$ for any  $x,y\in \mathbb{R}$, we have
\begin{equation}\label{equa1}
J_0\leq J_1+J_2\,,
\end{equation}
where 
\begin{equation}
\begin{split}
J_1:=\frac{2}{T}
\sum\limits_{n=2}^{\infty}p_n(T_*)\sum_{j=1}^{\widetilde{n}}\int_{\frac{j-1}{\widetilde{n}}T}^{\frac{j}{\widetilde{n}}T}\rmd s~
\mathbb{E}_{\pi^*}\left[|\Theta_{j-1}B(s)|^2~\big| K=n\right]\,
\end{split}
\end{equation}
and 
\begin{equation}
\begin{split}
J_2:=\frac{2}{T}
\sum\limits_{n=2}^{\infty}p_n(T_*)\sum_{j=1}^{\widetilde{n}}\int_{\frac{j-1}{\widetilde{n}}T}^{\frac{j}{\widetilde{n}}T}\rmd s~
\mathbb{E}_{\pi^*}\left[|Z(s)-B((j-1)T/\widetilde{n})|^2~\big| K=n\right]\,.
\end{split}
\end{equation}

\noindent
\textbf{Step 2:}
We start  with the estimate of $J_1$.
Since for each $j\in \{1,\ldots,\widetilde{n}\}$ the random variable $\Theta_{j-1}B(s)$ has Gaussian distribution with zero mean and variance $\sigma^2(s-(j-1)T/\widetilde{n})$, with $\sigma^2$ being defined in \eqref{eq:diffusivity}, we deduce 
\begin{equation}
\begin{split}
J_1 &\leq \frac{2}{T}
\sum\limits_{n=2}^{\infty}p_n(T_*)\sum_{j=1}^{\widetilde{n}}\int_{\frac{j-1}{\widetilde{n}}T}^{\frac{j}{\widetilde{n}}T}\rmd s~
\sigma^2\left(s-\frac{(j-1)}{\widetilde{n}}T\right)= T\sigma^2
\mean{\widetilde{n}^{-1}\cf{n\ge 2}}
_{\mathsf{Po}(T_*)}\,.
\end{split}
\end{equation}
The preceding inequality with the help of a slight shift in Lemma~\ref{lem:momentpoisson} in Appendix~\ref{ap:tools}, \eqref{eq:diffusivity},
the fact that  $\widetilde{m}\geq m/2-1$, $m\geq 2$ and the definition of $\sigma^2$ given in \eqref{eq:diffusivity}
 yields
\begin{equation}\label{eq:J1estimate}
\begin{split}
J_1&\leq 
 T\sigma^2
\mean{\widetilde{n}^{-1}\cf{n=2}}
_{\mathsf{Po}(T_*)}+
 T\sigma^2
\mean{\widetilde{n}^{-1}\cf{n\geq  3}}
_{\mathsf{Po}(T_*)}\\
&\leq T\sigma^2 p_2(T_*)+ 2T\sigma^2
\mean{(n-2)^{-1}\cf{n\geq  3}}
_{\mathsf{Po}(T_*)}\\
&\leq 2^{-1}\frac{T_*}{L^2_*}T^2_*e^{-T_*}+2^{7/2} \frac{T_*}{L^2_*}\frac{1}{T_*}\,.
\end{split}
\end{equation}
This finishes the estimate of $J_1$.

\noindent
\textbf{Step 3: Second splitting.}
We continue with the estimate of $J_2$.
The idea is to apply the 
Koml\'os--Major--Tusn\'ady coupling between a suitable random walk and the Brownian motion with diffusivity constant $\sigma^2$.
We define the random walk $(S_m(\fs):m\in \mathbb{N}_0)$ as follows:
\[
S_m(\fs):=\sum\limits_{j=1}^{m}\eta_j(\fs)\quad \textrm{ for }\quad m\in \mathbb{N}\quad \textrm{ and }\quad S_0(\fs):=0\,,
\]
where the sequence $(\eta_j(\fs):j\in \mathbb{N})$ are defined by
\[
\eta_j(\fs)=-v_0(s_{2j}-s_{2j-1})\quad \textrm{ for any }\quad j\in \mathbb{N}\,.
\]
Hence, the sequence $(\eta_j(\fs):j\in \mathbb{N})$ are i.i.d.\ random variables  with zero mean and variance 
$2v^2_0 \textrm{Var}[\fs_1]=2v^2_0\frac{1}{\lambda^2}$.
Set 
\begin{equation}\label{eq:lawe}
\eta^*_j(\fs):=\frac{\eta_j(\fs)}{L}=\frac{-v_0(s_{2j}-s_{2j-1})}{L}\quad \textrm{ for any }\quad j\in \mathbb{N}
\end{equation}
and note that 
\[
\frac{S_j(\fs)}{L}:=\sum\limits_{k=1}^{j}\eta^*_k(\fs)\quad \textrm{ for any }\quad j\in \mathbb{N}\,.
\]
Then
$
\mathbb{E}\left[\eta^*_j(\fs)\right]=0$ and 
$\mathrm{Var}\left[\eta^*_j(\fs)\right]=\frac{v^2_0}{L^2}2\textrm{Var}[\fs_1]=\frac{2v^2_0}{L^2\lambda^2}
=\frac{2}{L^2_*}$ for all $j\in \mathbb{N}$.

Similarly as \eqref{equa1} we obtain
\begin{equation}\label{equa2}
J_2\leq J^{(1)}_2+J^{(2)}_2\,,
\end{equation}
where
\begin{equation}
J^{(1)}_2:=\frac{4}{T}
\sum\limits_{n=2}^{\infty}p_n(T_*)\sum_{j=1}^{\widetilde{n}}\int_{\frac{j-1}{\widetilde{n}}T}^{\frac{j}{\widetilde{n}}T}\rmd s~
\mathbb{E}_{\pi^*}\left[|S_{j-1}(\fs)/L-B((j-1)T/\widetilde{n})|^2~\big| K=n\right]
\end{equation}
and
\begin{equation}
J^{(2)}_2:=\frac{4}{T}
\sum\limits_{n=2}^{\infty}p_n(T_*)\sum_{j=1}^{\widetilde{n}}\int_{\frac{j-1}{\widetilde{n}}T}^{\frac{j}{\widetilde{n}}T}\rmd s~
\mathbb{E}_{\pi^*}\left[|Z(s)-S_{j-1}(\fs)/L|^2~\big| K=n\right]\,.
\end{equation}
\noindent
\textbf{Step 4: Third splitting.}
In what follows we estimate $J^{(1)}_2$.
We start noticing that
\begin{equation}\label{eq:J1112}
J^{(1)}_2\leq J^{(1,1)}_2+J^{(1,2)}_2\,,
\end{equation}
where 
\begin{equation}
J^{(1,1)}_2:=\frac{8}{T}
\sum\limits_{n=2}^{\infty}p_n(T_*)\sum_{j=1}^{\widetilde{n}}\int_{\frac{j-1}{\widetilde{n}}T}^{\frac{j}{\widetilde{n}}T}\rmd s~
\mathbb{E}_{\pi^*}\left[|S_{j-1}(\fs)/L-\sum\limits_{\ell=1}^{j-1}\textsf{N}_\ell(0,2L^{-2}_*)|^2~\big| K=n\right]
\end{equation}
and
\begin{equation}
J^{(1,2)}_2:=\frac{8}{T}
\sum\limits_{n=2}^{\infty}p_n(T_*)\sum_{j=1}^{\widetilde{n}}\int_{\frac{j-1}{\widetilde{n}}T}^{\frac{j}{\widetilde{n}}T}\rmd s~
\mathbb{E}_{\pi^*}\left[|\sum\limits_{\ell=1}^{j-1}\textsf{N}_\ell(0,2L^{-2}_*)-\sum\limits_{\ell=1}^{j-1}\textsf{N}_\ell(0,\sigma^2 T/\widetilde{n})|^2~\big| K=n\right]\,,
\end{equation}
where $(\textsf{N}_\ell(0,\eta^2):\ell\in \mathbb{N})$ are i.i.d.\ Gaussian random variables with zero mean and variance $\eta^2>0$. 
We stress that the preceding sequence
$(\textsf{N}_\ell(0,\eta^2):\ell\in \mathbb{N})$ of random variables is the sequence used to construct the Koml\'os--Major--Tusn\'ady coupling $\pi^*$.

\noindent
\textbf{Step 5:}
Now, we estimate the simplest term $J^{(1,2)}_2$. It is straightforward to see that
\begin{equation}
\begin{split}
J^{(1,2)}_2&=\frac{8}{T}
\sum\limits_{n=2}^{\infty}p_n(T_*)\sum_{j=1}^{\widetilde{n}}\int_{\frac{j-1}{\widetilde{n}}T}^{\frac{j}{\widetilde{n}}T}\rmd s~(j-1)
|\sqrt{2L^{-2}_*}-
\sqrt{\sigma^2 T/\widetilde{n}}|^2\\
&=\frac{8}{T}
\sum\limits_{n=2}^{\infty}p_n(T_*)\sum_{j=1}^{\widetilde{n}}~(j-1)\frac{T}{\widetilde{n}}
|\sqrt{2L^{-2}_*}-
\sqrt{\sigma^2 T/\widetilde{n}}|^2\\
&=4
\sum\limits_{n=2}^{\infty}p_n(T_*)(\widetilde{n}-1)
|\sqrt{2L^{-2}_*}-
\sqrt{\sigma^2 T/\widetilde{n}}|^2\,.
\end{split}
\end{equation}
Since $\sigma^2=L^{-2}_* \lambda$,
we have
\begin{equation}
\begin{split}
J^{(1,2)}_2
&\leq 4
\sum\limits_{n=2}^{\infty}p_n(T_*)\widetilde{n}
|\sqrt{2L^{-2}_*}-
\sqrt{\sigma^2 T/\widetilde{n}}|^2= 4 L^{-2}_*
\sum\limits_{n=2}^{\infty}p_n(T_*)\widetilde{n}
|\sqrt{2}-
\sqrt{T_*/\widetilde{n}}|^2\\
&= 4 T_*L^{-2}_*
\sum\limits_{n=2}^{\infty}p_n(T_*)
\big|\sqrt{\frac{2\widetilde{n}}{T_*}}-
1\big|^2\leq 4 T_*L^{-2}_*
\sum\limits_{n=2}^{\infty}p_n(T_*)
\Big|\frac{2\widetilde{n}}{T_*}-
1\Big|\\
&\leq 4 T_*L^{-2}_*\Big(
\sum\limits_{n=2}^{\infty}p_n(T_*)
\big|\frac{2\widetilde{n}}{T_*}-
1\big|^2\Big)^{1/2}\,,
\end{split}
\end{equation}
where we have used the following subadditivity inequality:
$|\sqrt{x}-\sqrt{y}|\leq \sqrt{|x-y|}$ for any non-negative numbers $x$ and $y$, and in the last inequality we applied the Cauchy-Schwarz inequality.
The preceding inequality with the help of Lemma~\ref{lem:cotapois} in Appendix~\ref{ap:tools} implies
\begin{equation}\label{eqLtg}
J^{(1,2)}_2\leq 4 T_*L^{-2}_*\left(
\sum\limits_{n=2}^{\infty}p_n(T_*)
\left|\frac{2\widetilde{n}}{T_*}-
1\right|^2\right)^{1/2}\leq 
4\sqrt{35} T_*L^{-2}_* T^{-1/2}_*.
\end{equation}

\noindent
\textbf{Step 6:} In the sequel we estimate $J^{(1,1)}_2$. We note that
\begin{equation}\label{eq:J11}
J^{(1,1)}_2=8
\sum\limits_{n=2}^{\infty}\frac{p_n(T_*)}{\widetilde{n}}\sum_{j=1}^{\widetilde{n}}
\mathbb{E}_{\pi^*}\left[|S_{j-1}(\fs)/L-\sum\limits_{\ell=1}^{j-1}\textsf{N}_\ell(0,2L^{-2}_*)|^2~\big| K=n\right]\,.
\end{equation}
Let $\eta>0$ be fixed.
For any $j\in \{1,\ldots,\widetilde{n}\}$
we consider the following split
\begin{equation}\label{eq:J11eta}
\begin{split}
\mathbb{E}_{\pi^*}&\left[|S_{j-1}(\fs)/L-\sum\limits_{\ell=1}^{j-1}\textsf{N}_\ell(0,2L^{-2}_*)|^2~\big| K=n\right]\\
&
=\mathbb{E}_{\pi^*}\left[|S_{j-1}(\fs)/L-\sum\limits_{\ell=1}^{j-1}\textsf{N}_\ell(0,2L^{-2}_*)|^2 \mathbbm{1}_{A^{\textsf{c}}_\eta}~\big| K=n\right]\\
&\quad+\mathbb{E}_{\pi^*}\left[|S_{j-1}(\fs)/L-\sum\limits_{\ell=1}^{j-1}\textsf{N}_\ell(0,2L^{-2}_*)|^2 \mathbbm{1}_{A_\eta}~\big| K=n\right]\\
&
\leq \eta^2+\mathbb{E}_{\pi^*}\left[|S_{j-1}(\fs)/L-\sum\limits_{\ell=1}^{j-1}\textsf{N}_\ell(0,2L^{-2}_*)|^2 \mathbbm{1}_{A_\eta}~\big| K=n\right]\,,
\end{split}
\end{equation}
where 
\[
A_\eta:=\left\{
\sup\limits_{1\leq j\leq \widetilde{n}}|S_{j-1}(\fs)/L-\sum\limits_{\ell=1}^{j-1}\textsf{N}_\ell(0,2L^{-2}_*)|^2>\eta^2
\right\}\,.
\]
By the Cauchy-Schwarz inequality we have
\begin{equation}\label{eq:J11Young}
\begin{split}
\mathbb{E}&_{\pi^*}\left[|S_{j-1}(\fs)/L-\sum\limits_{\ell=1}^{j-1}\textsf{N}_\ell(0,2L^{-2}_*)|^2 \mathbbm{1}_{A_\eta}~\big| K=n\right]\\
&\quad\leq 
\left(\mathbb{E}_{\pi^*}\left[|S_{j-1}(\fs)/L-\sum\limits_{\ell=1}^{j-1}\textsf{N}_\ell(0,2L^{-2}_*)|^{4}~\big| K=n\right]\right)^{1/2}\left(\pi^*\left(A_\eta~\big| K=n\right)\right)^{1/2}\,.
\end{split}
\end{equation}
By the KMT coupling (see Theorem~1 of \cite{Komlos1976})
 there exist positive constants $C^{(1)}_{\textsf{KMT}}$, $C^{(2)}_{\textsf{KMT}}$ and $\vartheta$
 such that for every $\widetilde{n}$ it follows that
\begin{equation}\label{eq:J11KMT}
\begin{split}
\pi^*\left(A_{\eta_*(n)}~\big| K=n\right)\leq C^{(2)}_{\textsf{KMT}} e^{-\vartheta\ln(n)}=\frac{C^{(2)}_{\textsf{KMT}}}{n^\vartheta}\,,
\end{split}
\end{equation}
where 
\begin{equation}\label{eq:J12con}
\eta_*(n):=\frac{\sqrt{2}}{L_*}\left({C_{\textsf{KMT}}\ln(\widetilde{n})+\ln(n)}\right)\,.
\end{equation}
We point out that the constants
$C^{(1)}_{\textsf{KMT}}$, $C^{(2)}_{\textsf{KMT}}$ and $\vartheta$  only depend on the law of the random variable \eqref{eq:lawe}. Later on, we choose $\vartheta=4$.
Combining \eqref{eq:J11eta}, \eqref{eq:J11Young}, \eqref{eq:J11KMT} and \eqref{eq:J12con} in \eqref{eq:J11} we obtain
\begin{equation}\label{eq:htyu98}
\begin{split}
J^{(1,1)}_2\leq 8
\sum\limits_{n=2}^{\infty}p_n(T_*)\eta^2_*(n)+8
\sum\limits_{n=2}^{\infty}\frac{p_n(T_*)}{\widetilde{n}}\left(
\frac{C^{(2)}_{\textsf{KMT}}}{n^\vartheta}\right)^{1/2}
\sum_{j=1}^{\widetilde{n}}
a_j\,,
\end{split}
\end{equation}
where
\[
a_j:=
\left(\mathbb{E}_{\pi^*}\left[|S_{j-1}(\fs)/L-\sum\limits_{\ell=1}^{j-1}\textsf{N}_\ell(0,2L^{-2}_*)|^{4}~\big| K=n\right]\right)^{1/2}\quad \textrm{ for any }\quad j\in \{1,\ldots,\widetilde{n} \}\,.
\]
Note that for any $j\in \{1,\ldots,\widetilde{n    } \}$ we have
\[
b_j:=
\left(\mathbb{E}_{\pi^*}\left[\sum\limits_{\ell=1}^{j-1}\textsf{N}_\ell(0,2L^{-2}_*)|^{4}\right]\right)^{1/2}
=\left(\mathbb{E}_{\pi^*}\left[\textsf{N}_\ell(0,2(j-1)L^{-2}_*)|^{4}\right]\right)^{1/2}=2(j-1)L^{-2}_*
\]
and
\[
c_j:=
\left(\mathbb{E}_{\pi^*}\left[|S_{j-1}(\fs)/L|^{4}\right]\right)^{1/2}\,.
\]
Moreover, straightforward computations yields
$c_j=\frac{1}{L^2}(\mathbb{E}[|S_{j-1}|^4])^{1/2}=L^{-2}_*\sqrt{22j+2j^2}\leq 5L^{-2}_* j$.
Since
$a_j\leq 3(b_j+c_j)$,
\eqref{eq:htyu98} with the chose of $\vartheta=4$ and Lemma~\ref{lem:momentpoisson} in Appendix~\ref{ap:tools} implies
\begin{equation}
\begin{split}
J^{(1,1)}_2&\leq 32\max\{1,C^2_{\textsf{KMT}}\}L^{-2}_*
\sum\limits_{n=2}^{\infty}p_n(T_*)\ln^2(n)+84\sqrt{C^{(2)}_{\textsf{KMT}}}L^{-2}_*
\sum\limits_{n=2}^{\infty}{p_n(T_*)}n^{-1}
\\
&\leq 
32\max\{1,C^2_{\textsf{KMT}}\}L^{-2}_*
\sum\limits_{n=2}^{\infty}p_n(T_*)\ln^2(n)+252\sqrt{C^{(2)}_{\textsf{KMT}}}L^{-2}_*T^{-1}_*\,.
\end{split}
\end{equation}
Let $\gamma\in (0,1/2)$ be fixed and note that there exists a constant $C(\gamma)>0$ such that
$\ln(n)\leq C(\gamma )n^\gamma$ for all $n\in \mathbb{N}$. Hence, H\"older's inequality yields
\begin{equation}
\begin{split}
J^{(1,1)}_2
&\leq 
32C^2(\gamma)\max\{1,C^2_{\textsf{KMT}}\}T_*L^{-2}_*
T^{2\gamma-1}_*
+252\sqrt{C^{(2)}_{\textsf{KMT}}}T_*L^{-2}_*T^{-2}_*\,.
\end{split}
\end{equation}
The choice of $\gamma=1/4$ implies the existence of a positive constant $C^{(1)}_2$ such that 
\begin{equation}\label{eqp:gy}
\begin{split}
J^{(1,1)}_2
&\leq 
C^{(1)}_2 T_*L^{-2}_*
T^{-1/2}_*(1+T^{-3/2}_*)\,.
\end{split}
\end{equation}
Combining \eqref{eqLtg} and \eqref{eqp:gy} in \eqref{eq:J1112} we obtain
\begin{equation}\label{eq:parteuno}
J^{(1)}_2\leq 
(C^{(1)}_2+4\sqrt{35})T_*L^{-2}_*
T^{-1/2}_*(1+T^{-3/2}_*)\,.
\end{equation}

\noindent
\textbf{Step 7:}
Finally, we estimate $J^{(2)}_2$.
Let $(S(t):t\geq 0)$ be the C\`adl\`ag piecewise constant extension of the random walk $(\eta^*_{\ell}(\fs):\ell\in \mathbb{N}_0)$ given in \eqref{eq:lawe}.
Since 
\[
Z(t;n,\fs)=\sum\limits_{\ell=1}^{\widetilde{n}}\eta^*_{\ell}(\fs)\cf{t>t_{2\ell}(\fs)}
\quad \textrm{ and } \quad
S(t;n,\fs)=\sum\limits_{\ell=1}^{\widetilde{n}}\eta^*_{\ell}(\fs)
\cf{t>\frac{\ell T}{\widetilde{n}}}\,,
\]
we have 
\begin{equation}
\begin{split}
\frac{1}{T}\int_{0}^{T}\rmd t \,|Z(t;n,\fs)-S(t;n,\fs)|^2&\\
&\hspace{-4.5cm}=\sum\limits_{\ell=1}^{\widetilde{n}}\sum\limits_{\ell^\prime=1}^{\widetilde{n}}\eta^*_{\ell}(\fs)\eta^*_{\ell^\prime}(\fs)
\frac{1}{T}\int_{0}^{T}\rmd t
\left(\cf{t>t_{2\ell}(\fs)}-\cf{t>\frac{\ell T}{\widetilde{n}}}\right)\left(\cf{t>t_{2\ell^\prime}(\fs)}-\cf{t>\frac{\ell^\prime T}{\widetilde{n}}}\right)\,.
\end{split}
\end{equation}
In the sequel, we recall
 the change of variable introduced in \eqref{eq:ch1} and \eqref{eq:ch2}. That is,
for any $\ell\in \{1,\ldots,\widetilde{n}\}$ we set
\begin{equation}
\hat{s}_\ell :=s_{2\ell}+s_{2\ell-1}\quad \textrm{ and }\quad
\omega_\ell :=\frac{s_{2\ell}-s_{2\ell-1}}{\hat{s}_\ell}\,
\end{equation}
and observe that
\begin{equation}
\eta^*_{\ell}(\fs)=-\frac{v_0}{L} \hat{s}_\ell\omega_\ell\quad
\textrm{ for any } \quad\ell\in \{1,\ldots,\widetilde{n}\}\,.
\end{equation}
An analogous reasoning using in \eqref{eq:cita1} with the help of Item~i) and Item~iii) of Lemma~\ref{lem:independencia} in Appendix~\ref{ap:basic} implies
\begin{equation}\label{eq:nt1}
\begin{split}
\mean{\frac{1}{T}\int_{0}^{T}\rmd t \,|Z(t;n,\fs)-S(t;n,\fs)|^2}_{\omega}&\\
&\hspace{-4.5cm}
=\left(\frac{v_0}{L}\right)^2\sum\limits_{\ell=1}^{\widetilde{n}}\sum\limits_{\ell^\prime=1}^{\widetilde{n}} \hat{s}_\ell \hat{s}_{\ell^\prime}\mean{\omega_\ell\omega_{\ell^\prime}}_{\omega}
\frac{1}{T}\int_{0}^{T}\rmd t
\left(\cf{t>t_{2\ell}(\fs)}-\cf{t>\frac{\ell T}{\widetilde{n}}}\right)\left(\cf{t>t_{2\ell^\prime}(\fs)}-\cf{t>\frac{\ell^\prime T}{\widetilde{n}}}\right)\\
&\hspace{-4.5cm}
=\frac{1}{3}\left(\frac{v_0}{L}\right)^2\sum\limits_{\ell=1}^{\widetilde{n}} \hat{s}^2_\ell
\frac{1}{T}\int_{0}^{T}\rmd t
\left(\cf{t>t_{2\ell}(\fs)}-\cf{t>\frac{\ell T}{\widetilde{n}}}\right)^2\\
&\hspace{-4.5cm}
=\frac{1}{3}\left(\frac{v_0}{L}\right)^2\sum\limits_{\ell=1}^{\widetilde{n}} \hat{s}^2_\ell
\frac{1}{T}\int_{0}^{T}\rmd t
\cf{\min\{t_{2\ell}(\fs),\frac{\ell T}{\widetilde{n}}\}<t<\max\{t_{2\ell}(\fs),\frac{\ell T}{\widetilde{n}}\}}\\
&\hspace{-4.5cm}
\leq \frac{1}{3}\left(\frac{v_0}{L}\right)^2\sum\limits_{\ell=1}^{\widetilde{n}} \hat{s}^2_\ell
\frac{1}{T}\left(
\max\left\{t_{2\ell}(\fs),\frac{\ell T}{\widetilde{n}}\right\}-\min\left\{t_{2\ell}(\fs),\frac{\ell T}{\widetilde{n}}\right\}\right)\\
&\hspace{-4.5cm}
\leq \left(\frac{v_0}{L}\right)^2\sum\limits_{\ell=1}^{\widetilde{n}} \hat{s}^2_\ell \frac{1}{T}
\left|t_{2\ell}(\fs)-\frac{\ell T}{\widetilde{n}}\right|\,,
\end{split}
\end{equation}
where $\omega$ denotes the product measure given by the law of the random vector $(\omega_1,\ldots,\omega_{\widetilde{n}})$.
We note that
\begin{equation}\label{eq:nt2}
\sum\limits_{\ell=1}^{\widetilde{n}} \hat{s}^2_\ell
\big|t_{2\ell}(\fs)-\frac{\ell T}{\widetilde{n}}\big|\leq
\sum\limits_{\ell=1}^{\widetilde{n}} \hat{s}^2_\ell
\big|t_{2\ell}(\fs)-\frac{2\ell }{\lambda}\big| +
\sum\limits_{\ell=1}^{\widetilde{n}} \hat{s}^2_\ell
\big|\frac{2\ell }{\lambda}-\frac{\ell T}{\widetilde{n}}\big|\,
\end{equation}
and let $\hat{s}$ denote the product measure given by the law of the random vector $(\hat{s}_1,\ldots,\hat{s}_{\widetilde{n}})$.
On the one hand,
Item~ii) of Lemma~\ref{lem:independencia} in Appendix~\ref{ap:basic} with the help of Fubini's theorem yields
\begin{equation}\label{eq:nt3}
\begin{split}
&\mean{\left(\frac{v_0}{L}\right)^2
\sum\limits_{\ell=1}^{\widetilde{n}} \hat{s}^2_\ell \frac{1}{T}
\left|\frac{2\ell }{\lambda}-\frac{\ell T}{\widetilde{n}}\right|}_{\hat{s}}=\left(\frac{v_0}{L}\right)^2
\sum\limits_{\ell=1}^{\widetilde{n}} \mean{\hat{s}^2_\ell}_{\hat{s}} \frac{1}{T}
\left|\frac{2\ell }{\lambda}-\frac{\ell T}{\widetilde{n}}\right|\\
&\qquad=2\left(\frac{v_0}{\lambda L}\right)^2
\sum\limits_{\ell=1}^{\widetilde{n}} 2\ell
\left|\frac{1}{T_*}-\frac{1}{2\widetilde{n}}\right|\leq 2\left(\frac{v_0}{\lambda L}\right)^2 
{\widetilde{n}}^2
\left|\frac{1}{T_*}-\frac{1}{2\widetilde{n}}\right|\leq \frac{T_*}{L^2_*}
\frac{\widetilde{n}}{T_*}
\left|\frac{2\widetilde{n}}{T_*}-1\right|\,.
\end{split}
\end{equation}
On the other hand,  the Cauchy-Schwarz inequality implies
\begin{equation}\label{eq:nt4}
\begin{split}
&\mean{\left(\frac{v_0}{L}\right)^2
\sum\limits_{\ell=1}^{\widetilde{n}} \hat{s}^2_\ell \frac{1}{T}
\left|t_{2\ell}(\fs)-\frac{2\ell }{\lambda}\right|}_{\hat{s}}=
\mean{\frac{1}{L^2_* T}
\sum\limits_{\ell=1}^{\widetilde{n}} \lambda^2\hat{s}^2_\ell 
\left|t_{2\ell}(\fs)-\frac{2\ell }{\lambda}\right|}_{\hat{s}}\\
&\qquad\leq \frac{1}{L^2_* T}
\sqrt{\sum\limits_{\ell=1}^{\widetilde{n}} \lambda^4\mean{\hat{s}^4_\ell}_{\hat{s}}}
\sqrt{\sum\limits_{\ell=1}^{\widetilde{n}}
\mean{
\left|t_{2\ell}(\fs)-\frac{2\ell }{\lambda}\right|^2}_{\hat{s}}}
\leq \frac{1}{L^2_* T}
\sqrt{120\widetilde{n}}
\sqrt{\sum\limits_{\ell=1}^{\widetilde{n}}
\frac{2\ell}{\lambda^2}}\\
&\qquad\leq 
\frac{1}{L^2_* T}
\sqrt{120\widetilde{n}}
\frac{\widetilde{n}}{\lambda}\leq 
11\frac{T_*}{L^2_*}
\frac{n^{3/2}}{T^2_*}\,.
\end{split}
\end{equation}
Combining \eqref{eq:nt2}, \eqref{eq:nt3} and \eqref{eq:nt4} in \eqref{eq:nt1} with the help of Fubini's theorem implies
\begin{equation}\label{eq:nt5}
\mean{\frac{1}{T}\int_{0}^{T}\rmd t \,|Z(t;n,\fs)-S(t;n,\fs)|^2}_{\hat{s},\omega}
\leq \frac{T_*}{L^2_*}
\frac{\widetilde{n}}{T_*}
\left|\frac{2\widetilde{n}}{T_*}-1\right|+11\frac{T_*}{L^2_*}
\frac{n^{3/2}}{T^2_*}\,.
\end{equation}
Recall that $n$ has Poisson distribution of parameter $T_*$. Then we have
\begin{equation}\label{eq:nt6}
\begin{split}
\sum\limits_{n=2}^{\infty}& p_n(T_*)
\mean{\frac{1}{T}\int_{0}^{T}\rmd t \,|Z(t;n,\fs)-S(t;n,\fs)|^2}_{\hat{s},\omega}\\
&\hspace{3cm}\leq \frac{T_*}{L^2_*}\mean{
\frac{\widetilde{n}}{T_*}
\left|\frac{2\widetilde{n}}{T_*}-1\right|\cf{n\geq 2}}_{\mathsf{Po}(T_*)}+11\frac{T_*}{L^2_*}
\frac{\mean{n^{3/2}\cf{n\geq 2}}_{\textsf{Po}(T_*)}}{T^2_*}\,.
\end{split}
\end{equation}
Now, we estimate the first term of the right-hand side of \eqref{eq:nt6}.
By the Cauchy-Schwarz inequality we obtain
\begin{equation}
\begin{split}
\frac{T_*}{L^2_*}\mean{
\frac{\widetilde{n}}{T_*}
\left|\frac{2\widetilde{n}}{T_*}-1\right|\cf{n\geq 2}}_{\mathsf{Po}(T_*)}\leq 
\frac{T_*}{L^2_*}\Big(\mean{
\frac{\widetilde{n}^2}{T^2_*}
\cf{n\geq 2}}_{\mathsf{Po}(T_*)}\Big)^{1/2}
\Big(\mean{
\left|\frac{2\widetilde{n}}{T_*}-1\right|^2
\cf{n\geq 2}}_{\mathsf{Po}(T_*)}\Big)^{1/2}\,.
\end{split}
\end{equation}
We note that 
\[
\mean{
\frac{\widetilde{n}^2}{T^2_*}
\cf{n\geq 2}}_{\mathsf{Po}(T_*)}\leq 
\frac{1}{4T^2_*}
\mean{n^2}_{\mathsf{Po}(T_*)}\leq 
\frac{1}{4T_*}+\frac{1}{4}\,.
\]
Also, by Lemma~\ref{lem:cotapois} in Appendix~\ref{ap:tools} we obtain
\[
\mean{
\left|\frac{2\widetilde{n}}{T_*}-1\right|^2
\cf{n\geq 2}}_{\mathsf{Po}(T_*)}
\leq \frac{35}{T_*}.
\]
Then
\begin{equation}\label{eq:primera}
\begin{split}
\frac{T_*}{L^2_*}\mean{
\frac{\widetilde{n}}{T_*}
\left|\frac{2\widetilde{n}}{T_*}-1\right|\cf{n\geq 2}}_{\mathsf{Po}(T_*)}\leq 
3\frac{T_*}{L^2_*}T^{-1/2}_*(T^{-1/2}_*+1)\,.
\end{split}
\end{equation}

Now, we estimate the second term of the right-hand side of \eqref{eq:nt6}.
Then H\"older's inequality yields
\begin{equation}\label{eq:segunda}
\begin{split}
\frac{T_*}{L^2_*}
\frac{\mean{n^{3/2}\cf{n\geq 2}}_{\textsf{Po}(T_*)}}{T^2_*}&\leq \frac{T_*}{L^2_*}
\frac{\mean{n^{3/2}}_{\textsf{Po}(T_*)}}{T^2_*}\leq \frac{T_*}{L^2_*}\frac{1}{T^2_*}
\left(\mean{n^{2}}_{\textsf{Po}(T_*)}\right)^{3/4}\\
&= \frac{T_*}{L^2_*}\frac{1}{T^2_*}
\left(T_*+T^2_*\right)^{3/4}\leq  \frac{T_*}{L^2_*}(T^{-1/2}_*+T^{-5/4}_*).
\end{split}
\end{equation}
Combining \eqref{eq:primera} and \eqref{eq:segunda} in \eqref{eq:nt6} we obtain
\begin{equation}\label{eq:nt7}
J^{(2)}_2\leq 3\frac{T_*}{L^2_*}T^{-1/2}_*(T^{-1/2}_*+1)+
11\frac{T_*}{L^2_*}T^{-1/2}_*(1+T^{-3/4}_*)\,.
\end{equation}
Therefore, \eqref{equa2}, \eqref{eq:parteuno} and \eqref{eq:nt7} yields the existence of a constant $K_1>0$ satisfying
\begin{equation}\label{eq:nt8}
\begin{split}
J_2&\leq (C^{(1)}_2+24)T_*L^{-2}_*
T^{-1/2}_*(1+T^{-3/2}_*)+3T_*L^{-2}_*T^{-1/2}_*(1+T^{-1/2}_*)+
11T_*L^{-2}_*T^{-1/2}_*(1+T^{-3/4}_*)\\
&\leq K_1 T_*L^{-2}_*
T^{-1/2}_*(1+T^{-1/2}_*+T^{-3/4}_*+T^{-3/2}_*)\,.
\end{split}
\end{equation}

\noindent
\textbf{Step 8: Concluding step.}
As a consequence of \eqref{equa1}, \eqref{eq:J1estimate} and \eqref{eq:nt8}
 we obtain an upper bound for $J_0$ defined in \eqref{def:J0}. The latter with the help of \eqref{eq:replicafub} yields \eqref{eq:final}.
\appendix
{
\section{\textbf{Moment estimates}}\label{ap:tools}
This section contains derivation of the moment estimates used in the main text.
\subsection{\textbf{Poisson moment estimates}}
\begin{lemma}[Bounds for the inverse moments  of Poisson r.v.]\label{lem:momentpoisson}
Let $N$ be a Poisson random variable with parameter $\lambda>0$. Then for any $p>0$ it follows that
\begin{equation}\label{eq:pmomentpoi}
\mean{N^{-p}\cf{N\ge 1}}_{\mathsf{Po}(\lambda)}\leq 
\frac{C_p}{\lambda^{p}}\,,
\quad \textrm{ where }
\quad C_p:=2^{p(p+2)/2}\,.
\end{equation}
In addition,
\begin{equation}\label{eq:logmoment}
\mean{N^{-p}\ln(N+1)\cf{N\ge 1}}_{\mathsf{Po}(\lambda)}\leq 
\frac{C^{1/2}_{2p} \ln(\lambda+3)}{\lambda^{p}}\,.
\end{equation}
\end{lemma}
\begin{proof}
We start with the case $p\in \mathbb{N}$.
By definition we have 
\begin{equation}\label{eq:pmoment}
\begin{split}
\mean{N^{-p}\cf{N\ge 1}}_{\mathsf{Po}(\lambda)}=\sum\limits_{j\geq 1}\frac{1}{j^p}e^{-\lambda }\frac{\lambda^j}{j!}
=
\frac{e^{-\lambda}}{\lambda^p}
\sum\limits_{j\geq 1}\frac{\lambda^{j+p}}{j^p(j!)}\,.
\end{split}
\end{equation}
Note that
\[
\frac{1}{j}\leq \frac{2}{j+1} \quad \textrm{ for any } j\in \mathbb{N}\,.
\]
Hence,
\begin{equation}\label{eq:productbound}
\frac{1}{j^p}\leq 
\prod_{k=1}^{p}
\frac{2^k}{j+k}\quad \textrm{ for any } j\in \mathbb{N}\,.
\end{equation}
By \eqref{eq:pmoment} and \eqref{eq:productbound} we have 
\begin{equation}\label{eq:momentpinteger}
\begin{split}
\mean{N^{-p}\cf{N\ge 1}}_{\mathsf{Po}(\lambda)}\leq K_p
\frac{e^{-\lambda}}{\lambda^p}
\sum\limits_{j\geq 1}\frac{\lambda^{j+p}}{(j+p)!}=
K_p
\frac{e^{-\lambda}}{\lambda^p}\left(e^{-\lambda}-\sum\limits_{j=0}^{p}\frac{\lambda^{j}}{j!}\right)
\leq 
\frac{K_p}{\lambda^p}\,,
\end{split}
\end{equation}
where $K_p:=\prod_{k=1}^{p}2^k=2^{p(p+1)/2}\leq C_p$.

We continue with the case $p>1$ and $p\not \in \mathbb{N}$.
Let $\lceil p\rceil$ be the least integer greater than or equal to $p$ and set $p_1=\lceil p\rceil/p$.
We observe that $p_1>1$ and let $p_2$ be the H\"older conjugate of $p_1$, that is, $1/p_1+1/p_2=1$. 
By H\"older's inequality we have 
\begin{equation}\label{eq:eqnueva1}
\begin{split}
\mean{N^{-p}\cf{N\geq 1}}^{1/p}_{\mathsf{Po}(\lambda)}&=
\mean{N^{-p}\cf{N\geq 1}\cdot 1}^{1/p}_{\mathsf{Po}(\lambda)}
\leq \left(\mean{\left(N^{-p}\cf{N\geq 1}\right)^{p_1}}^{1/p_1}_{\mathsf{Po}(\lambda)}
\mean{1^{p_2}}^{1/p_2}_{\mathsf{Po}(\lambda)}\right)^{1/p}
\\
&\leq \mean{\left(N^{-p}\cf{N\geq 1}\right)^{p_1}}^{\frac{1}{pp_1}}_{\mathsf{Po}(\lambda)}
=
\mean{N^{-\lceil p\rceil}\cf{N\geq 1}}_{\mathsf{Po}(\lambda)}^{1/\lceil p\rceil}\,.
\end{split}
\end{equation}
Since 
$\lceil p\rceil \in \mathbb{N}\setminus \{1\}$, \eqref{eq:momentpinteger} yields
\begin{equation}\label{eq:eqnueva2}
\begin{split}
\mean{N^{-\lceil p\rceil}\cf{N\geq 1}}_{\mathsf{Po}(\lambda)}^{1/\lceil p\rceil}
\leq \left(\frac{K_{\lceil p\rceil}}{\lambda^{\lceil p\rceil}}\right)^{1/\lceil p\rceil}=\frac{K_{\lceil p\rceil}^{1/\lceil p\rceil}}{\lambda}\,,
\end{split}
\end{equation}
where 
$K_{\lceil p\rceil}=2^{\lceil p\rceil(\lceil p\rceil+1)/2}$.
By \eqref{eq:eqnueva1} and \eqref{eq:eqnueva2} we obtain
\[
\mean{N^{-p}\cf{N\geq 1}}_{\mathsf{Po}(\lambda)}
\leq\frac{K_{\lceil p\rceil}^{p/\lceil p\rceil}}{\lambda^p}\leq \frac{2^{p(p+2)/2}}{\lambda^p}=\frac{C_p}{\lambda^p}\,.
\]

Finally, for $p\in (0,1)$, let $q:=1/p>1$.
By H\"older's inequality and
\eqref{eq:momentpinteger} (for $p=1$)
 we obtain
\[
\mean{N^{-p}\cf{N\geq 1}}_{\mathsf{Po}(\lambda)}\leq 
\mean{(N^{-p})^q\cf{N\geq 1}}^{1/q}_{\mathsf{Po}(\lambda)}=\mean{N^{-1}\cf{N\geq 1}}^{1/q}_{\mathsf{Po}(\lambda)}\leq \left(\frac{K_{1}}{\lambda}\right)^{p}\leq \frac{C_p}{\lambda^p}\,.
\]

Now, we prove \eqref{eq:logmoment}.
Note that 
the Jensen inequality implies
\begin{equation}\label{e:jensen}
\begin{split}
\mean{(\ln(N+3))^2}_{\mathsf{Po}(\lambda)}\leq (\ln(\lambda +3))^2.
\end{split}
\end{equation}
By the Cauchy-Schwarz inequality, inequality \eqref{eq:pmomentpoi} and the Jensen inequality we have
\begin{equation*}
\begin{split}
\mean{N^{-p}\ln(N+1)\cf{N\ge 1}}_{\mathsf{Po}(\lambda)}&\leq
\left(\mean{N^{-2p}\cf{N\ge 1}}_{\mathsf{Po}(\lambda)}\right)^{1/2}\left(\mean{(\ln(N+1))^2\cf{N\ge 1}}_{\mathsf{Po}(\lambda)}\right)^{1/2}\\
&\leq \frac{C^{1/2}_{2p}}{\lambda^p}\left(\mean{(\ln(N+1))^2}_{\mathsf{Po}(\lambda)}\right)^{1/2}\leq \frac{C^{1/2}_{2p}\ln(\lambda+3)}{\lambda^p}\,.
\end{split}
\end{equation*}
\end{proof}

\begin{lemma}\label{lem:cotapois}
Let $N$ be a Poisson random variable with parameter $\lambda>0$. Then 
\begin{equation}\label{eq:f0}
\mean{\left|\frac{2\lfloor N/2  \rfloor}{\lambda}-
1\right|^2 \cf{N\geq 2}}_{\mathsf{Po}(\lambda)}\leq 
\frac{e^{-\lambda} (2-\lambda)^2}{2}+\frac{9}{\lambda}\,.
\end{equation}
In particular,
\[
\mean{\left|\frac{2\lfloor N/2  \rfloor}{\lambda}-
1\right|^2 \cf{N\geq 2}}_{\mathsf{Po}(\lambda)}\leq 
\frac{35}{\lambda}\,.
\]
\end{lemma}
\begin{proof}
We observe that 
\begin{equation}\label{eq:f1}
\begin{split}
\mean{\left|\frac{2\lfloor N/2  \rfloor}{\lambda}-
1\right|^2 \cf{N\geq 2}}_{\mathsf{Po}(\lambda)}=
\left|\frac{2}{\lambda}-1\right|^2 p_2(\lambda)+
\mean{\left|\frac{2\lfloor N/2  \rfloor}{\lambda}-
1\right|^2 \cf{N\geq 3}}_{\mathsf{Po}(\lambda)}\,.
\end{split}
\end{equation}
An straightforward computation yields
\begin{equation}\label{eq:f2}
\begin{split}
\mean{\left|\frac{2\lfloor N/2  \rfloor}{\lambda}-
1\right|^2 \cf{N\geq 3}}_{\mathsf{Po}(\lambda)}&=
\mean{\frac{4\lfloor N/2  \rfloor^2}{\lambda^2}\cf{N\geq 3}}_{\mathsf{Po}(\lambda)}
-\mean{\frac{4\lfloor N/2  \rfloor}{\lambda}\cf{N\geq 3}}_{\mathsf{Po}(\lambda)}
+\mean{\cf{N\geq 3}}_{\mathsf{Po}(\lambda)}\\
&\leq 
\mean{\frac{N^2}{\lambda^2}}_{\mathsf{Po}(\lambda)}
+\mean{\left(\frac{4-2N}{\lambda}\right)\cf{N\geq 3}}_{\mathsf{Po}(\lambda)}
+1\\
&\leq 
\frac{5}{\lambda}+2
-\frac{2}{\lambda}\mean{N\cf{N\geq 3}}_{\mathsf{Po}(\lambda)}\\
&\leq 
\frac{5}{\lambda}+2
-\frac{2}{\lambda}\left(\lambda-p_1(\lambda)-2p_2(\lambda)\right)\\
&\leq 
\frac{5}{\lambda}
+\frac{2}{\lambda}\left(p_1(\lambda)+2 p_2(\lambda)\right)\leq 
\frac{9}{\lambda}\,.
\end{split}
\end{equation}
Finally, \eqref{eq:f1} and \eqref{eq:f2} imply  \eqref{eq:f0}.
\end{proof}

\subsection{\textbf{$\delta$-constraint moment estimates}}
\begin{lemma}[Polynomial moments and correlations for the $\delta$-function]\label{lem:moments}
Let $n\in \mathbb{N}$.
For $p\in \N$ and any $i\in \{1,\ldots,n\}$ it follows that 
\[
\mean{u^p_i}=\frac{p!}{\prod_{k=0}^{p-1} (n+k)}\,,
\]
where the integration is respect to the probability measure $\rmd^{n}\fu \cdot (n-1)! \delta(t_n(\fu)-1)$.
In particular,
\[
\mean{u_i}=\frac{1}{n}\,,\quad \mean{u^2_i}  = \frac{2}{n(n+1)}\quad \textrm{ and }\quad  \mean{u_i^3}  = \frac{6}{n(n+1)(n+2)}\,.\]
Moreover, for $n\in\mathbb{N}\setminus\{1\}$ and any $i,j\in \{1,\ldots,n\}$ $i\neq j$ it follows that 
\[\mean{u_i u_j}= \frac{1}{n(n+1)}\,.\]
\end{lemma}
\begin{proof}
For $n=1$ the result is obvious. 
We assume $n\in \N \setminus \{1\}$.

The estimates rely on the properties of the measure $\rmd^{n}\fu \cdot (n-1)! \delta(t_n(\fu)-1)$ proven in Appendix \ref{ap:delta}.  In particular, we know that it
is permutation invariant.  Therefore,
we have that $u_{i}$, $i\in \{1,\ldots,n\}$ are identically distributed and then
\begin{equation}\label{eq:identically}
\begin{split}
\mean{u^p_i}&=\mean{u^p_1}\quad\quad \textrm{ for all }  i\in \{1,\ldots,n\} \textrm{ and }  p\in \N\,,\\
\mean{u_iu_j}&=\mean{u_1u_2}\quad  \textrm{ for all } i,j\in \{1,\ldots,n\} \textrm{ satisfying } i\neq j\,.
\end{split}
\end{equation}

Let $p\in \N$.  We next compute
$\mean{u^p_1}$. By Lemma \ref{lem:deltaprop}, then
\begin{equation}
\begin{split}
\mean{u^p_1}=(n-1)\int_{0}^{1}u^p_1(1-u_1)^{n-2}  \rmd u_1\,.
\end{split}
\end{equation}
Recall that the density $\mathsf{Beta}(\alpha,\beta)(\rmd x)$ of the Beta distribution $\mathsf{Beta}(\alpha,\beta)$ with parameters $\alpha>0$ and $\beta>0$ is given by
\begin{equation}\label{eq:betadist}
\mathsf{Beta}(\alpha,\beta)(\rmd x)=
\frac{\Gamma(\alpha+\beta)}{\Gamma(\alpha)\Gamma(\beta)}x^{\alpha-1}(1-x)^{\beta-1}\cf{x\in [0,1]}\rmd x\,.
\end{equation}
As a consequence, 
\begin{equation} 
\begin{split}
(n-1)\int_{0}^{1}u^p_1(1-u_1)^{n-2}  \rmd u_1&=(n-1)\frac{\Gamma(p+1)\Gamma(n-1)}{\Gamma(n+p)}\\
&=\frac{p!(n-1)!}{(n-1)!\prod_{k=0}^{p-1} (n+k)}=\frac{p!}{\prod_{k=0}^{p-1} (n+k)}\,.
\end{split}
\end{equation}
Hence,
\[
\mean{u^p_1}=\frac{p!}{\prod_{k=0}^{p-1} (n+k)}\,.
\]

Next, we continue with the computation of $\mean{u_1u_2}$.
Assume that $n\geq 3$.
By Lemma \ref{lem:deltaprop}, then
\begin{equation}
\begin{split}
\mean{u_1u_2}&
=(n-1)(n-2)
\int_{0}^{1}u_1
\int_{0}^{1-u_1} u_2(1-u_1-u_2)^{n-3}\rmd u_2\, \rmd u_1\,.
\end{split}
\end{equation}
Let $x=\frac{u_2}{1-u_1}$. Then using the fact that $\mathsf{Beta}(\alpha,\beta)(\rmd x)$ given in \eqref{eq:betadist} is probability measure for any $\alpha>0$ and $\beta>0$, we obtain
\begin{equation*}
\begin{split}
\mean{u_1u_2}&=(n-1)(n-2)
\int_{0}^{1}u_1(1-u_1)^{n-1}\rmd u_1
\int_{0}^{1}x(1-x)^{n-3}\rmd x\\
&=(n-1)(n-2)\frac{\Gamma(2)\Gamma(n)}{\Gamma(n+2)}\frac{\Gamma(2)\Gamma(n-2)}{\Gamma(n)}=\frac{1}{n(n+1)}\,.
\end{split}
\end{equation*}
Finally, for $n=2$ recall the constraint $u_1+u_2=1$ and observe that
\[
\mean{u_1u_2}=\mean{u_1(1-u_1)}=\mean{u_1}-\mean{u^2_1}=\frac{1}{n}-\frac{2}{n(n+1)}=\frac{n-1}{n(n+1)}=\frac{1}{6}\,.
\]
\end{proof}

\begin{lemma}[Exponential moments formula for the $\delta$-function]\label{lem:expmoments}
Let $n\in \N\setminus \{1\}$.
For all $i\in \{1,\ldots,n\}$ and $\theta \in \mathbb{R}$ it follows that 
\[
\mean{e^{\theta u_i}}=(n-1) e^{\theta}\int_{0}^{1} x^{n-2}e^{-\theta x}\rmd x\,,
\]
where the integration is respect to the probability measure $\rmd^{n}\fu \cdot (n-1)! \delta(t_n(\fu)-1)$.
In addition, 
\begin{equation}\label{eq:3n}
\mean{e^{(n-1) u_i}}\leq 3\sqrt{n-1}\,.
\end{equation}
\end{lemma}
\begin{proof}
By Lemma \ref{lem:deltaprop} and using \eqref{eq:identically}, we find
\[
\mean{e^{\theta u_1}}=\mean{e^{\theta u_i}}=(n-1)\int_{0}^{1} e^{\theta u_1}(1-u_1)^{n-2}\rmd u_1
\]
for all $i\in \{1,\ldots,n\}$.
Using the change of variable to $x=1-u_1$," we obtain
\[
\mean{e^{\theta u_1}}=(n-1)\int_{0}^{1} e^{\theta (1-x)}x^{n-2}\rmd x=
(n-1)e^{\theta}\int_{0}^{1} e^{-\theta x}x^{n-2}\rmd x\,.
\]
Let $y=\theta x$. Then
\begin{equation}
\begin{split}
\mean{e^{\theta u_1}}&=
\frac{(n-1)e^{\theta}}{\theta^{n-1}}\int_{0}^{\theta} e^{-y}y^{n-2}\rmd y\leq 
\frac{(n-1)e^{\theta}}{\theta^{n-1}}\int_{0}^{\infty} e^{-y}y^{(n-1)-1}\rmd y\\
&=\frac{(n-1)e^{\theta}}{\theta^{n-1}}\Gamma(n-1)=\frac{(n-1)!e^{\theta}}{\theta^{n-1}}\,.
\end{split}
\end{equation}
By \cite{Robbins1955} we have 
\[
j!\leq \sqrt{2\pi}e^{1/12}\sqrt{j}\frac{j^{j}}{e^{j}}\quad \textrm{ for all } j\in \mathbb{N}\,.
\]
Then the choice $\theta=n-1$ yields
\[
\mean{e^{\theta u_1}}\leq 
\frac{(n-1)!e^{n-1}}{(n-1)^{n-1}}\leq \sqrt{2\pi}e^{1/12} \sqrt{n-1}<3\sqrt{n-1}\,.
\]
\end{proof}

\begin{lemma}[Moment estimates for the maximum]\label{lem:maxmoments}
For any measurable and bound observable $f:\mathbb{R}^n_*\to \mathbb{R}$, we denote by $\mean{f}$  the integral of $f$ with respect to the probability measure $\rmd^{n}\fu \cdot (n-1)! \delta(t_n(\fu)-1)$.
Then
there exists a positive constant $C$ such that for all $n\in \mathbb{N}$ it follows that
\begin{equation}\label{eq:ln/n}
\mean{\max_{i\in \{1,\ldots,n\}}u_i}\leq C\frac{\ln(n+1)}{n}\,.
\end{equation}
\end{lemma}
\begin{proof}
The statement for $n=1$ is immediate.
In the sequel, we assume $n\in \mathbb{N}\setminus\{1\}$.
We start with the proof of \eqref{eq:ln/n}.
We use the so-called Chernoff's trick explained in Section~2.5 of \cite{Boucheron2013}.
Recall that the random variables $u_1,\ldots, u_n$ are identically distributed.
By the Jensen inequality for the convex function $f(x)=\exp\left(\theta x\right)$, $x\in \mathbb{R}$ and $\theta>0$ yields
\begin{equation}
\begin{split}
\exp\left(\theta \mean{\max\limits_{i\in \{1,\ldots,n\}}u_i}\right)&=f\left(\mean{\max_{i\in \{1,\ldots,n\}}u_i}\right)\leq \mean{f\left(\max_{i\in \{1,\ldots,n\}}u_i\right)}\\
&=
\mean{\exp\left(\theta\max\limits_{i\in \{1,\ldots,n\}}u_i}\right)=\mean{\max\limits_{i\in \{1,\ldots,n\}}  \exp\left(\theta u_i\right)}\\
&\leq \mean{\sum\limits_{i=1}^n  \exp\left(\theta u_i\right)}=n\mean{\exp\left(\theta u_1\right)}\,.
\end{split}
\end{equation}
Hence, 
\begin{equation}\label{eq:maxfirst}
\mean{\max\limits_{i\in \{1,\ldots,n\}}u_i}\leq \frac{\ln(n)}{\theta}+\frac{\ln\left(\mean{\exp\left(\theta u_1\right)}\right)}{\theta}\quad \textrm{ for all } \theta>0\,.
\end{equation}
Inequality \eqref{eq:maxfirst} for $\theta=n-1$ with the help of
Lemma~\ref{lem:expmoments} in Appendix~\ref{ap:tools} implies
\begin{equation}
\mean{\max\limits_{i\in \{1,\ldots,n\}}u_i}\leq \frac{\ln(n)}{n-1}+\frac{\ln(3\sqrt{n})}{n-1}\leq 
3\frac{\ln(n+1)}{n}+2\frac{\ln(3)}{n}\leq C_1\frac{\ln(n+1)}{n}\,,
\end{equation}
where $C_1:=\max\{6,4\ln(3)/\ln(2)\}$.
\end{proof}

\section{\textbf{Free velocity flip moment estimates}}\label{ap:coupling}
In this section, we estimate the moments of the free velocity flip model via Theorem~1 in \cite{Janssen1990}.
\begin{lemma}[Moment estimates for the free velocity flip model]\label{lem:momentskac}
The following moment estimates are valid.
\begin{enumerate}
\item  The first-moment of $X(t;\fs)$ is given by
\[
\mathbb{E}\left[(X(t;\fs))\right]=
\frac{v_0}{2\lambda}\left(1-e^{-2\lambda t}\right)\quad \textrm{ for all }\quad t\geq 0\,.
\]
\item The second-moment and the variance of $X(t;\fs)$ are given by
\[
\mathbb{E}\left[(X(t;\fs))^2\right]=\frac{v^2_0}{2\lambda^2}\left({2\lambda t}-(1-e^{-2\lambda t}) \right)
\]
and
\[
\mathrm{Var}\left[X(t;\fs)\right]=\frac{v^2_0}{\lambda^2}\left(\lambda t+e^{-2\lambda t}-\frac{e^{-4\lambda t}}{4}-\frac{3}{4}\right)\,.
\]
for all $t\geq 0$, respectively.
\item 
For any $r>0$ there exists a constant $C(r)$ (does not depend on $\lambda$) such that
\begin{equation}\label{eq:momentskacr}
\mathbb{E}\left[|X(t;\fs)|^r\right]\leq |v_0|^{r}\left( \widetilde{C}(r) {\lambda}^{-r/2}t^{r/2}+C(r){\lambda}^{-r/2-1} t^{r/2-1}\right) \quad \textrm{ for all }\quad t\geq 0\,,
\end{equation}
where
\begin{equation}\label{eq:absmomengauss}
\begin{split}
\widetilde{C}(r)&:=\int_{\mathbb{R}} \rmd z\, \varphi(z)|z|^r=
2^{r/2}\frac{\Gamma(\frac{r+1
}{2})}{\sqrt{\pi}}
\quad \textrm{ with }\quad
\varphi(z):=\frac{e^{-z^2/2}}{\sqrt{2\pi}}\mathbbm{1}_{\mathbb{R}}(z)
\end{split}
\end{equation}
and $\Gamma$ denotes the Gamma function. 
The constant $C(r)$ can be estimated from the proof of Theorem~1 in \cite{Janssen1990}.
The following crude bound is also true
\begin{equation*}
\mathbb{E}\left[|X(t;\fs)|^r\right]\leq 
|v_0|^{r} t^r\quad \textrm{ for all }\quad t\geq 0\,.
\end{equation*}
The last bound is good for times $t\ll 1$.
\end{enumerate}
\end{lemma}
\begin{proof}
We start with the proof of Item~(1).
By \eqref{def:path01} and Fubini's theorem we have
\begin{equation}
\begin{split}
\mathbb{E}\left[(X(t;\fs))\right]& =
v_0\int_{0}^{t} \rmd s \,\mathbb{E}\left[(-1)^{N(s;\fs)}\right]
=v_0\int_{0}^{t} \rmd s \,\sum\limits_{k=0}^{\infty}(-1)^{k}e^{-\lambda s}\frac{(\lambda s)^k}{k!}\\
&=v_0\int_{0}^{t} \rmd s \,e^{-2\lambda s}=
\frac{v_0}{2\lambda}\left(1-e^{-2\lambda t}\right)\quad \textrm{ for all }\quad t\geq 0\,,
\end{split}
\end{equation}
which yields Item~(1).
The proof of Item~(2) follows directly from Lemma~1 in \cite{Orsingher1990}.
In the sequel, we prove Item~(3).
To prove Item~(3) we apply the estimate given in Theorem~1 in \cite{Janssen1990}.
For $p=1$, $\delta=1$, $M=1$ in Theorem~1 in \cite{Janssen1990}
we have the existence of a positive constant $C$ such that
\begin{equation}\label{eq:expX}
\begin{split}
\int_\mathbb{R} \rmd x\,e^{|x|}\,
\left |
t^{1/2}h(t,t^{1/2}x,1)-\varphi(x)-t^{-1/2}\frac{x\varphi(x)}{2}
\right |\leq \frac{C}{t}\quad\textrm{ for all }\quad t>0\,,
\end{split}
\end{equation}
where for $t>0$, $\lambda>0$, $x\in \mathbb{R}$ the function $h$ is given by
\begin{equation}
h(t,x,\lambda):=\frac{\lambda}{2}e^{-\lambda t}
\left\{
\mathcal{I}_0\left(\lambda(t^2-x^2)^{1/2}\right)
+\left((t+x)/(t-x)\right)^{1/2}
\mathcal{I}_1\left(\lambda(t^2-x^2)^{1/2}\right)
\right\}\mathbbm{1}_{(-t,t)}(x)\,.
\end{equation}
Here, $\mathcal{I}_{0}$ and $\mathcal{I}_{1}$ denote the modified Bessel functions of order $0$ and $1$, respectively.
In other words, 
\begin{equation}
\mathcal{I}_{0}(z)=\sum\limits_{k=0}^\infty \frac{(z^2/4)^k}{(k!)^2}\quad \textrm{ and } \quad
\mathcal{I}_{1}(z)=\frac{z}{2}\sum\limits_{k=0}^\infty \frac{(z^2/4)^k}{k!(k+1)!}\,\quad
\textrm{ for any }\quad z\in \mathbb{R}\,.
\end{equation}
Let $r>0$ be fixed. 
Note that there exists a positive constant $K_r$ such that
\begin{equation}\label{eq:polyexpo}
|x|^r\leq K_r e^{|x|}\quad \textrm{ for all }\quad x\in \mathbb{R}\,.
\end{equation}
Recall that 
\[
X(t;\fs)=v_0\int_{0}^{t}\rmd s\,(-1)^{N(s;\fs)},
\]
where the entries of $\fs=(s_1,s_2,\ldots,)$ are i.i.d.\ with exponential distribution of parameter $\lambda$.
Denote by $\stackrel{\mathcal{D}}{=}$ equality in distribution.
Note that 
\[
X(t;\fs)\stackrel{\mathcal{D}}{=}v_0\int_{0}^{t}\rmd s\,(-1)^{N(\lambda s;\fu)},
\]
where 
the entries of $\fu=(u_1,u_2,\ldots,)$ are i.i.d.\ with exponential distribution of parameter $1$.
Hence for any $r>0$ we have
\[
\mathbb{E}[|X(t;\fs)|^r]=\frac{1}{\lambda^r}
\mathbb{E}\left[\left|v_0
\int_{0}^{\lambda t}\rmd s\,(-1)^{N(s;\fu)}\right|^r
\right]\,.
\]
Consequently,
\begin{equation}\label{eq:claimscaling}
\int_{\mathbb{R}}\rmd z\,h(t,z,\lambda)|z|^r=\frac{1}{\lambda^r}\int_{\mathbb{R}}\rmd z\,h(\lambda t,z,1)|z|^r\,.
\end{equation}

We continue with the estimate of the $r$-th moment.
The change of variable $z=\widetilde{t}^{1/2}x$
with $\widetilde{t}:=\lambda t$
 yields
\begin{equation}\label{eq:splitd}
\begin{split}
\int_{\mathbb{R}}\rmd z\,h(\lambda t,z,1)|z|^r&=
\int_{\mathbb{R}}\rmd x\,\widetilde{t}^{1/2}h(\widetilde{t},\widetilde{t}^{1/2}x,1)|\widetilde{t}^{1/2}x|^r\\
&=
\int_{\mathbb{R}}\rmd x\,
\left(
\widetilde{t}^{1/2}h(\widetilde{t},\widetilde{t}^{1/2}x,1)-\varphi(x)-\widetilde{t}^{-1/2}\frac{x\varphi(x)}{2}
\right)
|\widetilde{t}^{1/2}x|^r\\
&\qquad +
\int_{\mathbb{R}}\rmd x\,
\left(
\varphi(x)+\widetilde{t}^{-1/2}\frac{x\varphi(x)}{2}
\right)
|\widetilde{t}^{1/2}x|^r\,.
\end{split}
\end{equation}
By \eqref{eq:expX} and \eqref{eq:polyexpo} we have
\begin{equation}\label{eq:cota1}
\begin{split}
\left|
\int_{\mathbb{R}}\rmd x\,
\left(
\widetilde{t}^{1/2}h(\widetilde{t},\widetilde{t}^{1/2}x,1)-\varphi(x)-\widetilde{t}^{-1/2}\frac{x\varphi(x)}{2}
\right)
|\widetilde{t}^{1/2}x|^r\right|\leq \widetilde{t}^{r/2-1}K_rC=
{t}^{r/2-1}{\lambda}^{r/2-1}K_rC\,.
\end{split}
\end{equation}
We also note that
\begin{equation}\label{eq:cota2}
\begin{split}
\int_{\mathbb{R}}\rmd x\,
\left(
\varphi(x)+\widetilde{t}^{-1/2}\frac{x\varphi(x)}{2}
\right)
|\widetilde{t}^{1/2}x|^r=
\int_{\mathbb{R}}\rmd x\,
\varphi(x)
|\widetilde{t}^{1/2}x|^r=\widetilde{C}(r)\widetilde{t}^{r/2}=\widetilde{C}(r){\lambda}^{r/2}{t}^{r/2}\,.
\end{split}
\end{equation}
Relations \eqref{eq:claimscaling} and  \eqref{eq:splitd} with the help of  \eqref{eq:cota1}  and \eqref{eq:cota2} imply Item~(3).
\end{proof}

The next lemma shows that the $p$-th moment and the $p$-th absolute moment are observables that satisfy \eqref{def:lip}. 

\begin{lemma}[Moments and absolute moments]\label{lem:momentabs}
For any $p\in \mathbb{N}$ it follows that 
\begin{equation}\label{eq:momentosw}
| \|x\|^p-\|y\|^p |\leq K\frac{w(x)+w(y)}{2}\|x-y\|\quad \textrm{ for all } x,y\in \mathbb{R}^d\,,
\end{equation}
where $K=2p$ and $w:\mathbb{R}^d\to [0,\infty)$ is given by $w(z)=\|z\|^{p-1}$ for all $z\in \mathbb{R}^d$.
For $d=1$ let $K=2p$ and $w(z)=|z|^{p-1}$ for all $z\in \mathbb{R}$.
The following inequalities are also valid.
\begin{itemize}
\item[i)] For any $p>0$ it follows that
\begin{equation}\label{eq:momentosw1d}
||x|^p-|y|^p |\leq K\frac{w(x)+w(y)}{2}|x-y|\quad \textrm{ for all } x,y\in \mathbb{R}\,,
\end{equation}
\item[ii)] 
For any $p\in \mathbb{N}$ it follows that
\begin{equation}\label{eq:momentosw1d090}
|x^p-y^p |\leq K\frac{w(x)+w(y)}{2}|x-y|\quad \textrm{ for all } x,y\in \mathbb{R}\,,
\end{equation}
\end{itemize}
\end{lemma}
\begin{proof}
We start with the proof of \eqref{eq:momentosw}.
We first note that if $\|x\|=0$ or $\|y\|=0$, then \eqref{eq:momentosw} holds. Hence, we assume that $\|x\|\neq 0$ and $\|y\|\neq 0$.
By the trichotomy property we have the following cases: 
\begin{itemize}
\item[a)] $0\neq \|x\|=\|y\|$.
\item[b)] $0<\|x\|<\|y\|$.
\item[c)] $0<\|y\|<\|x\|$.
\end{itemize}
The case a) is immediate.
We start with the case b).
Let $r:=\|x\|/\|y\|\in (0,1)$ and note that 
\begin{equation*}
\begin{split}
| \|x\|^p-\|y\|^p |&= \|y\|^p-\|x\|^p =\|y\|^p (1-r^p)=\|y\|^p(1-r)\sum\limits_{k=0}^{p-1}r^k\\
&\leq p\|y\|^p (1-r) =p\|y\|^{p-1}(\|y\|-\|x\|)\\
&=
p w(y)|\|x\|-\|y\||\leq K\frac{w(x)+w(y)}{2} \|x-y\|\,.
\end{split}
\end{equation*}
The case c) is completely analogous to the case b).
The proof of \eqref{eq:momentosw} is complete.

We continue with the proof of \eqref{eq:momentosw1d}. 
For $p=1$ it is easy to see  that \eqref{eq:momentosw1d} holds.
In the sequel, we assume that $p\neq 1$.
Let $x,y\in \mathbb{R}$ be fixed.
For $|x|=|y|$ it is easy to see that \eqref{eq:momentosw1d} holds.
Without loss of generality, we assume that $|x|\neq |y|$.
By the Fundamental Theorem of Calculus we have
\begin{equation*}
\begin{split}
||y|^p-|x|^p|&=\left |p\int_{|x|}^{|y|}\rmd z \,z^{p-1}\right|\leq 
\left |p\int_{|x|}^{|y|}\rmd z \,|z|^{p-1}\right|\leq 
p\max\limits_{|x|\wedge |y|\leq z\leq |x|\vee |y|}w(z)||x|-|y||\\
&
\leq p(w(x)+w(y))||x|-|y||\leq 
K\frac{w(x)+w(y)}{2}|x-y|
\,,
\end{split}
\end{equation*}
where the third inequality follows from the fact that the function $(0,\infty)\ni z \mapsto w(z)$ is increasing for $p>1$ and decreasing for $p\in (0,1)$.
The proof of \eqref{eq:momentosw1d} is complete.

The proof of \eqref{eq:momentosw1d090} follows from an analogous reasoning as used in the proof of \eqref{eq:momentosw1d}.
\end{proof}
Lemma~\ref{lem:momentabs} in Appendix~\ref{ap:coupling} yields that 
for the observables ($p$-th moment and the $p$-th absolute moment),
the weighted function $w$ that appears in \eqref{def:lip} can be chosen as $w(z)=\|z\|^{p-1}$, $z\in \mathbb{R}^d$.

In the sequel, we estimate the constants $C_1$ and $C_2$ (for $w(z)=|z|^{p-1}$, $z\in \mathbb{R}$) that appears in \eqref{eq:Cdef}, i.e.,
\begin{equation}\label{eq:Cdef78}
\begin{split}
&C_1^2:=\int \mu_1(\rmd X) \frac{1}{T}\int_{0}^{T}\rmd s\, (w(X(s)))^2<\infty\,,\\
&
C_2^2:=\int \mu_2(\rmd Y) \frac{1}{T} \int_{0}^{T}\rmd s\, (w(Y(s)))^2 <\infty\,,
\end{split}
\end{equation}
for
 the processes $X$ and $Y$ being replaced by the rescaled telegraph process $L^{-1}X$ and a Brownian motion $B$ with a suitable diffusivity constant $\sigma^2$, respectively.
Hence the constant $\max\{C_1,C_2\}$, which is needed in the estimate given in the right-hand side of \eqref{eq:momentCs}, is also estimated.
\begin{lemma}[Time average estimates for the  $w$-weight]\label{lem:averagew}
For $p=1$ it follows  that $C^2_1=C^2_2=1$.
For any $p>0$ with $p\neq 1$,
the following estimates are valid
\[
C^2_1\leq 
\frac{\widetilde{C}T^{p-1}_*}{p L^{2(p-1)}_*}
+
\frac{{C}T^{p-2}_*}{(p-1) L^{2(p-1)}_*}\quad \textrm{ and } \quad
C^2_2=\frac{2^{p-1}\Gamma(\frac{2p-1
}{2})}{p\sqrt{\pi}}\frac{T^{p-1}_*}{L^{2(p-1)}_*}\,,
\]
where the constants 
$T_*$ and $L_*$ are the scaling parameters that appears in  \eqref{e:scalings},
$\widetilde{C}$ and $C$ are the constants that appears for $r=2(p-1)$ in Item~(3), \eqref{eq:momentskacr},
of Lemma~\ref{lem:momentskac} in Appendix~\ref{ap:coupling} and $\Gamma$ denotes the Gamma function.
\end{lemma}
\begin{proof}
We observe that $C^2_1=C^2_2=1$ for $p=1$.
We start with the estimate of $C^2_1$ for $p\neq 1$.
In this case, the process $X$ is replaced by $L^{-1}X$ in \eqref{eq:Cdef78}.
A direct computation yields 
\begin{equation}\label{eqt:p}
\begin{split}
C_1^2&=\int \mu_1(\rmd X) \frac{1}{T}\int_{0}^{T}\rmd s\, (w(L^{-1}X(s;\fs)))^2=\int \mu_1(\rmd X) \frac{1}{T}\int_{0}^{T}\rmd s\, |L^{-1}X(s;\fs)|^{2(p-1)} \\
&=\frac{1}{L^{2(p-1)}T}\int_{0}^{T}\rmd s\,\mathbb{E}[|X(s;\fs)|^{2(p-1)}]\leq 
\frac{|v_0|^{2(p-1)}}{L^{2(p-1)}T}\int_{0}^{T}\rmd s\,\big(
\widetilde{C} {\lambda}^{-(p-1)}s^{p-1}+C{\lambda}^{-p} s^{p-2}
\big)\\
&=
\frac{|v_0|^{2(p-1)}\widetilde{C}{\lambda}^{-(p-1)}T^{p-1}}{p L^{2(p-1)}}
+
\frac{|v_0|^{2(p-1)}{C}{\lambda}^{-p}T^{p-2}}{(p-1) L^{2(p-1)}}
=
\frac{\widetilde{C}T^{p-1}_*}{p L^{2(p-1)}_*}
+
\frac{{C}T^{p-2}_*}{(p-1) L^{2(p-1)}_*}
\,,
\end{split}
\end{equation}
where the constants  $\widetilde{C}$ and $C$ are the constants that appears for $r=2(p-1)$ in Item~(3), \eqref{eq:momentskacr}, 
of Lemma~\ref{lem:momentskac} in Appendix~\ref{ap:coupling}.

We continue with the estimate of $C^2_2$. 
Since the process  $Y$ in \eqref{eq:Cdef78} is the Brownian motion
with diffusivity constant $\sigma^2=\lambda L^{-2}_*$ given in \eqref{eq:diffusivity},
we have 
\begin{equation}
\begin{split}
C_2^2&=\int \mu_2(\rmd Y) \frac{1}{T} \int_{0}^{T}\rmd s\, (w(B(s)))^2=\int \mu_2(\rmd Y) \frac{1}{T} \int_{0}^{T}\rmd s\, |B(s)|^{2(p-1)}\\
&=\frac{1}{T}\int_{0}^{T}\rmd s\, \mathbb{E}[|B(s)|^{2(p-1)}]
=\mathbb{E}[|\sigma B(1)|^{2(p-1)}] \frac{1}{T}\int_{0}^{T}\rmd s\, s^{p-1}\\
&=\frac{\sigma^{2(p-1)}2^{p-1}\Gamma(\frac{2p-1
}{2})}{\sqrt{\pi}}\frac{T^{p-1}}{p}=
\frac{2^{p-1}\Gamma(\frac{2p-1
}{2})}{p\sqrt{\pi}}\frac{T^{p-1}_*}{L^{2(p-1)}_*}\,,
\end{split}
\end{equation}
where in the last inequality we used relation \eqref{eq:absmomengauss} given in Item~(3) of Lemma~\ref{lem:momentskac} in Appendix~\ref{ap:coupling}.
\end{proof}

In the following we estimate the distance between $L^{-1}\mathbb{X}^{\fs}_{[0,T]}=(L^{-1}X(t;\fs):0\leq t\leq T)$ and $\mathbb{B}_{[0,T]}=(B(t):0\leq t\leq T)$ by the independent coupling between them.
\begin{lemma}[Independent coupling bound]\label{lem:indcoup}
For any $L>0$, $T>0$, $\lambda>0$ and $v_0\in \mathbb{R}\setminus\{0\}$ it follows that
\begin{equation}
\cW_2\left(L^{-1}\mathbb{X}^{\fs}_{[0,T]},\mathbb{B}_{[0,T]}\right)\leq \left(\frac{1}{4T_* L^2_*}
\left(1-e^{-2 T_*}\right)
-\frac{1}{2L_*^2}+
\frac{T_*}{L^2_*}\right)^{{1}/{2}}\,,
\end{equation}
where $T_*=\lambda T$ and ${L}_*=|v_0|^{-1}\lambda L$.
\end{lemma}
\begin{proof}
We estimate the following expectation
\begin{equation}
\mathbb{E}_{\gamma}\left[\frac{1}{T}\int_{0}^{T} \rmd t \, \left({L}^{-1}{X(t;\fs)}-B(t)\right)^2\right]
\end{equation}
for a suitable coupling $\gamma$  between  $L^{-1}\mathbb{X}^{\fs}_{[0,T]}$ and 
$\mathbb{B}_{[0,T]}$.
We recall that the diffusivity constant of $B$ is given by $
\sigma^2=\lambda L^{-2}_*$ given in \eqref{eq:diffusivity}.
By Fubini's theorem and Pythagoras's theorem we obtain
\begin{equation}\label{eq:fubpiyt}
\begin{split}
\mathbb{E}_{\gamma}&\left[\frac{1}{T}\int_{0}^{T} \rmd t \, \left({L}^{-1}{X(t;\fs)}-B(t)\right)^2\right]=
\frac{1}{T}\int_{0}^{T} \rmd t \,
\mathbb{E}_{\gamma}\Big[ \left({L}^{-1}{X(t;\fs)}-B(t)\right)^2\Big]\,\\
&\hspace{3cm}=\frac{1}{T}\int_{0}^{T} \rmd t\, 
\big( 
\mathbb{E}_{\gamma}\big[\left({L}^{-1}{X(t;\fs)}\right)^2\big]+\mathbb{E}_{\gamma}\big[(B(t))^2\big]\big)-
\frac{2}{T}\int_{0}^{T} \rmd t\, 
\mathbb{E}_{\gamma}\big[{L}^{-1}{X(t;\fs)} B(t)\big].
\end{split}
\end{equation}
By Lemma~\ref{lem:momentskac} in 
Appendix~\ref{ap:coupling}  we have
\begin{equation}\label{eq:momento2gral}
\mathbb{E}_{\gamma}\left[(X(t;\fs))^2\right]=\frac{v^2_0}{2\lambda^2}\left({2\lambda t}-(1-e^{-2\lambda t}) \right)\quad\textrm{ for all }\quad t\geq 0\,
\end{equation}
and for any coupling $\gamma$.
Since $(B(t):t\geq 0)$ is a Brownian motion with diffusivity $\sigma^2$, the second moment of its marginal $B(t)$ is given by
\begin{equation}\label{eq:momento2Brow}
\mathbb{E}_{\gamma}[(B(t))^2 ]=\sigma^2 t\quad
\textrm{ for all }\quad t\geq 0\,
\end{equation}
for any coupling $\gamma$.
Next, we estimate the contribution of the term
$\mathbb{E}_\gamma [{L}^{-1}{X(t;\fs)} B(t)]$
for a suitable coupling $\gamma$.
By the use of the independent coupling $\gamma=\pi$ we obtain  
\begin{equation}\label{eq:momentcross}
\mathbb{E}_{\pi}[{L}^{-1}{X(t;\fs)} B(t)]=
\mathbb{E}_\pi[{L}^{-1}{X(t;\fs)}]
\mathbb{E}_\pi\left[B(t)\right]=0\,.
\end{equation}
Combining \eqref{eq:momento2gral}, \eqref{eq:momento2Brow} and \eqref{eq:momentcross} in \eqref{eq:fubpiyt}  implies
\begin{equation}\label{e:reformula}
\begin{split}
\inf_{\gamma }\,\,
\mathbb{E}_{\gamma}&\left[\frac{1}{T}\int_{0}^{T} \rmd t \, \left({L}^{-1}{X(t;\fs)}-B(t)\right)^2\right]\leq \mathbb{E}_{\pi}\left[\frac{1}{T}\int_{0}^{T} \rmd t \, \left({L}^{-1}{X(t;\fs)}-B(t)\right)^2\right]\\
&\hspace{3cm}=\frac{1}{T}\int_{0}^{T} \rmd t\, 
\big( 
\mathbb{E}_{\pi}\big[\big({L}^{-1}{X(t;\fs)}\big)^2\big]+\mathbb{E}_{\pi}\big[(B(t))^2\big]\big)\\
&\hspace{3cm}=\frac{1}{T}\int_{0}^{T} \rmd t\, 
\left( 
\frac{v^2_0}{2\lambda^2 L^2}\left({2\lambda t}-1+e^{-2\lambda t} \right)+\sigma^2 t\right)\\
&\hspace{3cm}=\frac{1}{T}\left(\frac{v^2_0}{2\lambda^2L^2}\lambda T^2-
\frac{v^2_0}{2\lambda^2L^2}
T
+\frac{v^2_0}{2\lambda^2L^2}\frac{1}{2\lambda}(1-e^{-2\lambda T})+
\frac{\sigma^2 T^2}{2}
\right)\\
&\hspace{3cm}=
\frac{T_*}{2L^2_*}-\frac{1}{2L^2_*}+\frac{1}{4 L^2_* T_*}(1-e^{-2T_*})+
\frac{1}{2}\sigma^2 T\,,
\end{split}
\end{equation}
where $T_*=\lambda T$ and ${L}_*=|v_0|^{-1}\lambda L$. Since $\sigma^2=\lambda L^{-2}_*$, \eqref{e:reformula} yields
\begin{equation}
\cW^2_2\left(L^{-1}\mathbb{X}^{\fs}_{[0,T]},\mathbb{B}_{[0,T]}\right)\leq \frac{1}{4T_* L^2_*}
\left(1-e^{-2 T_*}\right)
-\frac{1}{2L_*^2}+
\frac{T_*}{L^2_*}\,,
\end{equation}
which yields the statement.
\end{proof}

For completeness of the presentation, we recall 
the Koml\'os--Major--Tusn\'ady coupling (see Theorem~1 of \cite{Komlos1976}) and we state it as a lemma.

\begin{lemma}[Koml\'os--Major--Tusn\'ady Theorem]\label{lem:KMT}
Let $F:\mathbb{R}\to [0,1]$ be a distribution function satisfying 
\[
\int_{\mathbb{R}} e^{r z}F(\rmd z)<\infty\quad \textrm{ for any } \quad r\in (-r_0,r_0)\,
\]
for some positive number  $r_0>0$. 
In addition, we assume
\[
\int_{\mathbb{R}} zF(\rmd z)=0\quad \textrm{ and } \quad 
\int_{\mathbb{R}} z^2F(\rmd z)=1\,.
\]
Then there exists a probability space $(\Omega_*,\mathcal{F}_*,\pi^*)$ for which we can construct
sequences $X:=(X_n:n\in \mathbb{N})$ and $Y:=(Y_n:n\in \mathbb{N})$ of random variables 
satisfying
\begin{itemize}
\item[a)]  $X$ is i.i.d.\ with common distribution function $F$,
\item[b)]   $Y$ is i.i.d.\ with standard Gaussian distribution
\end{itemize}
in a way that for all $x>0$ and every $n\in \mathbb{N}$ the following estimate is valid
\[
\pi^*\left(\max\limits_{1\leq m\leq n}\left|\sum_{j=1}^m X_j-\sum_{j=1}^m Y_j\right| \geq C\ln(n)+x\right)\leq K e^{-\vartheta x}\,.
\]
Here, the positive constants $C,K,\vartheta $ depend only on $F$, and $\vartheta$ can be taken as large as desired by choosing $C$ large enough. Consequently, $|\sum_{j=1}^n X_j-\sum_{j=1}^n Y_j|=\mathcal{O}(\ln(n))$ almost surely.
\end{lemma}

\section{\textbf{Definition and basic properties of the $\delta$-constraints}}\label{ap:delta}

In the text,  we frequently integrate with respect to measures which are identified using the somewhat formal Dirac $\delta$-function notation.  As explained in Appendix~A in \cite{Lukkarinen2019}, this slightly formal notion can often be given a fully rigorous definition as a positive Radon measure obtained via the 
Riesz--Markov--Kakutani representation theorem applied to a limit of a sequence of approximations where the $\delta$-function is replaced by a suitably chosen positive ordinary function.  The benefit of using the approximations becomes apparent when one needs to perform operations which are standard to Lebesgue measures, such as changes of integration variables or splitting with aid of Fubini's theorem.

In this appendix, we explain the definition and derive the properties used in the text related to integration with respect to the probability measure
 $\nu(\rmd^n \fu)$ which is defined on $\mathbb{R}_*^n$, $n\in \mathbb{N}\setminus\{1\}$, by the formula
\[
\nu(\rmd^n \fu)=\rmd^{n}\fu \, (n-1)! \delta(t_n(\fu)-1)\,,
\]
where  $t_n(\fu)=\sum_{j=1}^{n}u_j$ and
$\rmd^n \fu$ denotes the restriction of the 
Lebesgue measure to $\mathbb{R}^n_*$.  These measures are obtained by scaling from the more general family of positive measures
\[
\tilde{\nu}_r(\rmd^n \fu)=\rmd^{n}\fu \, \delta(t_n(\fu)-r)\,, \qquad r>0\,,
\]
since $\nu(\rmd^n \fu) = (n-1)! \tilde{\nu}_1(\rmd^n \fu)$.

The most direct definition of the measure $\tilde{\nu}_r$ is obtained by formal integration of the $\delta$-function over the last coordinate, $u_n$. We define a positive semidefinite
linear functional $\Lambda(\cdot;r)$ on $C_c(\mathbb{R}_*^n)$ by setting for any continuous function $f:\R_*^n \to \C$ which has a compact support,
\begin{align}\label{eq:defLambdafunct}
 \Lambda(f;r) :=   \int_{X'(n,r)}\rmd^{n-1}u\,
 f\Bigl(u,r-\sum_{j=1}^{n-1}u_j\Bigr)\,,
\end{align}
where $X'(n,r):= \{ u\in \R_*^{n-1} : \sum_{j=1}^{n-1}u_j<r\}$ is a bounded Borel set.
Since $\mathbb{R}_*^n$ is a locally compact Hausdorff space, the Riesz--Markov--Kakutani representation theorem implies that there exists a unique Radon measure $\tilde{\nu}_r$
for which $ \Lambda(f;r)=\int_{\mathbb{R}_*^n}\tilde{\nu}_r(\rmd^n \fu) f(\fu )$.
Now 
$K_m := \{ u\in \R_*^n : \frac{1}{m}\le u_j \le m\text{ for all }j\}$, $m\in \N$,
form an increasing covering sequence of compact subsets of $\R_*^n$.  Thus by Urysohn's lemma, for each $m\in \N$, we can find a continuous function $\phi_m:\R_*^n \to \R$ such that $0\le \phi_m \le 1$ and $\phi_m(u)=1$, if $u\in K_m$, and $\phi_m(u)=0$, if $u\not\in K_{m+1}$.
Then each $\phi_m$ has a compact support and, by the dominated convergence theorem, we find that $\Lambda(\phi_m;r)\to \int_{X'(n,r)}\rmd^{n-1}u \,1$, as $m\to \infty$.
Here, 
by induction in $n$, one can check the well-known simplex volume result,
\begin{align}\label{eq:simplexvol}
&
 \int_{X'(n,r)}\rmd^{n-1}u \, 1 =
  \int_{\R_*^{n-1}}\rmd^{n-1} u\,\cf{\sum_{i=1}^{n-1}u_j<r} = \frac{r^{n-1}}{(n-1)!}\,, 
\end{align}
which is valid for all $r>0$ and $n\in \N$ with $n\ge 2$. Since clearly also $|\Lambda(f;r)|\le \frac{r^{n-1}}{(n-1)!}\Vert f\Vert_\infty$ for any $f\in C_c(\mathbb{R}_*^n)$, we find that $\int_{\mathbb{R}_*^n}\tilde{\nu}_r(\rmd^n \fu)\,1=\frac{r^{n-1}}{(n-1)!}$.  Therefore, each
$\tilde{\nu}_r$ is a bounded, positive semidefinite Radon measure, and $\nu(\rmd^n \fu)=(n-1)!\tilde{\nu}_1(\rmd^n \fu)$ is a Radon probability measure on $\mathbb{R}_*^n$. 

The following technical lemma shows that the formula \eqref{eq:defLambdafunct} indeed applies to a much larger class of functions.  In addition, it also gives a convenient approximation result by ordinary integrals for integration of continuous functions. This can be used to derive many further properties about integrals over $\tilde{\nu}_r(\rmd^n \fu)$,
similar to what is done in Appendix~A in \cite{Lukkarinen2019}.   We derive the main properties used in the present work from these, as summarized later in 
Lemma~\ref{lem:deltaprop}.
\begin{lemma}\label{lem:Borelprop}
 Suppose $n\ge 2$ and  $r>0$.  Denote
 $E_r:=\{\fu \in \R_*^n : t_n(\fu)\ne r\}$.
 Then all of the statements below hold.
\begin{enumerate}
 \item  $\tilde{\nu}_r(E_r)=0$.
 \item If $A\subset \R_*^{n-1}$ has Lebesgue measure zero, then $\tilde{\nu}_r(A\times \R_*)=0$.
 \item  Suppose
 $f:\R_*^{n}\to \C$ is a Borel measurable function.
 If $f$ is non-negative or $f\in L^1(\tilde{\nu}_r)$,
 we have
 \begin{align}\label{eq:mainnurtech}
  \int_{\mathbb{R}_*^n}\tilde{\nu}_r(\rmd^n \fu) f(\fu )
  =
   \int_{\R_*^{n-1}}\rmd^{n-1} u\,\cf{t_{n-1}(u)<r} f(u,r-t_{n-1}(u))
   = \frac{r^{n-1}}{(n-1)!}   \int_{\mathbb{R}_*^n}\nu(\rmd^n \fu) f(r \fu )
\,.  
 \end{align}
 \item Suppose $\varphi:\R \to \R$ is a continuous function with a compact support 
for which $\varphi \ge 0$ and $\int_{\R}\rmd x \,\varphi(x)=1$.  Define $\Phi_\varepsilon(\fu;r) :=  \varepsilon^{-1}\varphi\!\left((t_n(\fu)-r)/\varepsilon\right)$ for $\fu\in \R_*^n$ and $\varepsilon>0$.  Then for any function 
$F:[0,\infty)^{n}\to \C$ which is continuous, we have
\[
  \int_{\mathbb{R}_*^n}\tilde{\nu}_r(\rmd^n \fu) F(\fu )
  = \lim_{\varepsilon\to 0^+} \int_{\R_*^n}\!\rmd^{n}\fu \,
 \Phi_\varepsilon(\fu;r) F(\fu )\,.
\]
\item The measure $\tilde{\nu}_r$ is permutation invariant: if $P$ is a permutation of $\set{1,2,\ldots,n}$, then for any measurable function $f$ the function $\fu \mapsto f(\fu_P)$, $(\fu_P)_j := u_{P(j)}$, $j=1,2,\ldots,n$, has the same integral as $f$.
\end{enumerate}
\end{lemma}
\begin{proof}
For the first item, consider the set $S=\R_*^n \setminus E_r$.
We choose a cutoff function $\eta:\R\to\R$ which is continuous, $0\le \eta \le 1$, and $\eta(x)=1$, if $|x|\le 1$, and $\eta(x)=0$, if $|x|\ge 2$.
Recall the earlier cutoff functions $\phi_m$, and define
$F_m(\fu) := \phi_m(\fu) \eta(m(t_n(\fu)- r))$.  If $\fu \in S$, then 
$F_m(\fu) = \phi_m(\fu) \to 1$ as $m\to \infty$.
If $\fu\not \in S$, then $\eta(m(t_n(\fu)- r))\to 0$ as $m\to \infty$, and thus also $\lim_{m\to\infty}F_m(\fu) =0$. Therefore, by the  dominated convergence theorem, 
$\tilde{\nu}_r(S)=\lim_{m\to \infty}  \int_{\R_*^{n}}\tilde{\nu}_r(\rmd^n \fu) F_m(\fu)$.
Since $F_m\in C_c(\mathbb{R}_*^n)$, we have
\[
 \int_{\R_*^{n}}\tilde{\nu}_r (\rmd^n \fu) F_m(\fu) = 
 \int_{X'(n,r)}\rmd^{n-1}u\,
 \phi_m\Bigl(u,r-\sum_{j=1}^{n-1}u_j\Bigr) \eta(0)=
  \int_{\R_*^{n}}\tilde{\nu}_r (\rmd^n \fu) \phi_m(\fu) \,,
\]
which goes to $\tilde{\nu}_r(\R_*^{n})<\infty$ as $m\to\infty$.
Therefore, $S$ has full measure, and thus its complement $E_r$ must have measure zero. This implies that for any bounded measurable function $f$ we have
\[
  \int_{\mathbb{R}_*^n}\tilde{\nu}_r(\rmd^n \fu) f(\fu )
  = \int_{S}\tilde{\nu}_r(\rmd^n \fu) f(\fu )
  =  \int_{S}\tilde{\nu}_r(\rmd^n \fu) \,f(u,r-t_{n-1}(u))|_{u=P_{n-1}\fu}
\]
where $P_{n-1}\fu := (u_1,u_2,\ldots,u_{n-1})\in \R_*^{n-1}$. 

For the next item, let us first suppose that
$A\subset \R_*^{n-1}$ is open.  Then it can be written as a limit of an increasing sequence of compact sets. 
Applying Urysohn's lemma with open sets $A\times (0,m+1)$ and an increasing sequence of compact sets obtained  by taking a product with the above sets and $[1/m,m]$, we find a sequence of functions $f_m\in C_c(\R_*^{n})$ which pointwise monotone increase to the indicator function of $A\times \R_*$.
Thus by using the monotone convergence theorem,
we conclude that $\tilde{\nu}_r(A\times \R_*)= \int_{A}\rmd^{n-1} u$.  If $A\subset \R_*^{n-1}$ has zero Lebesgue measure, it can be covered with countably many open sets so that the sum of their measures is arbitrarily small.
Hence, we conclude that then $\tilde{\nu}_r(A\times \R_*)=0$.

For the third item, we construct a suitable regularisation so that the defining functional can be used to evaluate the integral.  
In fact, we only need to do this assuming that  $f:\R_*^{n}\to \C$ is bounded and Borel measurable.
Namely, suppose that \eqref{eq:mainnurtech} holds for such $f$.  If $f$ is non-negative, we can then apply the result to every $f_N(\fu) := \cf{f(\fu)\le N} f(\fu)$, $N\in \N$.
This is a monotone non-decreasing sequence converging to $f$, and thus the monotone convergence theorem may be applied to each of the three integrals in \eqref{eq:mainnurtech}.
This shows that \eqref{eq:mainnurtech} holds for all measurable, non-negative functions.
Finally, if $g\in L^1(\tilde{\nu}_r)$, then $f=|g|$ is a non-negative function for which 
the first integral in \eqref{eq:mainnurtech} is finite.
Thus by the previous result also the other two integrals are then finite.  Therefore, $f$ is a dominant for the 
function sequence $g_N$, defined by
$g_N(\fu) := \cf{|g(\fu)|\le N} g(\fu)$, $N\in \N$.
Since each $g_N$ is bounded and 
$g_N\to g$ pointwise, dominated convergence theorem, applied to each of the three integrals separately, implies that \eqref{eq:mainnurtech} holds also for $g$.

Let us thus assume that
$f:\R_*^{n}\to \C$ is bounded and Borel measurable.
We then choose continuous cutoff functions $\eta_\delta:\R \to \R$, $0<\delta\le 1$, which satisfy $0\le \eta_\delta\le 1$, $\eta_\delta(x)=0$ if $|x|\ge 1+\delta$, and $\eta_\delta(x)=1$ if $|x|\le 1$.  For $\varepsilon>0$, we  define 
\[
 \chi_{\delta,\varepsilon}(x) := \eta_\delta(|x|/\varepsilon)\,,\qquad x\in \R^{n-1}\,,
\]
and set, for all $x\in \R^{n-1}$,
\[
 h(x) := \begin{cases}
 f(x,r-t_{n-1}(x))\,, & \text{if } x\in \R_*^{n-1}\text{ and }t_{n-1}(x)<r\,,\\
 0\,, & \text{otherwise}\,.
         \end{cases}
\]
Then $h$ is bounded, $h\in L^1(\R^{n-1})$ and 
\[
 \int_{\R^{n-1}}\rmd^{n-1} x\, h(x) =
  \int_{\R_*^{n-1}}\rmd^{n-1} u\,\cf{t_{n-1}(u)<r} f(u,r-t_{n-1}(u))\,.
\]
The regularisation we need will be provided by the functions
\[
 h_{\delta,\varepsilon}(x) := 
\frac{1}{Z_{\delta,\varepsilon}} \int_{\R^{n-1}}\rmd^{n-1} y\,
 \chi_{\delta,\varepsilon}(x-y) h(y)\,,
\]
where
\[
 Z_{\delta,\varepsilon} := \int_{\R^{n-1}}\rmd^{n-1} y\,
 \chi_{\delta,\varepsilon}(y)<\infty\,.
\]
By applying the Lebesgue dominated convergence theorem,
we then find that each $h_{\delta,\varepsilon}$ is  continuous function on $\R^{n-1}$.  Also, $\Vert h_{\delta,\varepsilon}\Vert_\infty\le \Vert f\Vert_\infty$ 
and the support of each $h_{\delta,\varepsilon}$ is contained in the compact set $\set{x\in [0,\infty)^{n-1} : |x|\le 2 + r}$.

Let $X_0\subset \R^{n-1}$ collect all Lebesgue points of $h$. Since $h\in L^1(\R^{n-1})$, we know that its complement has Lebesgue measure zero.  By definition, for every $x\in X_0$ we then have
\[
 \lim_{\vep \to 0^+} \frac{1}{|B_\vep|}
 \int_{B_\vep}\rmd y \left|h(x-y)-h(x)\right|=0\,,
\]
where $B_\vep$ denotes the closed ball with a radius $\vep$, centred at the origin.  Now, if $\vep>0$,  
\[
 \lim_{\delta\to 0^+}\chi_{\delta,\varepsilon}(y)
 = \cf{|y|\le \vep}\,, 
\]
and thus by applying the dominated convergence theorem twice, we find that for all $x$
\[
\lim_{\delta \to 0^+} h_{\delta,\varepsilon}(x)=
\frac{1}{|B_\vep|}
 \int_{B_\vep}\rmd y\, h(x-y)\,.
\]
Therefore, 
\[
 \lim_{\vep \to 0^+}\left(\lim_{\delta \to 0^+} h_{\delta,\varepsilon}(x)\right)=h(x)\,, \qquad x\in X_0\,.
\]
For notational convenience, let us denote from now on the above double limit by ``$\lim_{\vep,\delta}$''.
Finally, we recall the increasing sequence of cutoff functions $\phi_m$ used earlier.  For $m\in \N$ and $\fu\in \R_*^n$ we define
\[
 f_{m,\delta,\varepsilon}(\fu)
 := h_{\delta,\varepsilon}(P_{n-1}\fu) \phi_m(\fu)\,,
\]
where the projection $P_{n-1}$ was defined above. By construction, $f_{m,\delta,\varepsilon} \in C_c(\mathbb{R}_*^n)$, and thus
\begin{align}\label{eq:hmdeeqs}
& \int_{\mathbb{R}_*^n}\tilde{\nu}_r(\rmd^n \fu) 
 f_{m,\delta,\varepsilon}(\fu)
 = \int_{\R_*^{n-1}}\rmd^{n-1} u\,\cf{t_{n-1}(u)<r} 
  f_{m,\delta,\varepsilon}(u,r-t_{n-1}(u))
\\ & \quad\nonumber 
  = \int_{\R_*^{n-1}}\rmd^{n-1} u\,\cf{t_{n-1}(u)<r} 
  h_{\delta,\varepsilon}(u) \phi_m(u,r-t_{n-1}(u))\,. 
\end{align}
The dominated convergence theorem allows us to conclude that
the limit of the right hand side is given by
\[
 \int_{\R_*^{n-1}}\rmd^{n-1} u\,\cf{t_{n-1}(u)<r}
 h(u) =  \int_{\R_*^{n-1}}\rmd^{n-1} u\,\cf{t_{n-1}(u)<r} f(u,r-t_{n-1}(u))\,.
\]
Again, by the dominated convergence theorem, the limit of the left-hand side of \eqref{eq:hmdeeqs}, as $m\to \infty$, is given by
$\int_{\mathbb{R}_*^n}\tilde{\nu}_r(\rmd^n \fu) h_{\delta,\varepsilon}(P_{n-1}\fu)$.  Furthermore,
\[
 \lim_{\vep,\delta}
 \int_{X_0\times \mathbb{R}_*}\tilde{\nu}_r(\rmd^n \fu) h_{\delta,\varepsilon}(P_{n-1}\fu)
 =  \int_{X_0\times \mathbb{R}_*}\tilde{\nu}_r(\rmd^n \fu) h(P_{n-1}\fu) =
 \int_{X_0\times \mathbb{R}_*}\tilde{\nu}_r(\rmd^n \fu) f(\fu)\,,
\]
where in the last equality we used the observation that if $t_n(\fu)=r$, then also $t_{n-1}(P_{n-1}\fu)<r$ and thus
$f(\fu)=f(P_{n-1}\fu,r-\sum_{j=1}^{n-1}u_j)=h(P_{n-1}\fu)$.
However, by the second item, now $\tilde{\nu}_r(X_0^c\times \mathbb{R}_*)=0$, and thus
\[
 \lim_{\vep,\delta}
 \int_{X_0\times \mathbb{R}_*}\tilde{\nu}_r(\rmd^n \fu) h_{\delta,\varepsilon}(P_{n-1}\fu) =
  \int_{\mathbb{R}_*^n}\tilde{\nu}_r(\rmd^n \fu) f(\fu)\,. 
\]
We can conclude that the first equality in \eqref{eq:mainnurtech} holds for bounded Borel measurable $f$.  But then for any such function $f$, we also have
\begin{align*}
 &  \frac{1}{(n-1)!}\int_{\mathbb{R}_*^n}\nu(\rmd^n \fu) f(r\fu )
 =   \int_{\mathbb{R}_*^n}\tilde{\nu}_1(\rmd^n \fu) f(r\fu)
= 
  \int_{\R_*^{n-1}}\rmd^{n-1} u\,\cf{t_{n-1}(u)<1} f(r u,r-r t_{n-1}(u))\,.
\end{align*}
A change of variables to $v = r u$ in the last integral yields
\[
r^{-(n-1)}
\int_{\R_*^{n-1}}\rmd^{n-1} v\,\cf{t_{n-1}(v)<r} f(v,r-t_{n-1}(v))
= r^{-(n-1)} \int_{\mathbb{R}_*^n}\tilde{\nu}_r(\rmd^n \fv) f(\fv)\,.
\]
Therefore, 
\eqref{eq:mainnurtech} holds for all bounded Borel measurable $f$.  This completes the proof of the third item.

We next prove the fourth statement.
Given $0<\varepsilon\le 1$, let us denote
\[
  I_\varepsilon =  \int_{\R_*^n}\!\rmd^{n}\fu \,
 \Phi_\varepsilon(\fu) F(\fu )\,.
\]
Since $\varphi$ has a compact support, there is $R>0$ such that $\varphi(x)=0$ for all $|x|\ge R$.  Thus $\Phi_\varepsilon(\fu)=0$ if $t_n(\fu)\ge r+R$, using the assumption $\varepsilon\le 1 $.  Denote $Y:=\{x\in [0,\infty)^n : t_n(x)\le r+R\}$ and $Y':= Y\cap \R_*^n$.  
Since $F$ is continuous and $Y$ is compact, $M:=\sup_{x\in Y}|F(x)|<\infty$.  We thus find that $ \Phi_\varepsilon(\fu) |F(\fu )|\le M\Phi_\varepsilon(\fu)$.

We can thus  use Fubini's theorem to reorder the integrals, resulting in
\[
  I_\varepsilon = \int_{\R_*^{n-1}}\!\rmd^{n-1} u\left(
 \int_0^\infty \rmd u_n \,
 \varepsilon^{-1}\varphi\!\left((u_n + t_{n-1}(u)-r)/\varepsilon\right) F(u,u_n) \right)\,.
\]
A change of variables from $u_n$ to $x := (u_n + t_{n-1}(u)-r)/\varepsilon$ yields
\[
 I_\varepsilon = \int_{\R_*^{n-1}}\!\rmd^{n-1} u\left(
 \int_{-\infty}^\infty \rmd x \, \cf{x>\left(\sum_{j=1}^{n-1}u_j-r\right)/\varepsilon}
 \varphi(x) F\Bigl(u,r-\sum_{j=1}^{n-1}u_j+\varepsilon x\Bigr) \right)\,.
\]
If $\sum_{j=1}^{n-1}u_j\ge r+R \varepsilon$, then either the indicator function or $\varphi(x)$ is zero, implying that the integrand is then zero.
On the other hand, 
the measure of the set $\{u\in \R_*^{n-1} : r\le \sum_{j=1}^{n-1}u_j< r+R \varepsilon\}$ goes to zero as $\varepsilon\to 0$.  Therefore, by the dominated convergence theorem and the continuity of $F$ we have
\[
 \lim_{\varepsilon\to 0} I_\varepsilon =  \int_{\R_*^{n-1}}\!\rmd^{n-1} u \,\cf{\sum_{j=1}^{n-1}u_j< r}
  \left(
 \int_{-\infty}^\infty \rmd x \, 
 \varphi(x) F\Bigl(u,r-\sum_{j=1}^{n-1}u_j\Bigr) \right) = \int_{\mathbb{R}_*^n}\tilde{\nu}_r(\rmd^n \fu) F(\fu )\,,
\]
where the last equality follows from the normalisation of $\varphi$ and the first part of the Lemma.  

The final item is a corollary of the previous one: if $f\in C_c(\R_*^n)$, we can apply the previous item to approximate its expectation.  However, 
$\rmd^{n}\fu \, \Phi_\varepsilon(\fu;r)$ is obviously permutation invariant, 
and we obtain the statement for such functions $f$. 
By the uniqueness statement in the 
Riesz--Markov--Kakutani theorem, this is sufficient to conclude that the Radon measure $\tilde{\nu}_r$ itself is permutation invariant.
\end{proof}

Essentially as a corollary, we obtain the following technical estimates used in the paper.
\begin{lemma}[$\delta$-integration]\label{lem:deltaprop}
The following formulas are valid.
\begin{itemize}
\item[i)] {\bf{Marginal moments}.}  Suppose $n\ge 2$ and $f:[0,\infty)\to \C$ is a continuous function.  Then 
\begin{equation}\label{eq:margmom}
\mean{f(u_1)}_{\nu(\rmd^n \fu)}=(n-1)\int_{0}^{1} \rmd u_1\,f(u_1)(1-u_1)^{n-2}  \,.
\end{equation}
\item[ii)]  {\bf{Correlations}.} 
Suppose $n\geq 3$ and $g:[0,\infty)^2\to \C$ is  a continuous function.  Then
\begin{equation}\label{eq:gcorrel}
\mean{g(u_1,u_2)}_{\nu(\rmd^n \fu)}=
(n-1)(n-2)
\int_{0}^{1} \rmd u_1\,
\int_{0}^{1-u_1} \rmd u_2\,g(u_1,u_2)(1-u_1-u_2)^{n-3}\,.
\end{equation}
\item[iii)] {\bf{Concentration}.} We have $t_n(\fu)=1$ almost surely
under the probability measure $\nu ( \rmd^{n} \fu )$.
\item[iv)] {\bf{``Fubini's theorem''}.} 
Suppose  $n\ge 2$ and $f:\R_*^{n}\to \C$ is a Borel measurable function.
If $f$ is also either non-negative, or Lebesgue integrable, then 
\begin{align}\label{eq:itemthreegoal}
\int_{\R_*^{n}} \rmd^{n} \fv \, f(\fv)
=  \int_{\R_*} \rmd r\,\left( \int_{\R_*^{n}} \tilde{\nu}_r(\rmd^n \fv) f(\fv) \right) 
=  \int_{\R_*} \rmd r\, \frac{r^{n-1}}{(n-1)!}\left(\int_{\R_*^{n}} \nu( \rmd^{n} \fu ) \, f(r \fu) \right)\,,
\end{align}
where the integrals are either all absolutely convergent or all infinite.
\end{itemize}
\end{lemma}

\begin{proof}
The third item is a direct corollary of item (1) of Lemma~\ref{lem:Borelprop}.  Therefore, the assumed continuity properties of $f$ and $g$ imply that
the observables $\fu \mapsto f(u_1)$ and $\fu \mapsto g(u_1,u_2)$ are bounded on the support of the measure $\nu$.  Hence, they are integrable, and we can apply item (3) of Lemma \ref{lem:Borelprop}.

Assuming $n\ge 3$, we first obtain
\[
 \mean{g(u_1,u_2)}_{\nu(\rmd^n \fu)}
 =  (n-1)!\int_{\R_*^{n-1}}\rmd^{n-1} u\,\cf{t_{n-1}(u)<1} g(u_1,u_2)\,.
\]
If $n=3$, this is equal to $\int_{0}^{1} \rmd u_1\,
\int_{0}^{1-u_1} \rmd u_2\,g(u_1,u_2)$.  If $n>3$, the indicator function is still zero if $u_1+u_2\ge 1$, and thus
we have 
\[
 \int_{\R_*^{n-1}}\rmd^{n-1} u\,\cf{t_{n-1}(u)<1} g(u_1,u_2) = 
\int_{0}^{1} \rmd u_1\,
\int_{0}^{1-u_1} \rmd u_2\,g(u_1,u_2)
 \int_{\R_*^{n-3}}\rmd^{n-3}v\,\cf{\sum_{i=1}^{n-3}v_i<1-u_1-u_2}\,.
\]
Hence, by using the earlier simplex volume result, we can conclude that for all $n\ge 3$,
\[
 \mean{g(u_1,u_2)}_{\nu(\rmd^n \fu)} = 
(n-1)(n-2)
\int_{0}^{1} \rmd u_1\,
\int_{0}^{1-u_1} \rmd u_2\,g(u_1,u_2)(1-u_1-u_2)^{n-3}\,.
\]
Similarly, if $n\ge 2$, we have 
\[
 \mean{f(u_1)}_{\nu(\rmd^n \fu)}
 = (n-1)! \int_{\R_*^{n-1}}\rmd^{n-1} u\,\cf{t_{n-1}(u)<1} f(u_1)\,.
\]
If $n=2$, this evaluates to 
$ \int_{0}^1\rmd u_1\, f(u_1) = (n-1)\int_{0}^{1} \rmd u_1\,f(u_1)(1-u_1)^{n-2}$.  If $n>2$, as above, we obtain
\[
  \mean{f(u_1)}_{\nu(\rmd^n \fu)}
  = (n-1)! \int_{0}^{1} \rmd u_1\, f(u_1)
  \int_{\R_*^{n-2}}\rmd^{n-2}v\,\cf{\sum_{i=1}^{n-2}v_i<1-u_1} \,.
\]
and evaluating the volume of the simplex  results in  \eqref{eq:margmom}.

Therefore, only the fourth item remains to be proven.
Let $n\ge 2$ and suppose $f:\R_*^{n}\to \C$ is Borel measurable.
Let us first consider the case where $f$ is also non-negative.
Applying item (3) of 
Lemma~\ref{lem:Borelprop}
we find that for any $r>0$,
\[
 \int_{\mathbb{R}_*^n}\tilde{\nu}_r(\rmd^n \fu) f(\fu )
  =
   \int_{\R_*^{n-1}}\rmd^{n-1} u\,\cf{t_{n-1}(u)<r} f(u,r-t_{n-1}(u))\,. 
\]
The integrand on the right hand side is non-negative and measurable under $\rmd r \times (\rmd^{n-1} u)$.
Thus by Fubini's theorem, we find that 
\[
 \int_{\R_*}\!\rmd r \left(\int_{\mathbb{R}_*^n}\tilde{\nu}_r(\rmd^n \fu) f(\fu )\right)
 =  \int_{\R_*^{n-1}}\rmd^{n-1} u\left(
  \int_{\R_*}\!\rmd r \cf{t_{n-1}(u)<r} f(u,r-t_{n-1}(u))\right)\,. 
\]
We change the integration variable $r$ to $u_n =r-t_{n-1}(u)$, yielding
\[
  \int_{\R_*^{n-1}}\rmd^{n-1} u\left(
  \int_{0}^\infty\!\rmd u_n f(u,u_n)\right)\,.
\]
By Fubini's theorem, this is equal to $\int_{\R_*^{n}}\rmd^{n} \fu f(\fu)$.  Hence we have proven
the first equality in \eqref{eq:itemthreegoal} for 
non-negative $f$ and the second equality follows then immediately from item (3) of Lemma~\ref{lem:Borelprop}.

Let us then assume that $f\in L^1(\R^n)$.
Then \eqref{eq:itemthreegoal} holds for $|f|$, as well as for positive and negative parts of $\mathrm{Re} f$ and $\mathrm{Im} f$.  Each of these functions is bounded by $|f|$, and thus all of the resulting integrals are finite.  Thus by taking the appropriate complex linear combination, we conclude that 
\eqref{eq:itemthreegoal} holds also for $f$, with absolutely convergent integrals.
This concludes the proof of the lemma.
\end{proof}
\section{\textbf{Tools}}\label{ap:basic}
The following section collects a few standard results, given here to facilitate following the argument in the main text.

\begin{lemma}\label{lem:independencia}
Let $U$ and $V$ be independent random variables with exponential distribution of parameter $\lambda>0$.
Set $X=(U-V)/(U+V)$ and $Y=U+V$. Then 
\begin{itemize}
\item[i)] the random variable $X$ has continuous uniform distribution in $[-1,1]$.
\item[ii)] the random variable $Y$ has Gamma distribution with parameter $2$ and $\lambda$.
\item[iii)] the random variables $X$ and $Y$ are  independent.
\end{itemize}
\end{lemma}
\begin{proof}
The proof of Item~i) and Item~ii) are straightforward. We leave the details to the interested reader.
In the sequel, we prove Item~iii).
For any continuous and bounded function $F:\mathbb{R}^2\to \mathbb{R}$ we have
\begin{equation}\label{eq:integralf}
\begin{split}
\mathbb{E}[F(U,V)]&=
\lambda^2\int_{0}^{\infty}\int_{0}^\infty \rmd u\rmd v\, e^{-\lambda u-\lambda v}F(u,v)\\
&=
\frac{\lambda^2}{2}\int_{\mathbb{R}^2} \rmd y\rmd z e^{-\lambda y} \cf{y+z>0}\cf{y-z>0}F\left(\frac{y+z}{2},\frac{y-z}{2}\right)\\
&=
\frac{\lambda^2}{2}\int_{0}^\infty \rmd y e^{-\lambda y}\int_{-y}^{y} \rmd z F\left(\frac{y+z}{2},\frac{y-z}{2}\right)\,,
\end{split}
\end{equation}
where in the second equality we apply the Change of Variable theorem for
$y=u+v$ and $z=u-v$. 
Equality \eqref{eq:integralf} with the help of
the change of variable $x=z/y$
yields
\begin{equation*}
\begin{split}
\mathbb{E}[F(U,V)]
&=
\int_{0}^\infty 
\rmd y\, \underbrace{\lambda^2 ye^{-\lambda y}}_{\textrm{density of  Y}} 
\int_{-1}^{1} \underbrace{\frac{1}{2}}_{\textrm{density of } X}\rmd x\,  F\left(y\frac{1+x}{2},y\frac{1-x}{2}\right)\,.
\end{split}
\end{equation*}
In particular, for continuous and bounded functions $f$ and $g$ the choice $F(u,v):=f((u-v)/(u+v))g(u+v)$, $u,v\in \mathbb{R}$ yields
\begin{equation*}
\begin{split}
\mathbb{E}[F(U,V)]=
\mathbb{E}[f(X)g(Y)]=
\int_{0}^\infty 
\rmd y\, {\lambda^2 ye^{-\lambda y}}g(y)
\int_{-1}^{1}\rmd x {\frac{1}{2}} f(x)\,.
\end{split}
\end{equation*}
Hence, the proof of Item iii) is finished.
\end{proof}

The following lemma provides distributional properties for the sum and difference of i.i.d.\ exponential random variables.
Since the proof is straightforward, we omit the details here.
\begin{lemma}\label{lem:propgamma}
Let $X$ and $Y$ be independent random variables  having exponential distribution of parameter $\theta>0$.
Then the following is valid.
\begin{enumerate}
\item $X-Y$ has Laplace distribution with location parameter zero and scale parameter $\frac{1}{\theta}$.
\item $X+Y$ has Gamma distribution with parameters $2$ and $\theta$.
\item $\mathbb{E}[X-Y]=0$ and $\mathbb{E}[X+Y]=\frac{2}{\theta}$.
\item $\textsf{Var}[X-Y]=\textsf{Var}[X+Y]=\frac{2}{\theta^2}$.
\item In general, if $Z$ is distributed according to a Gamma distribution of parameters $m\in \mathbb{N}$ and $\theta>0$, then for any $k\in \mathbb{N}$ we have
$\mathbb{E}[Z^k]=\theta^{-k}\frac{(m+k-1)!}{(m-1)!}$.
In addition, for the case $m=\theta$ we have 
\[
\mathbb{E}[|1-Z|^6]=\frac{15m^2 + 130m + 120}{m^5}\,.
\]
\end{enumerate}
\end{lemma}

In the next lemma, we recall the exact value for the mean and the second moment of a Poisson random variable.
\begin{lemma}[First and second moment of Poisson distribution]\label{lem:mopoisfull}
Let $X$ be a random variable with Poisson distribution of parameter $\lambda>0$. Then it follows that
\[
\mean{X}_{\textsf{Po}(\lambda)}=\lambda
\quad
\textrm{ and }\quad
\mean{X^2}_{\textsf{Po}(\lambda)}=\lambda^2+\lambda\,.
\]
\end{lemma}

\section*{\textbf{Declarations}}

\noindent
\textbf{Acknowledgements.}
The authors are grateful to the reviewer for the thorough
examination of the manuscript, which has lead to a significant improvement.
\bigskip

\noindent
\textbf{Funding.} The research has been supported by the Academy of Finland, via 
the Matter and Materials Profi4 university profiling action, 
an Academy project (project No. 339228)
and the Finnish centre of excellence in Randomness and STructures
(project No. 346306).
\bigskip

\noindent
\textbf{Availability of data and material.}
Data sharing not applicable to this article as no datasets were generated or analysed during the current study.
\bigskip

\noindent
\textbf{Conflict of interests.} 
The authors declare that they have no conflict of interest.
\bigskip

\noindent
\textbf{Authors' contributions.}
Both authors have contributed equally to the paper.

\bibliographystyle{amsplain}

\end{document}